\documentclass[leqno,12pt]{article}
\def\today{${\scriptscriptstyle\number\day-\number\month-\number\year}$}
\usepackage{graphicx}
\usepackage{fancyhdr}
\usepackage{amsmath}
\usepackage{amsthm}
\usepackage{amsfonts}
\newtheorem{theorem}{Theorem}[section]
\newtheorem*{thm}{Main Theorem}
\newtheorem{lemma}[theorem]{Lemma}
\newtheorem{corollary}[theorem]{Corollary}
\newtheorem{definition}[theorem]{Definition}

\newtheorem{remark}[theorem]{Remark}
\renewcommand{\theequation}{\thesection.\arabic{equation}}
\def\address#1{{\center{#1}}}
\date{}
\def\m@th{\mathsurround=0pt}
\def\eqal#1{\null\,\vcenter{\openup\jot\m@th
 \ialign{\strut\hfil$\displaystyle{##}$&&$\displaystyle{{}##}$\hfil
 \crcr#1\crcr}}\,}
\def\matrix#1{\null\,\vcenter{\normalbaselines\m@th
 \ialign{\hfil$##$\hfil&&\quad\hfil$##$\hfil\crcr
 \mathstrut\crcr\noalign{\kern-\baselineskip}
 #1\crcr\mathstrut\crcr\noalign{\kern-\baselineskip}}}\,}
\def\lc#1{\hbox to .89\hsize{\qquad$\displaystyle{#1}\hfill$}}

\def\Om{\Omega}
\def\ro{\varrho}
\def\nb{\nabla}
\def\eps{\varepsilon}
\def\Om{\Omega}
\def\si{\sigma}
\def\La{\Lambda}
\def\al{\alpha}

\def\pa{\partial}
\newcommand{\benn}{\begin{eqnarray*}}
\newcommand{\eenn}{\end{eqnarray*}}
\newcommand{\ben}{\begin{eqnarray}}
\newcommand{\een}{\end{eqnarray}}
\newcommand{\bal}{\begin{aligned}}
\newcommand{\eal}{\end{aligned}}

\def\D{{\mathbb D}}
\def\I{{\mathbb I}}
\def\N{{\mathbb N}}
\def\R{{\mathbb R}}
\def\T{{\mathbb T}}
\def\divv{{\rm div}\,}
\def\rot{{\rm rot}\,}

\def\dist{{\rm dist\,}}
\def\inf{{\rm inf\,}}
\def\supp{{\rm supp\,}}
\def\diam{{\rm diam\,}}
\def\const{{\rm const}}
\def\mize{\begin{itemize}\parskip-4pt}
\def\iintop{\mathop{\intop\!\!\!\intop}}

\def\bye{\end{document}}
\numberwithin{equation}{section}

\title{Global motions for nonhomogeneous Navier-Stokes equations with large flux}
\author{Joanna Renc\l awowicz$^{1)}$\quad Wojciech M. Zaj\c{a}czkowski$^{1),2)}$}

\begin{document}
\input amssym.def
\input amssym.tex
\maketitle
%%%%%%%%%%%%%%%%%%%%%%%%%%%%%%%%%%%%%%%%%%%%%%%%%%%%%%%%%%%%%%%%%%%%%%%%%
\thispagestyle{fancy}
%%%%%%%%%%%%%%%%%%%%%%%%%%%%%%%%%%%%%%%%%%%%%%%%%%%%%%%%%%%%%%%%%%%%%%%%%

\address{$^{1)}$\ Institute of Mathematics, Polish Academy of Sciences,\\
\'Sniadeckich 8, 00-656 Warsaw, Poland e-mail: jr@impan.pl\\
{\it corresponding autor}\\
$^{1)}$\ Institute of Mathematics, Polish Academy of Sciences,\\
\'Sniadeckich 8, 00-656 Warsaw, Poland e-mail: wz@impan.pl\\
$^{2)}$\ Institute of Mathematics and Cryptology, Cybernetics Faculty, \\
Military University of Technology,\\
S. Kaliskiego 2, 00-908 Warsaw, Poland\\}

\begin{abstract}
The nonhomogeneous Navier-Stokes equations are considered in a cylindrical domain in $\R^3$, parallel to the $x_3$-axis with large inflow and outflow on the top and the bottom. Moreover, on the lateral part of the cylinder the slip boundary conditions are assumed. The global existence of regular solutions is proved under assumptions that inflow and outflow are close to homogeneous and norms of derivative with respect to $x_3$ of the external force and initial velocity are sufficiently small. The key point of this paper is to verify that $x_3$-coordinate of velocity remains positive.


\end{abstract}

\section{Introduction}\label{s1}

We consider motions to the nonhomogeneous Navier-Stokes equations in cylindrical domains with large inflow and outflow. With "nonhomogeneous" we mean a density dependent system. Our aim is to prove the existence of global strong solutions without smallness restrictions on velocity and flux. In the proof we follow ideas and techniques from \cite{RZ1}.

 We introduce Cartesian coordinates $(x_1,x_2,x_3)$. Let $\Om\subset\R^3$ be a cylindrical domain parallel to the $x_3$-axis located inside. The boundary of $\Om$ denoted by $S$ is composed of two parts $S_1$ and $S_2$, where $S_1$ is parallel to the $x_3$-axis and $S_2$ is perpendicular to it. Moreover, $S_2=S_2(-a)\cup S_2(a)$, where $a\in\R$ is given and $S_2(a_i)$ meets the $x_3$-axis at $a_i$, $i=1,2$. It is assumed that $a_1=-a$, $a_2=a$.

Finally, $S_1$ and $S_2(a_i)$ meet along a curve $L(a_i)$, $i=1,2$.

Let $T>0$ be given. We consider in $\Om^T=\Om\times(0,T)$ the following initial-boundary value problem
\begin{equation}\eqal{
&\ro(v_t+v\cdot\nabla v)-\divv\T(v,p)=\ro f\quad &{\rm in}\ \ \Om^T,\cr
&\divv v=0\quad &{\rm in}\ \ \Om^T,\cr
&\ro_t+v\cdot\nabla\ro=0\quad &{\rm in}\ \ \Om^T,\cr
&v\cdot\bar n=0\quad &{\rm on}\ \ S_1^T=S_1\times(0,T),\cr
&\nu\bar n\cdot\D(v)\cdot\bar\tau_\alpha+\gamma v\cdot\bar\tau_\alpha=0,\ \ \alpha=1,2\quad &{\rm on}\ \ S_1^T,\cr
&v\cdot\bar n=-d_1,\ \ d_1>0\quad &{\rm on}\ \ S_2^T(-a),\cr
&v\cdot\bar n=d_2,\ \ d_2>0\quad &{\rm on}\ \ S_2^T(a),\cr
&\ro=\ro_1\quad &{\rm on}\ \ S_2^T(-a),\cr
&\bar n\cdot\D(v)\cdot\bar\tau_\alpha=0,\ \ \alpha=1,2\quad &{\rm on}\ \ S_2^T,\cr
&v|_{t=0}=v_0\quad &{\rm in}\ \ \Om,\cr
&\ro|_{t=0}=\ro_0\quad &{\rm in}\ \ \Om,\cr}
\label{1.1}
\end{equation}
where $\ro=\ro(x,t)\in\R^1$  denotes the density of the fluid, $v$ is the velocity with
$v(x,t)=(v_1(x,t),v_2(x,t),v_3(x,t))\in\R^3,$
$p=p(x,t)\in\R^1$ denotes the pressure, ,
$f=f(x,t)=(f_1(x,t),f_2(x,t),f_3(x,t))\in\R^3$ -- the external
force field, $x=(x_1, x_2, x_3)$ are the Cartesian coordinates.

By $\nu>0$ we denote the constant viscosity coefficient, $\gamma>0$ is the slip coefficient, $\bar n$ is the unit outward vector normal to $S$, $\bar\tau_\alpha$, $\alpha=1,2$, are vectors tangent to $S$.
Moreover, $I$ is the unit matrix, $\D(v)$ is the dilatation tensor of the form
$$
\D(v)=\{v_{i,x_j}+v_{j,x_i}\}_{i,j=1,2,3}
$$ and $\T(v,p)$ is the stress tensor
$$
\T(v,p)=\nu\D(v)-pI.
$$
Using Cartesian coordinates and assuming that with a given constant $c_0$, $\varphi_0(x_1,x_2)=c_0$ is a sufficiently smooth closed curve in the plane described with $x_3=\const\in(-a,a)$ located around the $x_3$-axis, we define
\begin{equation}\eqal{
&\Om=\{x\in\R^3\colon\varphi_0(x_1,x_2)\le c_0,-a<x_3<a\},\cr
&S_1=\{x\in\R^3\colon\varphi_0(x_1,x_2)=c_0,-a<x_3<a\},\cr
&S_2(-a)=\{x\in\R^3\colon\varphi_0(x_1,x_2)<c_0,x_3=-a\},\cr
&S_2(a)=\{x\in\R^3\colon\varphi_0(x_1,x_2)<c_0,x_3=a\}.\cr}
\label{1.2}
\end{equation}
To describe inflow and outflow we recall
\begin{equation}
d_1=-v\cdot\bar n|_{S_2(-a)},\quad d_2=v\cdot\bar n|_{S_2(a)},
\label{1.3}
\end{equation}
with $d_i>0$, $i=1,2$ and $\bar n$ is the unit outward vector normal to $S_2$.

Since incompressible motions are considered the following compatibility condition
\begin{equation}
\intop_{S_2(-a)}d_1dS_2=\intop_{S_2(a)}d_2dS_2
\label{1.4}
\end{equation}
holds.

Now, we formulate the main result: global existence of regular solutions to problem (\ref{1.1}). In order to prove that we need the existence of local regular solutions (theorem cited below, established in \cite{RZ3}) and global estimate for regular solutions. Then by standard argument local solutions can be extended in time as long as the global estimate holds.

\begin{theorem}\label{t1.1} (local existence, see \cite{RZ3})
Assume
\mize
\item[1.] Parameters $s,\sigma, r$ satisfy $s\in(0,1)$, ${3}/{s}<\sigma$, ${5}/{s}<r,$ $\sigma<r$.
\item[2.] Data functions are such that
\begin{itemize}
\item the initial density $\varrho_0\in W_r^1(\Omega)$ and $\varrho_1\in W_r^{1,1}(S_2^t(-a))$,
\item the initial velocity $v_0\in W_{\si}^{2+s-{2}/{\si}}(\Omega),$
\item the inflow $d_1\in W_{\si}^{2+s-{1}/{\si},1+{s}/{2}-{1}/{2\sigma}}(S_2^t(-a)),$
\item the outflow $d_2\in W_{\si}^{2+s-{1}/{\sigma},1+{s}/{2}-{1}/{2\sigma}}(S_2^t(a))$,
    \item    the external force $f\in W_{\si}^{s,{s}/{2}}(\Om^t)$.
\end{itemize}
\item[3.] There exist positive constants $\bar d_0$, $d_0$, $\bar d_0>d_0$, $d_\infty$ and $b_0$, $b_1$ $\bar b_0$, $\bar b_1$ such that $\bar d_0\ge v_3(0)\ge d_0$, $d_i\ge d_\infty$, $i=1,2$, $\bar b_0\ge\varrho(0)\ge b_0$, $\bar b_1\ge\varrho_1\ge b_1$.
\item[4.] The following quantities are finite
$$\eqal{
c_1 & = \|d_1\|_{L_{\infty}(S_2^t(-a))}(\|\ro_{1,x'}\|_{L_r(S_2^t(-a))}+\|\ro_{1,t}\|_{L_r(S_2^t(-a))})\cr & \quad +\|\ro_{0,x}\|_{L_r(\Om)}, \cr
c_2 & = \|f\|_{W_{\si}^{s,s/2}(\Om^t)}+ \sum_{i=1}^2 \|d_i\|_{W_{\si}^{2+s-{1}/{\si},1+{s}/{2}-{1}/{2\sigma}}(S_2^t(a_i))}\cr & \quad +\|\ro_1\|_{W_r^{1,1}(S_2^t(-a))} + \|v_0\|_{W_{\si}^{2+s-{2}/{\si}}(\Om)}.
}$$

\end{itemize}
Then there exists a local solution $(v,p,\ro)$ to the nonhomogeneous Navier-Stokes problem (\ref{1.1}) such that
$$
v\in W_{\si}^{2+s,1+s/2}(\Omega^t),\quad \nabla p\in W_{\si}^{s,s/2}(\Omega^t),\quad \varrho\in W_{r,\infty}^{1,1}(\Omega^t).
$$
Moreover, the density remains bounded
$$
\varrho_*\equiv{b_1b_0\over b_1+b_0}\le\varrho(x,t)\le\bar b_0+\bar b_1\equiv\varrho^*
$$
and the velocity and the pressure satisfy
$$
\|v\|_{W_{\si}^{2+s,1+s/2}(\Om^t)}+\|\nabla p\|_{W_{\si}^{s,s/2}(\Omega^t)}\le\phi({\rm data}),
$$
where data are described by assumptions 1,2 and 3, $\phi$ is an increasing positive function, $t\le T$ and $T$ is sufficiently small.

Finally, the $x_3$-coordinate of velocity is positive since there exists a positive constant $d_* = d_*(d_0,\bar{d}_0, \bar{c}_1, \bar{c}_2, d_{\infty}, \|f_3\|_{L_1(0,t;L_\infty(\Om))})$ such that
$$\eqal{
v_3&\ge d_*.\cr}
$$
\end{theorem}

\begin{thm}(global existence)
\label{mainthm}
Let $s\in (0,1)$, ${3}/{s}<\sigma$, ${5}/{s}<r,$ $\sigma<r$. Assume that:
\begin{itemize}
\item[1.] $f\in W_\sigma^{s,s/2}(\Om^t)$, $d\in W_\sigma^{2+s-1/\sigma,1+s/2-1/2\sigma}(S_2^t)$, $v_0\in W_\sigma^{2+s-2/\sigma}(\Om)$, $d=(d_1,d_2),$
\item[2.] $h(0)=v_{0,x_3}\in W_{5/3}^{4/5}(\Om)$, $g\in L_{5/3}(\Om^t)$, $h=v_{x_3}$, $q=p_{x_3},$ $g=f_{x_3},$
\item[3.] $\ro_0\in L_\infty(\Om)$, $\ro_1\in L_\infty(S_2^t(-a)),$ $1/\ro_0\in L_\infty(\Om)$, $1/\ro_1\in L_\infty(S_2^t(-a))$, $1/d_1\in L_\infty(S_2^t(-a))$, $1/d_2\in L_\infty(S_2^t(a))$.
\item[4.] The considered domain $\Om$ contains edges, $L_1$, $L_2$ such that $\bar S_1\cap\bar S_2\in\{L_1,L_2\}$. On the edges the following compatibility conditions hold
$$
\nu n_\alpha d_{i,x_\alpha}+\gamma d_i=0\quad {\rm on}\ \ L_i,\ \ i=1,2,
$$
when $\bar n=(n_1,n_2)$ is the normal vector to $S_1$.
\item[5.] Compatibility conditions:
$$
\intop_{S_2(-a)}d_1dS_2=\intop_{S_2(a)}d_2dS_2,
$$
$$
\pa_{x'}^\al v_0|_{S_2}=\pa_{x'}^\al d|_{t=0},\ \ \al = 0,1,2.
$$
\item[6.] The following quantities are finite
$$\eqal{
c_1 & = \|d_1\|_{L_{\infty}(S_2^t(-a))}(\|\ro_{1,x'}\|_{L_r(S_2^t(-a))}+\|\ro_{1,t}\|_{L_r(S_2^t(-a))})\cr & \quad +\|\ro_{0,x}\|_{L_r(\Om)}, \cr
c_2 & = \|f\|_{W_{\si}^{s,s/2}(\Om^t)}+ \sum_{i=1}^2 \|d_i\|_{W_{\si}^{2+s-{1}/{\si},1+{s}/{2}-{1}/{2\sigma}}(S_2^t(a_i))}\cr & \quad +\|\ro_1\|_{W_r^{1,1}(S_2^t(-a))} + \|v_0\|_{W_{\si}^{2+s-{2}/{\si}}(\Om)}.
}$$
\end{itemize}
We define finite parameters:
\begin{equation}\eqal{
 \La_1 & =\|d_{x'}\|_{L_2(0,t;W^1_3(S_2))}+\|d_{x'}\|_{L_\infty(0,t;L_2(S_2))}+ \|d_t\|_{L_2(0,t;H^1(S_2))}\cr & \quad +\|f_3\|_{L_2(0,t;L_{4/3}(S_2))} +\|f_{x_3}\|_{L_2(\Om^t)}, \ x'=(x_1,x_2),\cr
 \La_2 & \equiv \La_2(r) =\|\ro_{1,x'}\|_{L_r(S_2^t(-a))}+\|\ro_{1,t}\|_{L_r(S_2^t(-a))}+\|\ro_{0,x}\|_{L_r(\Om)},  \cr
 \La & =\La_1+\La_2+\|v_{0,x_3}\|_{L_2(\Om)}.\cr}
\label{Lambda}
\end{equation}

\noindent
{\bf Thesis 1}(upper and lower estimates for density): Let
$$\eqal{
&\ro_*= {1\over\left\|{1\over\ro_0}\right\|_{L_{\infty}(\Om)}+ \left\|{1\over\ro_1}\right\|_{L_{\infty}(S_2^t(-a))}},\cr
&\ro^*=\|\ro_1\|_{L_{\infty}(S_2^t(-a))}+\|\ro_0\|_{L_{\infty}(\Om)}.\cr}
$$
Then
\begin{equation}
\ro_*\le\ro(x,t)\le\ro^*.
\label{1.5}
\end{equation}
{\bf Thesis 2}(positivity of $v_3$): \\ Let $d_0$, $\bar d_0$, $d_\infty,$ $\bar{d}$ be positive constants with $\bar d_0\ge v_3(0)\ge d_0$, $ d_\infty \le d_i \le \bar{d},$  $i=1,2$.
Assume $f_3\in L_1(0,t;L_\infty(\Om))$.

\noindent Then there exists such a positive constant $d_*$ that
\begin{equation}
v_3\ge d_*
\label{1.6}
\end{equation}
where $d_*=d_*(d_0,\bar{d}_0, {c}_1, {c}_2, d_{\infty}, \|f_3\|_{L_1(0,t;L_\infty(\Om))}).$
%\label{d*}

\noindent{\bf Thesis 3}(energy estimate for $v$):

\noindent Let $\ro_1\in L_\infty(S_2^t)$, $d_1\in L_6(0,t;L_3(S_2)),$ $f\in L_2(0,t;L_{6/5}(\Om))$, $v_0\in L_2(\Om)$. Let $\tilde d_i$ be an extension of $d_i$, $i=1,2$, such that $\tilde d_i|_{S_2(a_i)}=d_i$ and assume $\tilde d_i\in L_\infty(0,t;W_{3,\infty}^1(\Om))\cap L_2(0,t;W_{3,\infty}^1(\Om))$, $\tilde d_{i,t}\in L_2(0,t;W_{6/5}^1(\Om))$.\\
Then
\begin{equation}\eqal{
& \|v\|^2_{L_{\infty}(0,t;L_2(\Om))} + \|\nb v\|^2_{L_2(\Om^t)} \equiv \|v\|^2_{V(\Om^t)}\le A_1^2, \cr
{\rm and} \quad  A_1^2 & = \phi(\|\ro_1\|_{L_{\infty}(S_2^t)},\|d_1\|_{L_6(0,t;L_3(S_2(-a)))}, \ro^*,\ro_*)\cdot\cr
&\quad\cdot\{\phi(\sup_t\|\tilde d\|_{W_{3,\infty}^1(\Om)},\ro^*)[\|\tilde d\|_{L_2(0,t;W_{3,\infty}^1(\Om))}^2\cr
&\quad+\|\tilde d_t\|_{L_2(0,t;W_{6/5}^1(\Om))}^2+\|f\|_{L_2(0,t;L_{6/5}(\Om))}^2]+\ro^* \|v_0\|_{L_2(\Om)}^2\},\cr}
\label{1.7}
\end{equation}
where $\tilde d$ replaces $\tilde d_1$, $\tilde d_2$ and $\phi$ is an increasing positive function.

\noindent
{\bf Thesis 4}(global estimate and existence):

\noindent Let
$$\eqal{
D&=\|f\|_{W_\sigma^{s,s/2}(\Om^t)}+ \|d\|_{W_\sigma^{2+s-1/\sigma,1+s/2-1/2\sigma}(S_2^t)}\cr
&\quad+\|v_0\|_{W_\sigma^{2+s-2/\sigma}(\Om)}+\|f_{x_3}\|_{L_{5/3}(\Om^t)}+ \|d_{x'}\|_{W_{5/3}^{7/5,7/10}(S_2^t)}\cr
&\quad+\|v_{0,x_3}\|_{W_{5/3}^{4/5}(\Om)}+ \|f_{2,x_1}-f_{1,x_2}\|_{L_2(0,t;L_{6/5}(\Om))}<\infty.\cr}
$$
Then for sufficiently small $\La$ solutions to problem (\ref{1.1}) exist and satisfy the following global estimate
\begin{equation}\eqal{
&\|v\|_{W_\sigma^{2+s,1+s/2}(\Om^t)}+\|\nabla p\|_{W_\sigma^{s,s/2}(\Om^t)}+\|v_{x_3}\|_{W_{5/3}^{2,1}(\Om^t)}\cr
&\quad+\|\nabla p_{x_3}\|_{L_{5/3}(\Om^t)} \le\phi\left(\ro^*,\frac{1}{\ro_*}, A_1,D\right),\cr
&\|\ro_x\|_{L_{\infty}(0,t;L_r(\Om))}+\|\ro_t\|_{L_{\infty}(0,t;L_r(\Om))} \le\phi\left(\ro^*,\frac{1}{\ro_*}, A_1,D\right)\cdot \La_2\cr}
\label{1.8}
\end{equation}
\end{thm}

\begin{remark}
We define Sobolev and other spaces in Section~\ref{s2}, in both Theorems appear a bit untypical norms, so we refer to Definitions~\ref{d2.1} and \ref{d2.3} to check on $W^{1,1}_r, W^{1,1}_{r,\infty}$ and $W^1_{p,\infty}.$

\end{remark}

\begin{proof}{\it (of Thesis 1)}
The lower and upper bounds for density described in (\ref{1.5}) are proved in Lemma \ref{l2.4}. We want to underline that this result can be established independently and the density in Thesis 1 is bounded by data only.\end{proof}

\begin{proof}{\it (of Thesis 2)}
Positivity of $v_3$ is analyzed and proved in Lemma~\ref{l7.1}.  \end{proof}

\begin{proof}{\it (of Thesis 3)}
The global energy estimate (\ref{1.7}) is proved in Section \ref{s3} and formulated in Theorem~\ref{3.4}. In order to establish this theorem we use that density has both lower and upper bounds, namely $\ro^*$ and $\ro_*$ (as in Thesis~1). \end{proof}

\begin{proof}{\it (of Thesis 4)}
The global existence to problem (\ref{1.1}) is proved step by step in time using local existence: by the standard argument we can extend in time local regular solutions as long as the global estimate holds, under assumption that $v_3\ge d_*>0$. As we can see, the energy estimate does not give enough regularity to accomplish such a goal, thus, we have to study derivatives of velocity and find appropriate bounds and inequalities in order to establish global estimate (\ref{1.8}). Therefore, we consider $h=v_{x_3}$ which is a solution to the problem (\ref{2.6}) and we derive Corollary~4.3 and inequality (\ref{4.19}):
$$ \eqal{
\|h\|_{V(\Om^t)}&\le c\phi(\ro_*,\ro^*,\mathcal{D}_1,V_1,A_1)[\La_1+ \|\ro_{x_3}\|_{L_\infty(0,t;L_3(\Om))}(\|f\|_{L_2(\Om^t)}\cr
&\quad+\|v_t\|_{L_2(\Om^t)}+\|v\|_{L_\infty(0,t;L_\infty(\Om))})]+\|h(0)\|_{L_2(\Om)},\cr}
$$
where  $\mathcal{D}_1 \equiv D_1=\|d_1\|_{L_2(0,t;L_3(S_2))},$ $V_1=\|v\|_{L_2(0,t;W^1_3(\Om))}$ are introduced in Le\-mma~4.2 and the norm $V$ in Definition~2.1.
The new difficulty arises: we have some norms of velocity $v$ to deal with but also $\ro_{x_3}$.
Fortunately, we can use the equation of continuity in the form
$$
\ro_t+v_3\ro_{x_3}+v_\alpha\ro_{x_\alpha}=0,
$$
where $\alpha=1,2$ and the summation convention with respect to repeated $\alpha$ is assumed, to conclude
$$\ro_{x_3}=-\frac{1}{v_3}(\ro_t+v_\alpha\ro_{x_\alpha})
\label{5.10}$$ as long as (\ref{1.6}) in Thesis~2 holds, i.e. $v_3\ge d_*>0.$
Then, in Section 5, Lemma~\ref{5.1}, the following density relation is found, with $x'=(x_1,x_2)$:
$$\eqal{
(\|\ro_{x'}(t)\|_{L_r(\Om)}^r+\|\ro_t(t)\|_{L_r(\Om)}^r)^{1/r}\le & \phi(v)\cdot\|d_1\|_{L_\infty(S_2^t)} \La_2. \cr}
\label{5.1}$$
where $\phi(v) = \phi(\|v\|_{L_{\infty}(\Om^t)},\|v_x\|_{{L_1(0,t;L_\infty(\Om))}},\|v_t\|_{L_1(0,t;L_\infty(\Om))})$ and  $\La_2$ is a small parameter defined in (\ref{Lambda}).

We apply some imbeddings for $v$  norms in $\phi(v).$ Let $t_0$ be a given positive number and $t \ge t_0$. Then 
\begin{equation}\eqal{
& \|v\|_{L_{\infty}(\Om^t)}\le c \sup_t\|v\|_{W^{2+s-2/\si}_\si(\Om)} \le c t^{1/\si'}\|v\|_{W^{2+s,1+s/2}_\si(\Om^t)}, \quad \frac{5}{\si} < 2+s, \cr
& \|v_x\|_{{L_1(0,t;L_\infty(\Om))}} \le t^{1/\si'} \left(\int_0^t \|v_x(t')\|^{\si}_{L_{\infty}(\Om)}dt'\right)^{1/\si} \cr & \le 
c t^{1/\si'} \left(\int_0^t \|v(t')\|_{W^{s}_{\si}(\Om)}dt'\right)^{1/\si}  \le c t^{1/\si'}\|v\|_{W^{2+s,1+s/2}_\si(\Om^t)}, \quad \frac{3}{\si} < 1+s, \cr
& \|v_t\|_{L_1(0,t;L_\infty(\Om))} \le t^{1/\si'} \left(\int_0^t \|v_t(t')\|^{\si}_{L_{\infty}(\Om)}dt'\right)^{1/\si} \cr & \le 
c t^{1/\si'} \left(\int_0^t \|v(t')\|_{W^s_{\si}(\Om)}dt'\right)^{1/\si}  \le c t^{1/\si'}\|v\|_{W^{2+s,1+s/2}_\si(\Om^t)}, \quad \frac{3}{\si} < s,\cr
& \|p_{x_3}\|_{{L_1(0,t;L_\infty(\Om))}} \le t^{1/\si'} \left( \int_0^t \|p_{x_3}(t')\|_{L_{\infty(\Om)}}^\si dt' \right)^{1/\si} \cr 
& \le c t^{1/\si'} \left(\int_0^t \|p_{x_3}(t')\|_{W_{\si}^s(\Om)}^\si dt' \right)^{1/\si} \le c t^{1/\si'}\|p_{x_3}\|_{W^{s,s/2}_\si(\Om^t)}, \quad \frac{3}{\si} < s,  \cr}
\end{equation}
where $1/\si+1/\si' = 1.$
Then, in view of above facts, we prove (see (\ref{5.15}))
$$\eqal{\|\ro_{x_3}\|_{L_\infty(0,t;L_r(\Om^t))} \le{1\over d_*}(1+\|v'\|_{L_\infty(\Om^t)})(\|\ro_{x'}(t)\|_{L_r(\Om)}^r+\|\ro_t(t)\|_{L_r(\Om)}^r)^{1/r} \cr  \!\le \!
\frac{1}{d_*}(1+\|v'\|_{L_\infty(\Om^t)})\phi(\|v\|_{W_\sigma^{2+s,1+\frac{s}{2}}(\Om^t)} ) \|d_1\|_{L_{\infty}(S_2^t(-a))} \La_2, \cr}$$
where $v'=(v_1,v_2)$.

In order to increase regularity of $v'$ we consider the third component of vorticity: $\chi=v_{2,x_1}-v_{1,x_2}$ and the following elliptic rot-div problem relating $\chi$ and $v'$
\begin{equation}\eqal{
&v_{2,x_1}- v_{1,x_2}=\chi\quad &{\rm in}\ \ \Om',\cr
&v_{1,x_1}+v_{2,x_2}=-h_3\quad &{\rm in}\ \ \Om',\cr
&v'\cdot\bar n'=0\quad &{\rm on}\ \ S'_1,\cr}
\label{1.15}
\end{equation}
where $x_3$ is treated as a parameter, $\Om'$ is the cross-section of $\Om$ with the plane perpendicular to the $x_3$-axis and passing through the point $x_3\in(-a,a)$ and $S'_1$ is the cross-section of $S_1$ with the same plane. Then $S'_1$ is the boundary of $\Om'$.
In Section~\ref{s8}, we analyze $\chi$ and we come up with the following estimate
\begin{equation}\eqal{
\|\chi\|_{V(\Om^t)}&\le\phi(\|v\|_{W_\sigma^{2+s,1+s/2}(\Om^t)})(1+\mathcal{D}_3)\La_2\cr
&\quad+\phi(\frac{1}{\ro_*},\ro^*,A_1,\bar{d})(\|v'\|_{W_2^{1,1/2}(\Om^t)}+ \|v'\|_{L_\infty(0,t;W_2^{5/6}(\Om))}\cr
&\quad+\|h\|_{L_\infty(0,t;L_3(\Om))}+\mathcal{D}_2),\cr}
\label{1.14}
\end{equation}
where  $\bar{d}$ is $L_\infty$ estimate for $d=(d_1,d_2)$, $\mathcal{D}_2=D_2, \mathcal{D}_3=D_3$ are constants defined in (\ref{9.15}), depending on data ($f$, $\chi_0$- initial data for $\chi$ and r.h.s. in $\chi$ problem, i.e. $F=(\rot f)_3$ and $A_1$ - the estimate for  $\|v\|_{V(\Om^t)}$ defined in (\ref{1.7})) with norm $V$ given by Definition~\ref{2.1}.
%$$\eqal{
%&D_2=A_1+|F|_{6/5,2,\Om^t}+|\chi(0)|_{3,\Om},\cr
%&D_3=|f'|_{6r/(5r-6),2,\Om^t}+|\chi(0)|_{3,\Om}.\cr}
%$$

Then, in Section~\ref{s10}, we make use of rot-div problem (\ref{1.15}) to deal with $v'$ terms and we conclude the estimate (\ref{10.1}) for $v'$ in higher norms (see Definition~\ref{2.1}):
\begin{equation}\eqal{
\|v'\|_{V^1(\Om^t)}&\le\bar\phi\cdot[\|v'\|_{L_2(\Om;H^{1/2}(0,t))}+\|h\|_{L_\infty(0,t;L_3(\Om))}\cr
&\quad+\mathcal{D}_2+(1+\mathcal{D}_3)\phi(\|v\|_{W_\sigma^{2+s,1+s/2}(\Om^t)})\La].\cr}
\label{v'-v1}
\end{equation}
where $\bar\phi \equiv \phi\left(\frac{1}{\ro_*},\ro^*,A_1,\mathcal{D}_1,\bar{d},\|v_0\|_{L_{\infty}(\Om)},\|f\|_{L_2(\Om^t)}\right).$

We emphasize that in the r.h.s. of (\ref{v'-v1}) only terms $\|v'\|_{L_2(\Om;H^{1/2}(0,t))}$ and $\|h\|_{L_\infty(0,t;L_3(\Om))}$ are possibly multiplied by large parameters, whereas other norms of $v$ and $p$ are multiplied by a small parameter $\La.$
For $h$ term, we postpone the discussion (briefly, we estimate this through $v$ norms at the end). For $v'$ term, we could incorporate the interpolation
$$
\|v'\|_{L_2(\Om,H^{1/2}(0,t))}\le\eps\|v\|_{W_{5/3}^{2,1}(\Om^t)}+ c(1/\varepsilon)A_1
$$
and deal with $W_{5/3}^{2,1}$ norm instead.

In order to derive estimates for $\|v\|_{W_{5/3}^{2,1}(\Om^t)}$ and next, to increase regularity, we are going to apply the regularity theory for Stokes problem of the form:
\begin{equation}\eqal{
&\ro v_t-\divv\T(v,p)=-\ro v'\cdot\nabla v-\ro v_3h+\ro f\quad &{\rm in}\ \ \Om^T,\cr
&\divv v=0\quad &{\rm in}\ \ \Om^T,\cr
&v\cdot\bar n=0,\ \ \nu\bar n\cdot\D(v)\cdot\bar\tau_\alpha+\gamma v\cdot\bar\tau_\alpha=0,\ \ \alpha=1,2\quad &{\rm on}\ \ S_1^T,\cr
&v\cdot\bar n=d,\ \ \bar n\cdot\D(v)\cdot\bar\tau_\alpha=0,\ \ \alpha=1,2\quad &{\rm on}\ \ S_2^T,\cr
&v|_{t=0}=v_0\quad &{\rm in}\ \ \Om,\cr}
\label{stokes}
\end{equation}
where $\ro$ is treated as given.

However, our coefficients depend on density $\ro$ and are not constant. Thus, for variable $\ro$ we need the H\"older continuity. Appropriate estimates for solutions to the Stokes system can be found in Appendix A. To show the H\"older continuity of $\ro$ we need that $\ro$ belongs to $W_{r,\infty}^{1,1}(\Om^t),$  $r>3$ and use appropriate imbedding, namely,
$$ \|\ro\|_{\dot C^\al(\Om^t)} \le \|\ro\|_{C^\al(\Om^t)}\le\|\ro\|_{W_{r,\infty}^{1,1}(\Om^t)}$$ which holds for
$ \frac{3}{r}+\al<1.$ The estimate for $\ro$ in the norm of Sobolev space $
\|\ro\|_{W_{r,\infty}^{1,1}(\Om^t)}$
is shown in Section \ref{s5}.

%To prove global estimate (\ref{1.8}) we need results on estimates and existence of solutions to the following linear problem
%\begin{equation}\eqal{
%&\ro w_t-\nu\Delta w+\nabla q=f \quad {\rm in} \ \Om^T, \cr
%&\divv w=0 \quad {\rm in} \  \Om^T \cr}
%\label{1.9}
%\end{equation}
%with boundary conditions $(\ref{1.1})_{4,5,6,7,9}$.

 For $\ro=\const$ the problem was considered in \cite{RZ4}. In this paper we find estimates and prove existence of solutions in Besov spaces. It is clear that the existence of solutions in anisotropic Sobolev spaces can also be proved.

For the norm of velocity $v$ in  ${W_{5/3}^{2,1}(\Om^t)}$ we show in Lemma~\ref{l10.2} that (see (\ref{10.1}))
$$\eqal {\|v\|_{W_{5/3}^{2,1}(\Om^t)}+ \|\nb p\|_{L_{5/3}(\Om^t)} & \le c(H + \mathcal{D}_4)+ \cr & +\bar{\phi}\cdot\phi(\|v\|_{W_\si^{2+s,1+s/2}(\Om^t)},\|\nb p\|_{W_\si^{s,s/2}(\Om^t)})\La, \cr}
$$
with
$$
H=\|h\|_{L_\infty(0,t;L_3(\Om))}+\|h\|_{L_{10/3}(\Om^t)}+ \|h(0)\|_{L_2(\Om)}.
$$
and $\mathcal{D}_4 = D_2 + D_4$ defined with formula (\ref{10.9}) is the constant dependent on norms of $f, d, v_0.$

We continue increasing of regularity in Lemmas \ref{l10.3}--\ref{l10.5} and derive similar inequalities for the following norms:
$$ \eqal{ &\|v\|_{W_2^{2,1}(\Om^t)},\|\nb p\|_{L_2(\Om^t)} \quad {\rm in} \ (\ref{10.22}) \quad {\rm with}\  D_2,D_5, \cr
&\|v\|_{W_{5/2}^{2,1}(\Om^t)},\|\nb p\|_{L_{5/2}(\Om^t)} \quad {\rm in} \ (\ref{10.26}) \quad {\rm with}\  D_7,  \cr
& \|v\|_{W_{5}^{2,1}(\Om^t)},\|\nb p\|_{L_5(\Om^t)} \quad {\rm in} \ (\ref{10.32}) \quad {\rm with}\ D_8, \cr
}$$
where $D_5,D_8$ are constant dependent on norms of $f, d, v_0,$ and $D_7$ additionally depends on norms of $\chi_0$, $F=(\rot f)_3$ and $A_1.$
Finally, (\ref{10.34}) in Lemma~\ref{l10.5} yields
\begin{equation} \eqal{
&\|v\|_{W_\si^{2+s,1+s/2}(\Om^t)}+\|\nb p\|_{W_\sigma^{s,s/2}(\Om^t)}  \le c(H,\mathcal{D}_5)+ \cr & \quad + \bar{\phi}\phi(\|v\|_{W_\si^{2+s,1+s/2}(\Om^t)},\|p\|_{W_\si^{s,s/2}(\Om^t)})\cdot
\La\cr}
\label{1.28}
\end{equation}
with $\mathcal{D}_5= D_9$ dependent on norms of $f, d, v_0.$
Since $$ H\le c \|h\|_{W^{2,1}_{5/3}(\Om^t)} $$
we get as well the relation, with $h=v_{x_3}, q=p_{x_3}$ and $\mathcal{D}_6(g,d,h_0)=D_{10},$  $g=f_{x_3},$ given in (\ref{D10}) in Lemma~\ref{r11.2}:
\begin{equation} \eqal{
\|h\|_{W^{2,1}_{5/3}(\Om^t)} + \|\nb q\|_{L_{5/3}(\Om^t)}  \le c \mathcal{D}_6 + \cr \bar{\phi}\phi(\|v\|_{W_\si^{2+s,1+s/2}(\Om^t)},\|\nb p\|_{W_\si^{s,s/2}(\Om^t)},\|h\|_{W^{2,1}_{5/3}(\Om^t)},\|\nb q\|_{L_{5/3}(\Om^t)})\cdot \La. \cr}
\label{1.30}
\end{equation}
Let
$$ X = \|v\|_{W_\si^{2+s,1+s/2}(\Om^t)}+\|\nb p\|_{W_\si^{s,s/2}(\Om^t)}+\|h\|_{W^{2,1}_{5/3}(\Om^t)}+\|\nb q\|_{L_{5/3}(\Om^t)},$$ then (\ref{1.28}) and (\ref{1.30}) imply (see Remark~\ref{r11.3}):
$$ \begin{eqal}{ X & \le \phi(X)\La + \phi(\frac{1}{\ro_*},\ro^*,\mathcal{D}_{7},\bar{\phi}), \cr \mathcal{D}_{7} & =\mathcal{D}_5+\mathcal{D}_{6}+\|f_{2,x_1}-f_{1,x_2}\|_{L_2(0,t;L_{6/5}(\Om))}+A_1 \cr}
 \end{eqal} $$
and for $\La$ sufficiently small in Lemma~\ref{11.3} we conclude the estimate
$$ X \le \phi\left(\frac{1}{\ro_*},\ro^*,\mathcal{D}_{7},\bar{\phi}\right). $$
\end{proof}

We refer here to some results related to the global or long time existence of regular solutions to nonhomogeneous Navier-Stokes equations. In \cite{DZ}, the global existence and uniqueness of solutions to nonhomogeneous Navier-Stokes system in the half-space $\R_+^n$, $n\ge 2$, has been established, with the initial density bounded and close enough to a positive constant, the initial velocity belonging to some critical Besov space and some smallness of data. Namely, $L_\infty$ norm of the inhomogeneity and the critical norm  to the horizontal components of the initial velocity has been assumed very  small compared to the exponential of the norm to the vertical component of the initial velocity. In the paper \cite{DM} the boundary value problem for the incompressible inhomogeneous Navier-Stokes equations in the half-space in the case of small data with critical regularity is analyzed. It is shown, in dimension $n\ge 3$, that if the initial density is close to a positive constant in $L_\infty\cap\dot W_n^1(\R_+^n)$ and the initial velocity is small with respect to the viscosity in the homogeneous Besov space $\dot B_{n,1}^0(\R_+^n)$ then the equations have a unique global solution. In \cite{Z3}, the author considered the equations in a bounded cylinder under boundary slip conditions. Assuming that the derivatives of density, velocity, external force with respect to the third co-ordinate are sufficiently small in some norms, the existence of large time regular solutions in Sobolev spaces has been proved, namely $v\in H^{s+2,s/2+1}(\Om^t)$, $s\in(1/2,1)$. In \cite{LS}, Ladyzhenskaya and Solonnikov have obtained existence results to nonhomogeneous Navier-Stokes equations for $v\in W_q^{2,1}$, $\nabla p\in L_q$, $q>n$ and $\ro\in C^1$, for small times with arbitrary $v_0$ and $f$ and for any given time interval with sufficiently small $v_0$ and $f$. The problem was analyzed in a bounded domain in $\R^n$ with boundary $S\in C^2$ and $v|_{S^T}=0$. We mention here some papers concerning nonhomogeneous magnetohydrodynamics equations, concerning Navier-Stokes equations and magnetohydrodynamics, see \cite{Z}, \cite{CLX}, \cite{BWY}.

\section{Notation and auxiliary results}\label{s2}

First we introduce the simplified notation.

\begin{itemize}
\item[$\bullet$] By $\phi$ and $\phi_k$, $k\in\N$, we denote increasing positive functions depending on quantities and norms of data which are not assumed to be small and $\phi$ contains some constants $c.$
\item[$\bullet$] $D_k$, $k\in\N$, depends linearly on norms of data which are not assumed to be small.
\item[$\bullet$] $\La_1$, $\La_2$, $\La$ are small parameters that depend on norms of data assumed to be small.
\item[$\bullet$] By dot $\cdot$ we mean a multiplication of functions.
\item[$\bullet$] Exponent $a>0$ can change its value from formula to formula.
\item[$\bullet$] $\Phi'_k$, $\phi'_k$ appear in proofs only and play an auxiliary, temporary role.
\end{itemize}

\begin{definition}\label{d2.1}
Let $Q$ be either $\Om$ or $S\subset\partial\Om$ or $\R^3$. For Lebesque and Sobolev spaces we set the notation
$$\eqal{
&\|u\|_{L_p(Q)}=|u|_{p,Q},\quad \|u\|_{L_p(Q^t)}=|u|_{p,Q^t},\cr
&\|u\|_{L_q(0,t;L_p(Q))}=|u|_{p,q,Q^t},\cr}
$$
where $p,q\in[1,\infty]$, $Q^t=Q\times(0,t)$.

Let $W_p^s(\Om)$, $s\in\N$, $\Om\subset\R^3$ be the Sobolev space with the finite norm
$$
\|u\|_{W_p^s(\Om)}=\bigg(\sum_{|\alpha|\le s}\intop_\Om|D_x^\alpha u|^pdx\bigg)^{1/p},
$$
where $D_x^\alpha= \partial_{x_1}^{\alpha_1}\partial_{x_2}^{\alpha_2}\partial_{x_3}^{\alpha_3}$, $|\alpha|=\alpha_1+\alpha_2+\alpha_3$, $\alpha_i\in\N_0=\N\cup\{0\}$, $i=1,2,3$ and $p\in[1,\infty]$.

Let $H^s(\Om)=W_2^s(\Om)$. Then we denote
$$
\|u\|_{H^s(\Om)}=\|u\|_{s,\Om},\quad \|u\|_{W_p^s(\Om)}=\|u\|_{s,p,\Om}.
$$
To define space $W_p^s(S)$ we need an appropriate partition of unity.

We have the compatibility between spaces
$$
H^0(\Om)=L_2(\Om),\quad W_p^0(\Om)=L_p(\Om).
$$
We also apply the notation
$$\eqal{
&\|u\|_{L_q(0,t;W_p^k(Q))}=\|u\|_{k,p,q,Q^t},\cr
&\|u\|_{L_q(0,t;H^k(Q))}=\|u\|_{k,q,Q^t}.\cr}
$$
and (we use these norms for density)
$$
\|u\|_{W_r^{1,1}(Q^t)}=\|u\|_{L_r(Q^t)}+\|u_{x}\|_{L_r(Q^t)}+\|u_{t}\|_{L_r(Q^t)},
$$
where $r\in[1,\infty]$ and $Q$ is equal either $\Omega$ or $S$. Moreover,
$$
\|u\|_{W_{r,s}^{1,1}(Q^t)}=\|u\|_{L_s(0,t;L_r(Q))}+\|u_{x}\|_{L_s(0,t;L_r(Q))}+ \|u_{t}\|_{L_s(0,t;L_r(Q))},
$$
where $r,s\in[1,\infty]$. \\
Finally, we introduce spaces appropriate for energy type estimates for solutions to the Navier-Stokes equations
$$
\|u\|_{V^k(\Om^t)}=\|u\|_{L_\infty(0,t;H^k(\Om))}+\|\nabla u\|_{L_2(0,t;H^k(\Om))},
$$
where $V^0(\Om^t)=V(\Om^t)$ and
$$
\|u\|_{V(\Om^t)}=\|u\|_{L_\infty(0,t;L_2(\Om))}+\|\nabla u\|_{L_2(\Om^t)}.
$$
\end{definition}

\begin{definition}\label{d2.2}
Anisotropic Sobolev-Slobodetskii spaces $W_p^{k,l}(\Om^T)$, $k,l\in\R_+$, $p\in[1,\infty]$ are defined in the following way,
$$
\|u\|_{W_p^{k,l}(\Om^T)}=\|u\|_{W_p^{k,0}(\Om^T)}+\|u\|_{W_p^{0,l}(\Om^T)}
$$
and
$$\eqal{
&\|u\|_{W_p^{k,0}(\Om^T)}=\bigg(\intop_0^T\|u(t)\|_{W_p^k(\Om)}dt\bigg)^{1/p},\cr
&\|u\|_{W_p^{0,l}(\Om^T)}=\bigg(\intop_\Om\|u(x)\|_{W_p^l(0,T)}dx\bigg)^{1/p}.\cr}
$$
Next, we have
$$\eqal{
\|u\|_{W_p^k(\Om)}&=\sum_{|\alpha|\le[k]}\|D_x^\alpha u\|_{L_p(\Om)}\cr
&\quad+\bigg(\sum_{|\alpha|=[k]}\intop_\Om\intop_\Om{|D_x^\alpha u(x,t)-D_{x'}^\alpha u(x',t)|^p\over|x-x'|^{s+p(k-[k])}}dxdx'\bigg)^{1/p}\cr}
$$
and
$$\eqal{
\|u\|_{W_p^l(0,T)}&=\sum_{i\le[l]}\|\partial_t^iu\|_{L_p(0,T)}\cr
&\quad+\bigg(\intop_0^T\intop_0^T {|\partial_t^{[l]}u(x,t)-\partial_{t'}^{[l]}u(x,t')|^p\over|t-t'|^{1+p(l-[l])}} \bigg)^{1/p},\cr}
$$
where $s=\dim\Om$, $[m]$ is the integer part of $m$, $D_x^\alpha=\partial_{x_1}^{\alpha_1}\dots\partial_{x_s}^{\alpha_s}$, $\alpha=(\alpha_1,\dots,\alpha_s)$ is a multiindex.

Finally, we introduce the following homogeneous spaces
$$
\|u\|_{\dot W_{r,s}^{1,1}(\Om^t)}=|u_{x}|_{r,s,\Om^t}+|u_{t}|_{r,s,\Om^t}
$$
By $C^\alpha(\Omega^T)$, $\alpha\in(0,1)$ we denote the H\"older space with the norm
$$\eqal{
\|u\|_{C^{\alpha}(\Omega^T)}&=\sup_{x,x',t} {|u(x,t)-u(x',t)|\over|x-x'|^\alpha}\cr
&\quad+\sup_{x,t,t'}{|u(x,t)-u(x,t')|\over|t-t'|^\alpha}+|u|_{\infty,\Omega^T}\cr}
$$
and anisotropic homogeneous H\"{o}lder spaces
$$
\|u\|_{\dot C^{\alpha}(\Om^t)}=\sup_{t,x',x''} {|u(x',t)-u(x'',t)|\over|x'-x''|^\alpha}+\sup_{x,t',t''} {|u(x,t')-u(x,t'')|\over|t'-t''|^\alpha}.
$$

Next, we introduce weighted spaces.

Let $\ro=\min_{i\in\{1,2\}}\dist\{x,S_2(a_i)\}$. Then weighted spaces $L_{p,\mu}(\Om)$ and $V_{p,\mu}^2(\Om)$ are defined by
$$
\|u\|_{L_{p,\mu}(\Om)}=\bigg(\intop_\Om|u|^p\ro^{p\mu}dx\bigg)^{1/p},\quad p\in(1,\infty),\ \ \mu\in\R
$$
and
$$
\|u\|_{V_{p,\mu}^2(\Om)}=\bigg(\intop_\Om(|\nabla^2u|^p+|\nabla u|^p\ro^{-p}+|u|^p\ro^{-2p})\ro^{p\mu}dx\bigg)^{1/p}.
$$
\end{definition}

\begin{definition}\label{d2.3}
Anisotropic Lebesque and Sobolev spaces $L_{p,\infty}(\Om)$ and\break $W_{p,\infty}^1(\Om)$, $p\in(1,\infty]$, are spaces with the following finite norms
$$
\|u\|_{L_{p,\infty}(\Om)}=\bigg(\sup_{x_3\in(-a,a)}\intop_{S_2(x_3)} |u(x',x_3)|^pdx'\bigg)^{1/p},
$$
where $S_2(x_3)$ is a cross-section of $\Om$ with the plane $x_3=\const\in(-a,a)$ and $x'=(x_1,x_2)$.

Moreover,
$$
\|u\|_{W_{p,\infty}^1(\Om)}=\sup_{x_3\in(-a,a)}\bigg(\intop_{S_2(x_3)} (|u(x',x_3)|^p+|\nabla u(x',x_3)|^p)dx'\bigg)^{1/p}.
$$
\end{definition}

We consider the problem
\begin{equation}\eqal{
&\ro_t+v\cdot\nabla\ro=0\ \quad {\rm in}\  \Om^T, \cr
&\ro|_{t=0}=\ro_0,\ \cr
&\ro|_{S_2(-a)}=\ro_1 \quad {\rm on}\ S_2, \cr
&\divv v=0,\quad {\rm in}\  \Om^T, \cr
&v\cdot\bar n|=-d_1,\ d_1>0 \quad {\rm on}\  S_2(-a).\cr}
\label{2.1}
\end{equation}

\begin{lemma}\label{l2.4}
Assume that $\ro_0\in L_\infty(\Om)$, $\ro_1\in L_\infty(S_2^t(-a))$. Assume that $\ro$ is a solution to (\ref{2.1}). Then
\begin{equation}
|\ro(t)|_{\infty,\Om}\le|\ro_1|_{\infty,S_2^t(-a)}+ |\ro_0|_{\infty,\Om}\equiv\ro^*.
\label{2.2}
\end{equation}
Assume that $1/\ro_0\in L_\infty(\Om)$, $1/\ro_1\in L_\infty(S_2^t(-a))$. Then
\begin{equation}
\ro_*\equiv{1\over\left|{1\over\ro_0}\right|_{\infty,\Om}+ \left|{1\over\ro_1}\right|_{\infty,S_2^t(-a)}}= {\inf\ro_1\cdot\inf\ro_0\over\inf\ro_1+\inf\ro_0}\le\ro.
\label{2.3}
\end{equation}
\end{lemma}

\begin{proof}
Multiply $(\ref{2.1})_1$ by $\ro|\ro|^{p-2}$, $p\in\R_+$ and integrate over $\Om$. Then we obtain
$$
{d\over dt}|\ro|_{p,\Om}^p+\intop_\Om v\cdot\nabla|\ro|^pdx=0.
$$
In view of properties of $v$ and boundary conditions for $\ro$, we have
$$
{d\over dt}|\ro|_{p,\Om}^p\le\intop_{S_2(-a)}d_1|\ro_1|^pdS_2.
$$
Integrating with respect to time yields
$$
|\ro(t)|_{p,\Om}^p\le\intop_{S_2^t(-a)}d_1|\ro_1|^pdS_2dt'+ |\ro_0|_{p,\Om}^p.
$$
Hence
$$\eqal{
|\ro(t)|_{p,\Om}&\le\bigg(\intop_{S_2^t(-a)}d_1|\ro_1|^pdS_2dt'\bigg)^{1/p} +|\ro_0|_{p,\Om}\cr
&\le|d_1|_{\infty,S_2^t(-a)}^{1/p}|\ro_1|_{p,S_2^t(-a)}+ |\ro_0|_{p,\Om}.\cr}
$$
Passing with $p\to\infty$ implies (\ref{2.2}).

Multiply $(\ref{2.1})_1$ by $\ro|\ro|^{-p-2}$, $p\in\R_+$, and integrate over $\Om$. Then we have
$$
{d\over dt}\bigg|{1\over\ro}\bigg|_{p,\Om}^p\le\intop_{S_2(-a)}d_1 \bigg|{1\over\ro_1}\bigg|^pdS_2.
$$
Integrating with respect to time gives
$$
\bigg|{1\over\ro}\bigg|_{p,\Om}^p\le\intop_{S_2^t(-a)}d_1 \bigg|{1\over\ro_1}\bigg|^pdS_2dt'+\bigg|{1\over\ro_0}\bigg|_{p,\Om}^p.
$$
Taking the above inequality to the power $1/p$ implies
$$\eqal{
\bigg|{1\over\ro}\bigg|_{p,\Om}&\le\bigg(\intop_{S_2^t(-a)}d_1 \bigg|{1\over\ro_1}\bigg|^pdS_2dt'\bigg)^{1/p}+ \bigg|{1\over\ro_0}\bigg|_{p,\Om}\cr
&\le|d_1|_{\infty,S_2^t(-a)}^{1/p}\bigg|{1\over\ro_1}\bigg|_{p,S_2^t(-a)}+ \bigg|{1\over\ro_0}\bigg|_{p,\Om}.\cr}
$$
Passing with $p\to\infty$ yields
$$
\bigg|{1\over\ro}\bigg|_{\infty,\Om}\le \bigg|{1\over\ro_1}\bigg|_{\infty,S_2^t(-a)}+ \bigg|{1\over\ro_0}\bigg|_{\infty,\Om}.
$$
The above inequality implies (\ref{2.3}) and concludes the proof.
\end{proof}

\begin{lemma}[The Korn inequality (see \cite{SS})]\label{l2.5}
Assume that
$$
E_\Om(w)=|\D(w)|_{2,\Om}^2<\infty,\quad \divv w=0,\ \ w\cdot\bar n|_S=0.
$$
Assume that $\Om$ is not axially symmetric. Then there exists a constant $c$ independent of $w$ such that
\begin{equation}
\|w\|_{H^1(\Om)}^2\le cE_\Om(w).
\label{2.4}
\end{equation}
\end{lemma}

To prove the existence of solutions to problem (\ref{1.1}) with large data we follow the ideas developed in \cite{RZ1}, \cite{RZ2}, \cite{Z1}, \cite{Z2}. To present them we introduce the quantities
\begin{equation}
h=v_{x_3},\quad q=p_{x_3},\quad g=f_{x_3},\quad \chi=(\rot v)_3,\quad F=(\rot f)_3.
\label{2.5}
\end{equation}

\begin{lemma}\label{l2.6}
Let $(\ro,v,p)$ be a solution to problem (\ref{1.1}). Then $(\ro,h,q)$ is a solution to the problem
\begin{equation}\eqal{
&\ro h_t-\divv\T(h,q)\cr
&=-\ro(v\cdot\nabla h+h\cdot\nabla v-g)-\ro_{x_3}(v_t+v\cdot\nabla v-f)\quad &{\rm in}\ \ \Om^T,\cr
&\divv h=0\quad &{\rm in}\ \ \Om^T,\cr
&h\cdot\bar n=0,\ \ \nu \bar n\cdot\T(h,q)\cdot\bar\tau_\alpha+\gamma h\cdot\bar\tau_\alpha=0,\ \ \alpha=1,2\quad &{\rm on}\ \ S_1^T,\cr
&h_i=-d_{1,x_i},\ \ i=1,2,\ \ h_{3,x_3}=\Delta'd_1\quad &{\rm on}\ \ S_2^T(-a),\cr
&h_i=-d_{2,x_i},\ \ i=1,2,\ \ h_{3,x_3}=\Delta'd_2\quad &{\rm on}\ \ S_2^T(a),\cr
&h|_{t=0}=h(0)\quad &{\rm in}\ \ \Om,\cr}
\label{2.6}
\end{equation}
where $\Delta'=\partial_{x_1}^2+\partial_{x_2}^2$ and $\ro$ is a solution to (\ref{2.1}).
\end{lemma}

\begin{proof}
$(\ref{2.6})_{1,2,3,6}$ directly follow from $(\ref{1.1})_{1,2,3,4,8}$ by differentiation with respect to $x_3$. To show $(\ref{2.6})_{4,5}$ we recall that
$$\eqal{
&v_3|_{S_2(-a)}=d_1,\quad &(v_{i,x_3}+v_{3,x_i})|_{S_2(-a)}=0,\ \ i=1,2,\cr
&v_3|_{S_2(a)}=d_2,\quad &(v_{i,x_3}+v_{3,x_i})|_{S_2(a)}=0,\ \ i=1,2.\cr}
$$
Hence $v_{i,x_3}|_{S_2(a_j)}=-d_{j,x_i}$, $i,j=1,2$. Then
$$
v_{3,x_3x_3}|_{S_2(a_j)}=-(v_{1,x_1x_3}+v_{2,x_2x_3})|_{S_2(a_j)}=d_{j,x_1x_1}+ d_{j,x_2x_2}=\Delta'd_j,\ \ j=1,2.
$$
Therefore, $(\ref{2.6})_{4,5}$ holds. This ends the proof.
\end{proof}

To formulate a problem for $\chi$ we introduce
\begin{equation}\eqal{
&\bar n|_{S_1}={\nabla\varphi_0\over|\nabla\varphi_0|},\ \ \bar\tau_1|_{S_1}={\nabla^\perp\varphi_0\over|\nabla^\perp\varphi_0|},\ \ \bar\tau_2|_{S_1}=(0,0,1)\equiv\bar e_3,\cr
&\bar n|_{S_2(a_j)}=(-1)^j\bar e_3,\ \ j=1,2,\ \ a_1=-a,\ \ a_2=a,\cr
&\bar\tau_1|_{S_2(a_j)}=(1,0,0)\equiv\bar e_1,\ \ \bar\tau_2|_{S_2(a_j)}=(0,1,0)\equiv\bar e_2.\cr}
\label{2.7}
\end{equation}

\begin{lemma}\label{l2.7}
Let $\ro$, $v$, $h$ be given. Then $\chi$ is a solution to the problem
\begin{equation}\eqal{
&\ro(\chi_t+v\cdot\nabla\chi)-\nu\Delta\chi=\ro(F+\chi h_3-v_{3,x_1}h_2+v_{3,x_2}h_1)\cr
&\quad+\ro_{x_1}(v_{2,t}+v\cdot\nabla v_2+f_2)-\ro_{x_2}(v_{1,t}+v\cdot\nabla v_1+f_1)\quad &{\rm in}\ \ \Om^T,\cr
&\chi=v_i(n_{i,x_j}\tau_{1j}+\tau_{1i,x_j}n_j)+v\cdot\bar\tau_1(\tau_{12,x_1}- \tau_{11,x_2})\cr
&\quad+{\gamma\over\nu}v_j\tau_{1j}\equiv\chi_*\quad &{\rm on}\ \ S_1^T,\cr
&\chi_{,x_3}=0\quad &{\rm on}\ \ S_2^T,\cr
&\chi|_{t=0}=\chi(0)\quad &{\rm in}\ \ \Om,\cr}
\label{2.8}
\end{equation}
where the summation convention over the repeated indices is assumed.
\end{lemma}

\begin{proof}
Applying the two-dimensional rotation to the first two equations of $(\ref{1.1})_1$ yields $(\ref{2.8})_1$. The boundary condition $(\ref{2.8})_2$ is proved in \cite{Z1}. To prove $(\ref{2.8})_3$ we calculate
$$
(v_{2,x_1}-v_{1,x_2})_{,x_3}=v_{1,x_1x_3}-v_{1,x_2x_3}=d_{,x_2x_1}-d_{,x_1x_2}=0,
$$
where $(\ref{1.1})_9$ and $v_3|_{S_2(a_i)}=d|_{S_2(a_i)}=d_i$, $i=1,2$. This concludes the proof.
\end{proof}

Consider the problem
\begin{equation}
\Delta\varphi=f\quad {\rm in}\ \ \Om,\ \ \bar n\cdot\nabla\varphi|_S=0,\ \ \intop_\Om\varphi dx=0
\label{2.9}
\end{equation}

\begin{lemma}[see \cite{RZ1}]\label{l2.8}
Assume that $f\in L_{p,\mu}(\Om)$, $p\in(1,\infty)$, $\mu\in\R_+$. Then $\varphi\in V_{p,\mu}^2(\Om)$ and $\|\nabla^2\varphi\|_{L_{p,\mu}(\Om)}\le c\|f\|_{L_{p,\mu}(\Om)}$.
\end{lemma}

\section{Energy estimate for solutions to (\ref{1.1})}\label{s3}

In this Section we define and analyze weak solutions in order to obtain the energy type estimate for solutions to problem (\ref{1.1}). To accomplish this, we have to integrate by parts but this requires homogeneous Dirichlet boundary conditions for $v.$ Therefore, to make $(\ref{1.1})_6$ homogeneous we introduce the Hopf function $\eta$,
\begin{equation}
\eta(\sigma;\varepsilon,\varkappa)=\begin{cases}
1\ &0\le\sigma\le\varkappa e^{-1/\varepsilon}\equiv r,\cr
-\varepsilon\ln{\sigma\over\varkappa}\ &r<\sigma\le\varkappa,\cr
0\ &\varkappa<\sigma<\infty.\cr
\end{cases}
\label{3.1}
\end{equation}
We find the derivative
$$
{d\eta\over d\sigma}=\eta'(\sigma;\varepsilon,\varkappa)=\begin{cases}
0\ &0<\sigma\le r,\cr
-{\varepsilon\over\sigma}\ &r<\sigma\le\varkappa,\cr
0\ &\varkappa<\sigma<\infty,\cr
\end{cases}
$$
so that $|\eta'(\sigma;\varepsilon,\varkappa)|\le{\varepsilon\over\sigma}$. We define locally functions $\eta_i$ in an internal neighborhood of $S_2$ by setting
$$
\eta_i=\eta(\sigma_i;\varepsilon,\varkappa),\ \ i=1,2,
$$
where $\sigma_i$ denotes a local coordinate defined on a small neighborhood of
\begin{equation}\eqal{
&S_2(a_1,\varkappa)=\{x\in\Om\colon x_3\in(-a,-a+\varkappa)\}\cr
&S_2(a_2,\varkappa)=\{x\in\Om\colon x_3\in(a-\varkappa,a)\},\cr}
\label{3.2}
\end{equation}
$\sigma_1=-x_3$, $x_3\in(-a,-a+\varkappa)$ and $\sigma_2=x_3$, $x_3\in(a-\varkappa,a)$. Hence, $\sigma_i$ $i=1,2$, are positive. We extend functions $d_1$, $d_2$ so that
\begin{equation}
\tilde d_i|_{S_2(a_i)}=d_i,\ \ i=1,2,\ \ a_1=-a,\ \ a_2=a.
\label{3.3}
\end{equation}
Next we set
\begin{equation}
\alpha=\sum_{i=1}^2\tilde d_i\eta_i,\ \ b=\alpha\bar e_3,\ \ \bar e_3=(0,0,1).
\label{3.4}
\end{equation}
Then we introduce the function
\begin{equation}
u=v-b.
\label{3.5}
\end{equation}
Therefore
$$
\divv u=-\divv b=-\alpha_{x_3}\quad {\rm in}\ \ \Om,\ \ u\cdot\bar n|_S=0.
$$
Thus boundary conditions for $u$ are homogeneous but $u$ is not divergence free. In order to correct this we define $\varphi$ as a solution to the Neumann problem for the Poisson equation
\begin{equation}\eqal{
&\Delta\varphi=-\divv b\quad &{\rm in}\ \ \Om,\cr
&\bar n\cdot\nabla\varphi=0\quad &{\rm on}\ \ S,\cr
&\intop_\Om\varphi dx=0.\cr}
\label{3.6}
\end{equation}
Next, we set
\begin{equation}
w=u-\nabla\varphi=v-(b+\nabla\varphi)\equiv v-\delta.
\label{3.7}
\end{equation}
Hence, $w$ is divergence free and $w|_S=0$.

Consequently, for a given density $\ro$, a pair $(w,p)$ is a solution to the problem
\begin{equation}\eqal{
&\ro(w_t+w\cdot\nabla w+w\cdot\nabla\delta+\delta\cdot\nabla w)-\divv\T(w,p)\cr
&=\ro f-\ro(\delta_t+\delta\cdot\nabla\delta)+\nu\divv\D(\delta)\equiv \bar{F}(f,\ro,\delta,t)\quad &{\rm in}\ \ \Om^T,\cr
&\divv w=0\quad &{\rm in}\ \ \Om^T,\cr
&w\cdot\bar n=0\quad &{\rm on}\ \ S^T,\cr
&\nu\bar n\cdot\D(w)\cdot\bar\tau_\alpha+\gamma w\cdot\bar\tau_\alpha=-\nu\bar n\cdot\D(\delta)\cdot\bar\tau_\alpha-\gamma\delta\cdot\bar\tau_\alpha\cr
&\equiv B_{1\alpha}(\delta),\ \ \alpha=1,2,\quad &{\rm on}\ \ S_1^T,\cr
&\bar n\cdot\D(w)\cdot\bar\tau_\alpha=-\bar n\cdot\D(\delta)\cdot\bar\tau_\alpha\equiv B_{2\alpha}(\delta),\ \ \alpha=1,2,\quad &{\rm on}\ \ S_2^T,\cr
&w|_{t=0}=v(0)-\delta(0)\equiv w(0)\equiv w_0\quad &{\rm in}\ \ \Om,\cr}
\label{3.8}
\end{equation}
where we used that $\divv\delta=0$. Moreover, we have
\begin{equation}\eqal{
&\bar n|_{S_1}= {(\varphi_{0,x_1},\varphi_{0,x_2},0)\over\sqrt{\varphi_{0,x_1}^2+\varphi_{0,x_2}^2}}, \quad \bar\tau_1|_{S_1}={(-\varphi_{0,x_2},\varphi_{0,x_1},0)\over \sqrt{\varphi_{0,x_1}^2+\varphi_{0,x_2}^2}},\cr
&\bar\tau_2|_{S_1}=(0,0,1)=\bar e_3,\quad \bar n|_{S_2(-a)}=-\bar e_3,\cr
&\bar n|_{S_2(a)}=\bar e_3,\ \ \bar\tau_1|_{S_2(a_j)}=\bar e_1,\ \ \bar\tau_2|_{S_2(a_j)}=\bar e_2,\ \ j=1,2.\cr}
\label{3.9}
\end{equation}
where $a_1=-a$, $a_2=a$, $\bar e_1=(1,0,0)$, $\bar e_2=(0,1,0)$.

Since Dirichlet boundary conditions for $w$ are homogeneous and $w$ is divergence free, we can define weak solutions to problem (\ref{3.8}).

\begin{definition}\label{d3.1}
We call $w$ a weak solution to problem (\ref{3.8}) if for any sufficiently smooth function $\psi$ such that
$$
\divv\psi|_\Om=0,\quad \psi\cdot\bar n|_S=0
$$
the integral identity holds:
$$\eqal{
&\intop_{\Om^T}\ro(w_t+w\cdot\nabla w+w\cdot\nabla\delta+\delta\cdot\nabla w)\cdot\psi dxdt+\nu\intop_{\Om^T}\D(v)\cdot\D(\psi)dxdt\cr
&\quad+\gamma\sum_{\alpha=1}^2\intop_{S_1^T}w\cdot\bar\tau_\alpha\psi\cdot\bar\tau_\alpha dS_1dt-\sum_{\alpha,\sigma=1}^2\intop_{S_\sigma^T}B_{\sigma\alpha}\psi\cdot\bar\tau_\alpha dS_\sigma dt=\intop_{\Om^T}\bar{F}\cdot\psi dxdt\cr}
$$
\end{definition}

Exploiting ideas from \cite[Ch. 3]{RZ1},  we have
\eject
%\marginpar{\ \\ brak nawiasu ")"}
\begin{lemma}\label{l3.2}
Assume that $(w,\ro)$ is a solution to (\ref{3.8}), (\ref{2.1}) and there exist constants $\ro_*$, $\ro^*$, $0<\ro_*<\ro^*$ such that $\ro_*\le\ro\le\ro^*$. Let $\ro_1\in L_\infty(S_1^t)$, $d_1\in L_6(0,t;L_3(S_2))$, $\tilde d=(\tilde d_1,\tilde d_2)\in L_\infty(0,t;W_{3,\infty}^1(\Om))\cap L_2(0,t;W_{3,\infty}^1(\Om))$, $\tilde d_t\in L_2(0,t;W_{6/5}^1(\Om))$, $f\in L_2(0,t;L_{6/5}(\Om))$, $w(0)\in L_2(\Om)$, $t\le T$.\\
Then, for some increasing positive functions $\phi_1$, $\phi$ and $t<T$ the following a priori inequality holds
\begin{equation}\eqal{
&\ro_*|w(t)|_{2,\Om}^2+{\nu\over 4}\|w(t)\|_{1,2,\Om^t}^2+\gamma |w\cdot\bar\tau_\alpha|_{2,S_1^t}^2\cr
&\le\phi_1(|\ro_1|_{\infty,S_2^t},|d_1|_{3,6,S_2^t},\ro_*)\{\phi(\|\tilde d\|_{L_\infty(0,t;W_{3,\infty}^1(\Om))},\ro^*)\cr
&\quad\cdot[\|\tilde d\|_{L_2(0,t;W_{3,\infty}^1(\Om))}^2+\|\tilde d_t\|_{L_2(0,t;W_{6/5}^1(\Om))}^2+\|f\|_{L_2(0,t;L_{6/5}(\Om))}^2]\cr
&\quad+\ro^*|w(0)|_{2,\Om}^2\}.\cr}
\label{3.10}
\end{equation}
\end{lemma}

\begin{proof}
We multiply $(\ref{3.8})_1$ by $w$ and integrate over $\Om$. Then we have
\begin{equation}\eqal{
&\intop_\Om\ro(w_t\cdot w+w\cdot\nabla w\cdot w+w\cdot\nabla\delta\cdot w+\delta\cdot\nabla w\cdot w)dx\cr
&\quad-\intop_\Om\divv\T(w+\delta,p)\cdot wdx=\intop_\Om(f-\ro(\delta_t+\delta\cdot\nabla\delta))\cdot wdx.\cr}
\label{3.11}
\end{equation}
Taking into account problem (\ref{2.1}) we obtain
\begin{equation}\eqal{
&\intop_\Om\ro(w_t+w\cdot\nabla w+\delta\cdot\nabla w)\cdot wdx=\intop_\Om\ro(w_t+v\cdot\nabla w)\cdot wdx\cr
&={1\over 2}\intop_\Om(\ro\partial_tw^2+\ro v\cdot\nabla w^2)dx={1\over 2}{d\over dt}\intop_\Om\ro w^2dx\cr
&\quad+{1\over 2}\intop_{S_2(-a)}\ro_1v\cdot\bar nw^2dS_2+{1\over 2}\intop_{S_2(a)}\ro v\cdot\bar nw^2dS_2\cr
&={1\over 2}{d\over dt}\intop_\Om\ro w^2dx-{1\over 2}\intop_{S_2(-a)}\ro_1d_1w^2dS_2+{1\over 2}\intop_{S_2(a)}\ro d_2w^2dS_2.\cr}
\label{3.12}
\end{equation}
We have to examine the following term from the first term on the l.h.s. of (\ref{3.11})
$$\eqal{
I_1&=\intop_\Om\ro w\cdot\nabla\delta\cdot wdx\cr
&=\intop_\Om\ro w\cdot\nabla b\cdot wdx+\intop_\Om\ro w\cdot\nabla^2\varphi\cdot wdx\equiv I_3+I_4.\cr}
$$
In order to estimate $I_3$ and $I_4$ we introduce the sets:
$$\eqal{
&\tilde S_2(a_1,r,\varkappa)=\{x\in\Om\colon x_3\in(-a+r,-a+\varkappa)\},\cr
&\tilde S_2(a_2,r,\varkappa)=\{x\in\Om\colon x_3\in(a-\varkappa,a-r)\}.\cr}
$$
In view of (\ref{3.4}) and the Hardy inequality we have
$$\eqal{
|I_3|&=\bigg|\intop_\Om\ro w\cdot\nabla\sum_{i=1}^2(\tilde d_i\eta_i)w_3dx\bigg|\le\ro^*|w|_{6,\Om}\bigg|\nabla\sum_{i=1}^2\tilde d_i\eta_i\bigg|_{L_{3,\mu}(\Om)}|w_3|_{L_{2,-\mu}(\Om)}\cr
&\le c\ro^*|w|_{6,\Om}\|w_{3,x_3}\|_{L_{2,1-\mu}(\Om)}\bigg[\varepsilon \bigg(\sum_{i=1}^2\intop_{S_2(a_i,r,\varkappa)}|\tilde d_i|^3{\sigma_i^{3\mu}\over\sigma_i^3}dx\bigg)^{1/3}\cr
&\quad+\bigg(\sum_{i=1}^2\intop_{S_2(a_i,\varkappa)}|\tilde d_{i,x_3}|^3|\sigma_i|^{3\mu}dx\bigg)^{1/3}\bigg]\equiv I'_3\cr}
$$
Since $\Om$ is bounded we have
\begin{equation}\eqal{
I'_3&\le c\ro^*\|w\|_{1,\Om}^2\bigg[\varepsilon\sum_{i=1}^2\bigg(\sup_{x_3} \intop_{S_2(a_i,\varkappa)}|\tilde d_i|^3dx'\intop_r^\varkappa {\sigma_i^{3\mu}\over\sigma_i^3}d\sigma_i\bigg)^{1/3}\cr
&\quad+\sum_{i=1}^2\bigg(\sup_{x_3}\intop_{S_2(a_i,\varkappa)}|\tilde d_{i,x_3}|^3dx'\intop_0^\varkappa\sigma_i^{3\mu}d\sigma_i\bigg)^{1/3}\bigg]\cr
&\le c\varepsilon^*\|w\|_{1,\Om}^2\bigg[\varepsilon\varkappa^{\mu-2/3} \sum_{i=1}^2\sup_{x_3}|\tilde d_i|_{3,S_2(a_i)}+c\varkappa^{\mu+1/3} \sum_{i=1}^2 \sup_{x_3}|\tilde d_{i,x_3}|_{3,S_2(a_i)}\bigg]\cr
&\equiv c\ro^*\|w\|_{1,\Om}^2E^2(\varepsilon,\varkappa),\cr}
\label{3.13}
\end{equation}
where $\sigma_i=\dist\{S_2(a_i),x\}$, $s\in S_2(a_i,\varkappa)$. Consider $I_4$,
$$
|I_4|=\bigg|\intop_\Om\ro w\cdot\nabla^2\varphi\cdot wdx\bigg|\le\ro^*|w|_{6,\Om}\|w\|_{L_{2,-\mu}(\Om)} \|\nabla^2\varphi\|_{L_{3,\mu}(\Om)}\equiv I'_4,
$$
where $\varphi$ is a solution to problem (\ref{3.6}). In view of Lemma \ref{l2.5} and the Hardy inequality we have
\begin{equation}
I'_4\le c\ro^*\|w\|_{H^1(\Om)}^2|\divv b|_{3,\mu,\Om}\equiv I_4^2.
\label{3.14}
\end{equation}
Employing the definition of $b$ and the properties of function $\eta$ we have
$$\eqal{
\|\divv b\|_{L_{3,\mu}(\Om)}&\le c\varepsilon\bigg(\sum_{i=1}^2\intop_{S_2(a_i,r,\varkappa)}|\tilde d_i|^3{\sigma_i^{3\mu}\over\sigma_i^3}dx\bigg)^{1/3}\cr
&\quad+\bigg(\sum_{i=1}^2\intop_{S_2(a_i,\varkappa)}|\tilde d_{i,x_3}|^3|\sigma_i(x)|^{3\mu}dx\bigg)^{1/3}\equiv I_5.\cr}
$$
Repeating the considerations performed in (\ref{3.13}) implies
\begin{equation}
I_5\le cE^2(\varepsilon,\varkappa).
\label{3.15}
\end{equation}
In view of boundary conditions $(\ref{3.8})_{4,5}$ the second integral on the l.h.s. of (\ref{3.11}) is reformulated as follows
$$\eqal{
&-\intop_\Om\divv\T(w+\delta,p)\cdot wdx=-\intop_\Om\divv[\nu\D(w+\delta)-p\I]\cdot wdx\cr
&=-\intop_\Om\divv[\nu\D(w+\delta)]\cdot wdx+\intop_\Om\nabla p\cdot wdx\cr
&=\nu\intop_\Om D_{ij}(w+\delta)w_{j,x_i}dx-\nu\intop_\Om\partial_{x_j}[D_{ij}(w+\delta)w_i] dx+\intop_\Om\divv(pw)dx\cr
&\equiv I.\cr}
$$
The first term in $I$ equals
\begin{equation}
{\nu\over 2}|D_{ij}(w)|_{2,\Om}^2+\nu\intop_\Om D_{ij}(\delta)w_{j,x_i}dx,
\label{3.16}
\end{equation}
where the summation over repeated indices is assumed. By the Green theorem the second term in $I$ takes the form
\begin{equation}\eqal{
&-\nu\intop_{S_1}n_jD_{ij}(w+\delta)w_idS_1-\nu\intop_{S_2}n_jD_{ij}(w+\delta) w_idS_2\cr
&=-\nu\intop_{S_1}n_jD_{ij}(w+\delta)(w_{\tau_\alpha}\tau_{\alpha i}+w_nn_i)dS_1\cr
&\quad-\nu\intop_{S_2}n_jD_{ij}(w+\delta)(w_{\tau_\alpha}\tau_{\alpha i}+w_nn_i)dS_2\cr
&=\gamma\intop_{S_1}(|w_{\tau_\alpha}|^2+w_{\tau_\alpha}\delta_{\tau_\alpha})dS_1,\cr}
\label{3.17}
\end{equation}where $w_{\tau_\alpha}=w\cdot\bar\tau_\alpha$, $\alpha=1,2$, $w_n=w\cdot\bar n$ and conditions $(\ref{3.8})_{4,5}$ were used.

Using estimates (\ref{3.12})--(\ref{3.17}) in (\ref{3.11}), the Korn inequality and that
$$
\sum_{i=1}^2\sup_{x_3}|\tilde d_i|_{3,S_2(a_i)}+\sum_{i=1}^2\sup_{x_3}|\tilde d_{i,x_3}|_{3,S_2(a_i)}<\infty,
$$
we obtain for sufficiently small $\varepsilon$ and $\varkappa$ the inequality
\begin{equation}\eqal{
&{1\over 2}{d\over dt}\intop_\Om\ro w^2dx+\nu\|w\|_{1,\Om}^2+\gamma\sum_{\alpha=1}^2|w\cdot\bar\tau_\alpha|_{2,S_1}^2\cr
&\le{1\over 2}\intop_{S_2(-a)}\ro_1d_1w^2dS_2-{1\over 2}\intop_{S_2(a)}\ro d_2w^2dS_2\cr
&\quad+c\ro^*\|w\|_{1,\Om}^2\bigg[\varepsilon\varkappa^{\mu-2/3} \sum_{i=1}^2\sup_{x_3}|\tilde d_i|_{3,S_2(a_i)}\cr
&\quad+\varkappa^{\mu+1/3}\sum_{i=1}^2\sup_{x_3}|\tilde d_{i,x_3}|_{3,S_2(a_i)}+c\sum_{\alpha=1}^2|\delta\cdot\bar\tau_\alpha|_{2,S_1}^2+c |\D(\delta)|_{2,\Om}^2\cr
&\quad+\bigg|\intop_\Om(f-\ro(\delta_{t}+ \delta\cdot\nabla\delta))\cdot wdx\bigg|\bigg].\cr}
\label{3.18}
\end{equation}
From \cite{RZ1} Ch. 3, (\ref{3.17}) we have
\begin{equation}
\sum_{\alpha=1}^2|\delta\cdot\bar\tau_\alpha|_{2,S_1}^2\le c\|\tilde d\|_{1,3/2,\Om}^2+c{\varepsilon^2\over\varkappa^{2/3}}\exp\bigg({2\over 3\varepsilon}\bigg)\sup_{x_3}|\tilde d|_{3/2,S_2}^2.
\label{3.19}
\end{equation}
Next (\ref{3.18}) from \cite{RZ1}, Ch. 3 yields
\begin{equation}
|\D(\delta)|_{2,\Om}^2\le c\sum_{i=1}^2[\|\tilde d_i\|_{1,2,\Om}^2+{\varepsilon^2\over\varkappa}e^{1/\varepsilon}\sup_{x_3}|\tilde d_i|_{2,S_2}^2].
\label{3.20}
\end{equation}
Estimating the last integral on the r.h.s. of (\ref{3.18}) implies
\begin{equation}\eqal{
&\intop_\Om\ro(f-\delta_t-\delta\cdot\nabla\delta)\cdot wdx\le\varepsilon_1|w|_{6,\Om}^2\cr
&\quad+c(1/\varepsilon_1)(\ro^*)^2[|f|_{6/5,\Om}^2+|\delta_t|_{6/5,\Om}^2]+ \bigg|\intop_\Om\ro\delta\cdot\nabla\delta\cdot wdx\bigg|.\cr}
\label{3.21}
\end{equation}
In view of (\ref{3.19}) from \cite{RZ1}, Ch. 3 we derive
\begin{equation}
|\delta_t|_{6/5,\Om}\le\|\tilde d_t\|_{1,6/5,\Om}+c{\varepsilon\over\varkappa^{1/6}}e^{1/6\varepsilon}\sup_{x_3} |\tilde d_t|_{6/5,S_2}.
\label{3.22}
\end{equation}
Finally, we examine (see \cite[Ch. 3 (3.20)]{RZ1})
\begin{equation}\eqal{
&\bigg|\iintop_\Om\ro\delta\cdot\nabla\delta\cdot wdx\bigg|\cr
&\le\varepsilon_2|w|_{6,\Om}^2+c(1/\varepsilon_2)(\ro^*)^2(\|\tilde d\|_{1,2,\Om}^4+{\varepsilon^4\over\varkappa^2}e^{2/\varepsilon}\sup_{x_3}|\tilde d|_{2,S_2}^4).\cr}
\label{3.23}
\end{equation}
Using estimates (\ref{3.19})--(\ref{3.23}) in (\ref{3.18}) yields
\begin{equation}\eqal{
&{d\over dt}\intop_\Om\ro w^2dx+\nu\|w\|_{1,\Om}^2+\gamma\sum_{\alpha=1}^2|w\cdot\bar\tau_\alpha|_{2,S_1}^2\le \intop_{S_2(-a)}\ro_1d_1w^2dS_2\cr
&\quad-\intop_{S_2(a)}\ro d_2w^2dS_2+c\ro^*\|w\|_{1,\Om}^2(\varepsilon\varkappa^{\mu-2/3}+ \varkappa^{\mu+1/3})\sum_{i=1}^2\|\tilde d_i\|_{W_{3,\infty}^1(\Om)}^2\cr
&\quad+c\sum_{i=1}^2\bigg(\|\tilde d_i\|_{1,2,\Om}^2+{\varepsilon^2\over\varkappa}e^{1/\varepsilon}\|\tilde d_i\|_{2,\infty,\Om}^2\bigg)\cr
&\quad+c(\ro^*)^2\bigg[|f|_{6/5,\Om}^2+\bigg(1+{\varepsilon\over\varkappa^{1/6}} e^{1/6\varepsilon}\bigg)\sum_{i=1}^2\|\tilde d_{i,t}\|^2_{1,6/5,\Om}\bigg]\cr
&\quad+c(\ro^*)^2\bigg(1+{\varepsilon^4\over\varkappa^2}e^{2/\varepsilon}\bigg) \sum_{i=1}^2\|\tilde d_i\|_{1,2,\Om}^4.\cr}
\label{3.24}
\end{equation}
Since $\mu>2/3$ and $\varkappa<1$ we have
$$
\varepsilon\varkappa^{\mu-2/3}+\varkappa^{\mu+1/3}\le\varepsilon+\varkappa.
$$
Assuming that $\varepsilon$ and $\varkappa$ are so small that
$$
c\ro^*(\varepsilon+\varkappa)\|\tilde d\|_{W_{3,\infty}^1(\Om)}\le{\nu\over 2},
$$
where $\tilde d$ replaces $(\tilde d_1,\tilde d_2)$, we obtain
\begin{equation}\eqal{
&\varepsilon={\nu\over 4c\ro^*\|\tilde d\|_{W_{3,\infty}^1(\Om)}},\cr
&\varkappa={\nu\over 4c\ro^*\|\tilde d\|_{W_{3,\infty}^1(\Om)}}\cr}
\label{3.25}
\end{equation}
and we conclude the following inequality
\begin{equation}\eqal{
&{d\over dt}\intop_\Om\ro w^2dx+\nu\|w\|_{1,\Om}^2+\gamma\sum_{\alpha=1}^2|w\cdot\bar\tau_\alpha|_{2,S_1}^2\cr
&\le\intop_{S_2(-a)}\ro_1d_1w^2dS_2-\intop_{S_2(a)}\ro d_2w^2dS_2\cr
&\quad+\phi(\|\tilde d\|_{W_{3,\infty}^1(\Om)},\ro^*)\cdot(\|\tilde d\|_{W_{3,\infty}^1(\Om)}^2+\|d_t\|_{1,6/5,\Om}^2+|f|_{6/5,\Om}^2),\cr}
\label{3.26}
\end{equation}
where $\phi$ is an increasing positive function.

Estimating the first term on the r.h.s of (\ref{3.26}) we have
$$
\intop_{S_2(-a)}\ro_1d_1w^2dS_2\le|\ro_1|_{\infty,S_2(-a)} |d_1|_{3,S_2(-a)}|w|_{3,S_2(-a)}^2\equiv I.
$$
By the interpolation
$$
|w|_{3,S_2(-a)}^2\le\varepsilon^{1/3}|\nabla w|_{2,\Om}^2+c\varepsilon^{-5/3}|w|_{2,\Om}^2
$$
we obtain
$$
I\le\varepsilon|\nabla w|_{2,\Om}^2+c(1/\varepsilon) |\ro_1|_{\infty,S_2(-a)}^6|d_1|_{3,S_2(-a)}^6|w|_{2,\Om}^2.
$$
Using the estimate in (\ref{3.26}) and assuming that $\varepsilon$ is sufficiently small we derive the inequality
\begin{equation}\eqal{
&{d\over dt}\intop_\Om\ro w^2dx+{\nu\over 2}\|w\|_{1,\Om}^2+\gamma\sum_{\alpha=1}^2|w\cdot\bar\tau_\alpha|_{2,S_1}^2\cr
&\le c|\ro_1|_{\infty,S_2(-a)}^6|d_1|_{3,S_2(-a)}^6|w|_{2,\Om}^2\cr
&\quad+\phi(\|\tilde d\|_{W_{3,\infty}^1(\Om)},\ro^*)\cdot(\|\tilde d\|_{W_{3,\infty}^1(\Om)}+\|\tilde d_t\|_{1,6/5,\Om}^2+|f|_{6/5,\Om}^2).\cr}
\label{3.27}
\end{equation}
Let $T>0$ be fixed. In order to obtain an energy type estimate in time interval $(0,T)$ we observe that
$$
\intop_\Om|w|^2dx\le\intop_\Om{\ro\over\ro_*}|w|^2dx\le{1\over\ro_*} \intop_\Om\ro|w|^2dx.
$$
\goodbreak

\noindent
Thus, employing this in (\ref{3.27}), yields
\begin{equation}\eqal{
&{d\over dt}\intop_\Om\ro w^2dx+{\nu\over 2}\|w\|_{1,\Om}^2+\gamma\sum_{\alpha=1}^2|w\cdot\bar\tau_\alpha|_{2,S_1}^2\cr
&\le c|\ro_1|_{\infty,S_2(-a)}^6|d_1|_{3,S_2(-a)}^6\intop_\Om\ro |w|^2dx\cr
&\quad+\phi(\|\tilde d\|_{W_{3,\infty}^1(\Om)},\ro^*)\cdot(\|\tilde d\|_{W_{3,\infty}^1(\Om)}+\|\tilde d_t\|_{1,6/5,\Om}^2+|f|_{6/5,\Om}^2).\cr}
\label{3.28}
\end{equation}
Now, we consider (\ref{3.28}) in the time interval $(0,T)$. Then
\begin{equation}\eqal{
&{d\over dt}\bigg\{\intop_\Om\ro w^2dx\exp\bigg[{\nu\over 4}t-{c\over\ro_*}\intop_0^t|\ro_1|_{\infty,S_2(-a)}^6|d_1|_{3,S_2(-a)}^6 dt'\bigg]\bigg\}\cr
&\le\phi(\|\tilde d\|_{W_{3,\infty}^1(\Om)},\ro^*)\cdot(\|\tilde d\|_{W_{3,\infty}^1(\Om)}+\|\tilde d_t\|_{1,6/5,\Om}^2+|f|_{6/5,\Om}^2)\cr
&\quad\cdot\exp\bigg[{\nu\over 4}t-{c\over\ro_*}\intop_0^t|\ro_1|_{\infty,S_2(-a)}^6 |d_1|_{3,S_2(-a)}^6dt'\bigg].\cr}
\label{3.29}
\end{equation}
Integrating (\ref{3.29}) with respect to time yields
\begin{equation}\eqal{
&\intop_\Om\ro(t)w(t)^2dx\le\exp\bigg[\intop_0^t|\ro_1|_{\infty,S_2(-a)}^6 |d_1|_{3,S_2(-a)}^6dt'\bigg]\cr
&\quad\cdot\intop_0^t\phi(\|\tilde d\|_{W_{3,\infty}^1(\Om)},\ro^*)\cdot(\|\tilde d\|_{W_{3,\infty}^1(\Om)}+\|\tilde d_t\|_{1,6/5,\Om}^2+|f|_{6/5,\Om}^2)dt'\cr
&\quad+\exp\bigg[-{\nu\over 4}t+{c\over\ro_*}\intop_0^t|\ro_1|_{\infty,S_2(-a)}^6 |d_1|_{3,S_2(-a)}^6dt'\bigg]\intop_\Om\ro(0)w(0)^2dx,\cr}
\label{3.30}
\end{equation}
where $t\le T$. Integrating (\ref{3.28}) with respect to time from 0 to $t\le T$ and using (\ref{3.30}) we have
\begin{equation}\eqal{
&\intop_\Om\ro(t)w^2(t)dx+{\nu\over 4}\|w\|_{1,2,\Om^t}^2+\gamma\sum_{\alpha=1}^2|w\cdot\bar\tau_\alpha|_{2,S_1^t}^2\cr
&\le{c\over\ro_*}\intop_0^t|\ro_1|_{\infty,S_2}^6|d_1|_{3,S_2}^6dt' \bigg\{\sup_{t'\le t}\exp\bigg[{c\over\ro_*}\intop_0^{t'}|\ro_1|_{\infty,S_2}^6 |d_1|_{3,S_2}^6d\tau\bigg]\cr
&\quad\cdot\intop_0^{t'}\phi(\|\tilde d\|_{W_{3,\infty}^1(\Om)},\ro^*)\cdot(\|\tilde d\|_{W_{3,\infty}^1(\Om)}^2+\|\tilde d_t\|_{1,6/5,\Om}^2+|f|_{6/5,\Om}^2)d\tau\cr
&\quad+\exp\bigg[-{\nu\over 4}+{c\over\ro_*}\intop_0^{t'}|\ro_1|_{\infty,S_2(-a)}^6 |d_1|_{3,S_2(-a)}^6d\tau\bigg]\intop_\Om\ro(0)w(0)^2dx\bigg\}\cr
&\quad+\intop_0^t\phi(\|\tilde d\|_{W_{3,\infty}^1(\Om)},\ro^*)\cdot(\|\tilde d\|_{W_{3,\infty}^1(\Om)}+\|\tilde d_t\|_{1,6/5,\Om}^2+|f|_{6/5,\Om}^2)dt'\cr
&\quad+\intop_\Om\ro(0)w(0)^2dx.\cr}
\label{3.31}
\end{equation}
Simplifying (\ref{3.31}), we get
\begin{equation}\eqal{
&\ro^*|w(t)|_{2,\Om}^2+{\nu\over 4}\|w\|_{1,2,\Om^t}^2+\gamma \sum_{\alpha=1}^2|w\cdot\bar\tau_\alpha|_{2,S_1^t}^2\cr
&\le\phi_1(|\ro_1|_{\infty,S_2^t},|d_1|_{3,6,S_2^t},\ro_*)\{\phi (\sup_t\|\tilde d\|_{W_{3,\infty}^1(\Om)},\ro^*)\cr
&\quad\cdot[\|\tilde d\|_{L_2(0,t;W_{3,\infty}^1(\Om))}+\|\tilde d_t\|_{1,6/5,2,\Om^t}^2+|f|_{6/5,2,\Om^t}^2]+\ro^*|w(0)|_{2,\Om}^2\},\ \ t\le T.\cr}
\label{3.32}
\end{equation}
The above inequality implies (\ref{3.10}) and concludes the proof of Lemma \ref{l3.2}.
\end{proof}

Finally, we conclude the energy estimate for $(v,p,\ro)$.
\begin{lemma}\label{l3.3}
Let the assumptions of Lemma \ref{l3.2} hold. Then a solution $(v,p,\ro)$ of (\ref{1.1}) satisfies the inequality
\begin{equation}\eqal{
\|v\|_{V(\Om^t)}^2&\le\phi_1(|\ro_1|_{\infty,S_2^t},|d_1|_{3,6,S_2^t}, \ro_*,\ro^*)\{\phi(\sup_t\|\tilde d\|_{W_{3,\infty}^1(\Om)},\ro^*)\cdot\cr
&\quad\cdot[\|\tilde d\|_{L_2(0,t;W_{3,\infty}^1(\Om))}^2+\|\tilde d_t\|_{1,6/5,2,\Om^t}^2+|f|_{6/5,2,\Om^t}^2] +\ro^*|v(0)|_{2,\Om}^2\},\cr}
\label{3.33}
\end{equation}
where $\ro_*\le\ro\le\ro^*$ and $\ro_*$, $\ro^*$ are described in Lemma \ref{l2.4}.
\end{lemma}

\begin{proof}
To prove the lemma we have to estimate the norm
\begin{equation}
\|v\|_{V(\Om^t)}^2\le\|w\|_{V(\Om^t)}^2+\|\delta\|_{V(\Om^t)}^2,
\label{3.34}
\end{equation}
where the first norm on the r.h.s. has been already considered in Lemma \ref{l3.2}, so
\begin{equation}\eqal{
\|\delta\|_{V(\Om^t)}^2&\le\|b\|_{V(\Om^t)}^2+\|\nabla\varphi\|_{V(\Om^t)}^2\cr
&\le|b|_{2,\infty,\Om^t}^2+|\nabla b|_{2,\Om^t}^2+|\nabla\varphi|_{2,\infty,\Om^t}^2+|\nabla^2\varphi|_{2,\Om^t}^2.\cr}
\label{3.35}
\end{equation}
To estimate norms of functions on the r.h.s. of (\ref{3.35}) we use estimate (3.40) in \cite[Ch. 3]{RZ1}. Then we have
\begin{equation}\eqal{
\|\delta\|_{V(\Om^t)}^2&\le c\bigg(|\tilde d|_{2,\infty,\Om^t}^2+\|\tilde d\|_{1,2,2,\Om^t}^2+\|\tilde d\|_{L_\infty(0,t;W_{3,\infty}^1(\Om))}^4\cr
&\quad\cdot\exp\bigg(c\sup_t\|\tilde d\|_{L_\infty(0,t;W_{3,\infty}^1(\Om))}\bigg)\bigg)\cdot|\tilde d|_{L_2(0,t;L_{2,\infty}(\Om))}^2\cr
&\quad+\sup_{t'\le t}|d(t')|_{2,\infty,\Om}^2.\cr}
\label{3.36}
\end{equation}
From (\ref{3.34}), (\ref{3.36}) and (\ref{3.10}) we obtain  (\ref{3.33}). This concludes the proof.
\end{proof}

Summarizing, we formulate the theorem on weak solutions

\begin{theorem}[see Lemma \ref{l3.3}]\label{t3.4}
Assume that there exist constants $\ro_*$, $\ro^*$, $0<\ro_*<\ro^*$ such that $\ro_*\le\ro\le\ro^*$. Assume that $\ro_1\in L_\infty(S_1^t)$, $d_1\in L_6(0,t;L_3(S_2(-a)))$. Assume that $\tilde d_i$ is an extension of $d_i$ into $\Om$, $i=1,2$ and $\tilde d=(\tilde d_1,\tilde d_2)$ is such that $\tilde d\in L_\infty(0,t;W_{3,\infty}^1(\Om))\cap L_2(0,t;W_{3,\infty}^1(\Om))$, $\tilde d_t\in L_2(0,t;W_{6/5}^1(\Om))$, $f\in L_2(0,t;L_{6/5}(\Om))$, $t\le T$ and $v(0)\in L_2(\Om)$.
Then for solutions to (\ref{1.1}) there exist positive increasing functions $\phi$, $\phi_1$ such that
\begin{equation}\eqal{
\|v\|_{V(\Om^t)}^2&\le\phi_1(|\ro_1|_{\infty,S_2^t},|
d_1|_{3,6,S_2^t(-a)},\ro_*,\ro^*)\cdot\cr
&\quad\cdot\{\phi(\|\tilde d\|_{L_\infty(0,t;W_{3,\infty}^1(\Om))})[\|\tilde d\|_{L_2(0,t;W_{3,\infty}^1(\Om))}^2\cr
&\quad+\|\tilde d_t\|_{1,6/5,2,\Om^t}^2+|f|_{6/5,2,\Om^t}^2]+ |v(0)|_{2,\Om}^2\}\equiv A_1^2\cr}
\label{3.37}
\end{equation}
\end{theorem}

\begin{remark}\label{r3.5}
We have to emphasize that there is no restriction on time $T$.
\end{remark}

\section{Energy estimate for solutions to problem (\ref{2.6})}\label{s4}

In this Section, we consider function $h =v_{x_3}$ and equations for $h$, i.e. (\ref{2.6}).  To find an energy type estimate for solutions to problem (\ref{2.6}) we have to make the Dirichlet boundary conditions $(\ref{2.6})_{4,5}$ homogeneous. For this purpose we introduce a function $\tilde h$ such that (see \cite{RZ1}, (4.14))
\begin{equation}\eqal{
&\divv\tilde h=0\quad &{\rm in}\ \ \Om,\cr
&\tilde h=0\quad &{\rm on}\ \ S_1,\cr
&\tilde h_i=-d_{x_i},\ \ i=1,2\quad &{\rm on}\ \ S_2,\cr
&\tilde h_3=0\quad &{\rm on}\ \ S_2.\cr}
\label{4.1}
\end{equation}

\begin{lemma} (see Lemma 4.3 from Ch. 4 \cite{RZ1}).\label{l4.1}
Assume that $d=(d_1,d_2)$, $d_{x'}\in W_\sigma^1(S_2)$, $d_{x't}\in L_\sigma(S_2)$, $\sigma\in(1,\infty)$. Then there exists a solution to problem (\ref{4.1}) such that $\tilde h\in W_\sigma^1(\Om)$, $\tilde h_t\in L_\sigma(\Om)$, and
\begin{equation}\eqal{
&\|\tilde h\|_{1,\sigma,\Om}\le c\|d_{x'}\|_{1,\sigma,S_2},\cr
&|\tilde h_t|_{\sigma,\Om}\le c|d_{x't}|_{\sigma,S_2},\cr}
\label{4.2}
\end{equation}
where $c$ does not depend on $\tilde h$. Let us introduce the function
\begin{equation}
k=h-\tilde h.
\label{4.3}
\end{equation}
Then $k$ is a solution to the problem
\begin{equation}\eqal{
&\ro k_t-\divv\T(h,q)=-\ro(v\cdot\nabla h+h\cdot\nabla v)-\ro\tilde h_t+\ro g\cr
&\quad-\ro_{x_3}(v_t+v\cdot\nabla v-f)\equiv G\quad &{\rm in}\ \ \Om^T,\cr
&\divv k=0\quad &{\rm in}\ \ \Om^T,\cr
&\bar n\cdot k=0,\ \ \nu\bar n\cdot\D(h)\cdot\bar\tau_\alpha+\gamma h\cdot\bar\tau_\alpha=0,\ \ \alpha=1,2\quad &{\rm on}\ \ S_1^T,\cr
&k_i=0,\ \ i=1,2,\ \ h_{3,x_3}=\Delta'd\quad &{\rm on}\ \ S_2^T,\cr
&k|_{t=0}=h(0)-\tilde h(0)\equiv k(0)\quad &{\rm in}\ \ \Om,\cr}
\label{4.4}
\end{equation}
where $g=f_{x_3}$, $\Delta'=\partial_{x_1}^2+\partial_{x_2}^2$ and $v$ is a solution to problem (\ref{1.1}).
\end{lemma}

In view of decomposition (\ref{4.3}) we write problem (\ref{4.4}) in the form
\begin{equation}\eqal{
&\ro k_t+\ro v\cdot\nabla k-\divv\T(h,q)=-\ro(k\cdot\nabla v+\tilde h\cdot\nabla v+v\cdot\nabla\tilde h)\cr
&\quad-\ro\tilde h_t+\ro g-\ro_{x_3}(v_t+v\cdot\nabla v-f)\equiv G\quad &{\rm in}\ \ \Om^T,\cr
&\divv k=0\quad &{\rm in}\ \ \Om^T,\cr
&\bar n\cdot k=0,\ \ \nu\bar n\cdot\D(h)\cdot\bar\tau_\alpha+\gamma h\cdot\bar\tau_\alpha=0,\ \ \alpha=1,2\quad &{\rm on}\ \ S_1^T,\cr
&k_i=0,\ \ i=1,2,\ \ h_{3,x_3}=\Delta'd\quad &{\rm on}\ \ S_2^T,\cr
&k|_{t=0}=k(0)\quad &{\rm in}\ \ \Om.\cr}
\label{4.5}
\end{equation}
Projecting div $k$ on $S_2$ yields
\begin{equation}
\divv k|_{S_2}=k_{3,x_3}|_{S_2}=0
\label{4.6}
\end{equation}

\begin{lemma}\label{l4.2}
Assume that $d_{x'}\in L_2(0,t;W_3^1(S_2))$, $d_t\in L_2(0,t;W_2^1(S_2))$, $f_3\in L_2(0,t;L_{4/3}(S_2))$, $g\in L_2(\Om^t)$, $v\in L_2(0,t;W_3^1(\Om))$, $v_t\in L_2(\Om^t)$, $v\in L_\infty(0,t;L_\infty(\Om))$, $\ro_*\le\ro\le\ro^*$, where $\ro_*$, $\ro^*$ are positive constants from Lemma \ref{l2.4}. Assume that $\ro_{x_3}\in L_\infty(0,t;L_3(\Om))$, $k(0)\in L_2(\Om)$.\\
Let $\La_1=\|d_{x'}\|_{1,3,2,S_2^t}+|d_{x'}|_{2,\infty,S_2^t}+\|d_t\|_{1,2,S_2^t}+ |f_3|_{4/3,2,S_2^t}+|g|_{2,\Om^t}$, $D_1=|d_1|_{3,2,S_2^t}$, $V_1=\|v\|_{1,3,2,\Om^t}$ and $A_1$ is defined in (\ref{3.37}). Then
\begin{equation}\eqal{
\|k\|_{V(\Om^t)}&\le c\phi(\ro_*,\ro^*,D_1,V_1,A_1)[\La_1+ |\ro_{x_3}|_{3,\infty,\Om^t}(|v_t|_{2,\Om^t}+|f|_{2,\Om^t}\cr
&\quad+|v|_{\infty,\infty,\Om^t})+|k(0)|_{2,\Om}].\cr}
\label{4.7}
\end{equation}
\end{lemma}

\begin{proof}
Multiply $(\ref{4.5})_1$ by $k$, use problem (\ref{2.1}) and integrate over $\Om$. Then we obtain
\begin{equation}\eqal{
&{1\over 2
}{d\over dt}\intop_\Om\ro k^2dx+\intop_\Om\divv(\ro vk^2)dx-\intop_\Om\divv\T(h,q)\cdot kdx\cr
&=-\intop_\Om\ro k\cdot\nabla v\cdot kdx-\intop_\Om\ro\tilde h\cdot\nabla v\cdot kdx-\intop_\Om\ro v\cdot\nabla\tilde h\cdot kdx\cr
&\quad-\intop_\Om\ro\tilde h_t\cdot kdx+\intop_\Om\ro g\cdot kdx-\intop_\Om\ro_{x_3}(v_t+v\cdot\nabla v-f)\cdot kdx\equiv\intop_\Om Gkdx.\cr}
\label{4.8}
\end{equation}
Applying the Green theorem the second term on the l.h.s. of (\ref{4.8}) equals
$$
-\intop_{S_2(-a)}\ro_1d_1k_3^2dS_2+\intop_{S_2(a)}\ro d_2k_3^2dS_2,
$$
where boundary condition $(\ref{4.5})_4$ is used.

Integrating by parts, the third term on the l.h.s. of (\ref{4.8}) takes the form
$$\eqal{
&-\intop_{S_1}\bar n\cdot\T(h,q)\cdot kdS_1-\sum_{i=1}^2\intop_{S_2(a_i)}\bar n\cdot\T(h,q)\cdot kdS_2\cr
&\quad+{\nu\over 2}\intop_\Om\D(h)\cdot\D(k)dx\equiv I_1+I_2+I_3,\cr}
$$
where
$$\eqal{
I_1&=-\intop_{S_1}\bar n\cdot\T(h,q)\cdot\bar\tau_\alpha k\cdot\bar\tau_\alpha dS_1=\gamma\intop_{S_1}h\cdot\bar\tau_\alpha k\cdot\bar\tau_\alpha dS_1\cr
&=\gamma|k\cdot\bar\tau_\alpha|_{2,S_1}^2+\gamma\intop_{S_1}\tilde h\cdot\bar\tau_\alpha k\cdot\bar\tau_\alpha dS_1,\cr}
$$
$$\eqal{
I_2&=-\sum_{i=1}^2\intop_{S_2(a_i)}\T_{33}(h,q)k_3dS_2=-\sum_{i=1}^2\intop_{S_2(a_i)} (2\nu h_{3,x_3}-q)k_3dS_2\cr
&=-\sum_{i=1}^2\intop_{S_2(a_i)}(2\nu\Delta'd_i-q)k_3dS_2\equiv I_2^1+I_2^2.\cr}
$$
To examine the last integral we use the third component of $(\ref{1.1})_1$ projected on $S_2$. On $S_2(-a)$ we have
$$
\ro_1d_{1,t}+\ro_1v'\cdot\nabla'd_1+\ro_1d_1h_3-\ro_1f_3= 2\nu\Delta'd_1-q,
$$
where $v'=(v_1,v_2)$, $\nabla'=(\partial_{x_1},\partial_{x_2})$.

Hence,
$$
I_2^1=\intop_{S_2(-a)}(-\ro_1d_{1,t}+\ro_1f_3)k_3dS_2-\intop_{S_2(-a)} \ro_1v'd_{1,x'}k_3dS_2-\intop_{S_2(-a)}\ro_1d_1k_3^2dS_2,
$$
where we utilized that $\tilde h_3|_{S_2}=0$.

We estimate the first term in $I_2^1$ by
$$
\varepsilon_1|k_3|_{4,S_2}^2+c(1/\varepsilon_1)(\ro_1^*)^2(|d_{1,t}|_{4/3,S_2}^2+ |f_3|_{4/3,S_2}^2),
$$
the second by
$$
\varepsilon_2|k_3|_{4,S_2}^2+c(1/\varepsilon_2)(\ro_1^*)^2|v'|_{4,S_2}^2 |d_{1,x'}|_{2,S_2}^2
$$
and the last one as follows (see \cite[Ch. 2, Sect. 10]{BIN})
$$\eqal{
&\intop_{S_2(-a)}\ro_1d_1k_3^2dS_2\le\ro_1^*|d_1|_{3,S_2}|k_3|_{3,S_2}^2\cr
&\le(\varepsilon^{1/3}|\nabla k_3|_{2,\Om}^2+c\varepsilon^{-5/3}|k_3|_{2,\Om}^2)\ro_1^*|d_1|_{3,S_2}\cr
&\le\varepsilon_3^{1/3}|\nabla k_3|_{2,\Om}^2+c\varepsilon_3^{-5/3}(\ro_1^*)^6|d_1|_{3,S_2}^6|k_3|_{2,\Om}^2.\cr}
$$
To consider $I_2^2$ we calculate
$$
(2\nu\Delta'd_2-q)|_{S_2}=\ro d_{1,t}+\ro v'\cdot\nabla'd_2+\ro d_2k_3-\ro f_3.
$$
Then
$$
I_2^2=\intop_{S_2(a)}(-\ro d_{2,t}+\ro f_3)k_3dS_2-\intop_{S_2(a)}\ro v'd_{2,x'}k_3dS_2-\intop_{S_2(a)}\ro d_2k_3^2dS_2.
$$
To estimate $I_2^2$ we use Lemma \ref{l2.4}. Similarly as in the estimate of $I_2^1$ the first term in $I_2^2$ is bounded by
$$
\varepsilon_4|k_3|_{4,S_2}^2+c(1/\varepsilon_4)(\ro^*)^2(|d_{2,t}|_{4/3,S_2}^2+ |f_3|_{4/3,S_2}^2),
$$
the second by
$$
\varepsilon_5|k_3|_{4,S_2}^2+c(1/\varepsilon_5)(\ro^*)^2|v'|_{4,S_2}^2 |d_{2,x'}|_{2,S_2}^2
$$
and finally the last one by
$$
\varepsilon_6^{1/3}|\nabla k_3|_{2,\Om}^2+c\varepsilon_6^{-5/3}(\ro^*)^6|d_2|_{3,S_2}^6|k_3|_{2,\Om}^2.
$$
Employing the above estimates in (\ref{4.8}) with $\varepsilon_1-\varepsilon_6$ sufficiently small, using the Korn inequality (see Lemma \ref{l2.5}) and exploiting the notation $d=(d_1,d_2)$ we derive the inequality
\begin{equation}\eqal{
&{d\over dt}\intop_\Om\ro k^2dx+\nu\|k\|_{1,\Om}^2+\gamma|k\cdot\bar\tau_\alpha|_{2,S_1}^2\le \intop_{S_2(-a)}\ro_1d_1k_3^2dS_2\cr
&\quad+c|\tilde h\cdot\bar\tau_\alpha|_{2,S_2}^2+c|\D(\tilde h)|_{2,\Om}^2+c(\ro^*)^2(|d_t|_{4/3,S_2}^2+|f_3|_{4/3,S_2}^2\cr
&\quad+|v'|_{4,S_2}^2|d_{x'}|_{2,S_2}^2)+c(\ro^*)^6|d|_{3,S_2}^6|k_3|_{2,\Om}^2+ \bigg|\intop_\Om G\cdot kdx\bigg|,\cr}
\label{4.9}
\end{equation}
where we used that $\ro_1^*\le\ro^*$ and $d=(d_1,d_2)$.

The first term on the r.h.s. of (\ref{4.9}) can be estimated by the same bound as the third terms in $I_2^1$, $I_2^2$. In view of Lemma \ref{l4.1} the second and the third terms on the r.h.s. of (\ref{4.9}) are bounded by
$$
c\|d_{x'}\|_{1,2,S_2}^2.
$$
Using the above estimates in (\ref{4.9}) yields
\begin{equation}\eqal{
&{d\over dt}\intop_\Om\ro k^2dx+\nu\|k\|_{1,\Om}^2+\gamma|k\cdot\bar\tau_\alpha|_{2,S-2}^2\le c(\ro^*)^6|d_1|_{3,S_2}^6|k_3|_{2,\Om}^2\cr
&\quad+c\|d_{x'}\|_{1,2,S_2}^2+c(\ro^*)^2(|d_t|_{4/3,S_2}^2+|f_3|_{4/3,S_2}^2+ |v'|_{4,S_2}^2|d_{x'}|_{2,S_2}^2)\cr
&\quad+\bigg|\intop_\Om G\cdot kdx\bigg|.\cr}
\label{4.10}
\end{equation}
Finally, we shall estimate the last term on the r.h.s. of (\ref{4.10}). To this end we use the r.h.s. of (\ref{4.8}). We estimate the first term by
$$
\ro^*\intop_\Om|\nabla v|k^2dx\le\varepsilon_1|k|_{6,\Om}^2+c(1/\varepsilon_1)(\ro^*)^2|\nabla v|_{3,\Om}^2|k|_{2,\Om}^2,
$$
the second term by
$$\eqal{
&\varepsilon_2|k|_{6,\Om}^2+c(1/\varepsilon_2)(\ro^*)^2|\nabla v|_{3,\Om}^2|\tilde h|_{2,\Om}^2\cr
&\le\varepsilon_2|k|_{6,\Om}^2+c(1/\varepsilon_2)(\ro^*)^2|\nabla v|_{3,\Om}^2|d_{x'}|_{2,S_2}^2,\cr}
$$
where Lemma \ref{l4.1} was used.

Continuing, we estimate the third term by
$$\eqal{
&\varepsilon_3|k|_{6,\Om}^2+c(1/\varepsilon_3)(\ro^*)^2|v|_{2,\Om}^2|\nabla\tilde h|_{3,\Om}^2\cr
&\le\varepsilon_3|k|_{6,\Om}^2+c(1/\varepsilon_3)(\ro^*)^2A_1^2 \|d_{x'}\|_{1,3,S_2}^2,\cr}
$$
where Lemmas \ref{l3.3} and \ref{l4.1} were used.

Next, the fourth term is bounded by
$$
\varepsilon_4|k|_{2,\Om}^2+c/\varepsilon_4(\ro^*)^2|\tilde h_t|_{2,\Om}^2
\le\varepsilon_4|k|_{2,\Om}^2+c/\varepsilon_4(\ro^*)^2 |d_{x't}|_{2,S_2}^2
$$
and the last but one term by
$$
\varepsilon_5|k|_{2,\Om}^2+c/\varepsilon_5(\ro^*)^2|g|_{2,\Om}^2.
$$
Finally, the last term on the r.h.s. of (\ref{4.8}) is bounded by
$$
\varepsilon_6|k|_{6,\Om}^2+c(1/\varepsilon_6)|\ro_{x_3}|_{3,\Om}^2 (|v_t|_{2,\Om}^2+|v|_{\infty,\Om}^2|\nabla v|_{2,\Om}^2+|f|_{2,\Om}^2).
$$
Using the aobve estimates with sufficiently small $\varepsilon_1-\varepsilon_6$ in (\ref{4.10}) we derive the inequality
\begin{equation}\eqal{
&{d\over dt}\intop_\Om\ro k^2dx+\nu\|k\|_{1,\Om}^2+\gamma|k\cdot\bar\tau_\alpha|_{2,S_1}^2\le c((\ro^*)^6|d_1|_{3,S_2}^6\cr
&\quad+|\nabla v|_{3,\Om}^2)|k|_{2,\Om}^2+c(\ro^*)^2(|\nabla v|_{3,\Om}^2+|v'|_{4,S_2}^2)|d_{x'}|_{2,S_2}^2\cr
&\quad+c(\ro^*)^2(1+A_1^2)[\|d_{x'}\|_{1,3,S_2}^2+\|d_t\|_{1,2,S_2}^2+ |f_3|_{4/3,S_2}^2+|g|_{2,\Om}^2]\cr
&\quad+c|\ro_{x_3}|_{3,\Om}^2(|v_t|_{2,\Om}^2+|f|_{2,\Om}^2)+c |\ro_{x_3}|_{3,\Om}^2|v|_{\infty,\Om}^2|\nabla v|_{2,\Om}^2.\cr}
\label{4.11}
\end{equation}
Let $\nu_1=\nu/\ro^*$. Then (\ref{4.11}) implies the inequality
\begin{equation}\eqal{
&{d\over dt}\intop_\Om\ro k^2dx+\nu_1\intop_\Om\ro k^2dx-\phi(\ro_*,\ro^*)(|d_1|_{3,S_2}^6+\|v\|_{1,3,\Om}^2) \intop_\Om\ro k^2dx\cr
&\le c(\ro^*)^2\|v\|_{1,3,\Om}^2|d_{x'}|_{2,S_2}^2\cr
&\quad+c(\ro^*)^2(1+A_1^2)[\|d_{x'}\|_{1,3,S_2}^2+\|d_t\|_{1,2,S_2}^2+ |f_3|_{4/3,S_2}^2+|g|_{2,\Om}^2]\cr
&\quad+c|\ro_{x_3}|_{3,\Om}^2(|v_t|_{2,\Om}^2+|f|_{2,\Om}^2)+c |\ro_{x_3}|_{3,\Om}^2|v|_{\infty,\Om}^2|\nabla v|_{2,\Om}^2.\cr}
\label{4.12}
\end{equation}
Introduce the notation
\begin{equation}
D_1(t)=|d_1|_{3,2,S_2^t},\quad V_1(t)=\|v\|_{1,3,2,\Om^t}
\label{4.13}
\end{equation}
Then (\ref{4.12}) can be written in the form
\begin{equation}\eqal{
&{d\over dt}\bigg[\intop_\Om\ro k^2dx\exp(\nu_1t-\phi(\ro_*,\ro^*)(D_1^6(t)+V_1^2(t)))\bigg]\cr
&\le c\{(\ro^*)^2\|v\|_{1,3,\Om}^2|d_{x'}|_{2,S_2}^2\cr
&\quad+(\ro^*)^2(1+A_1^2)[\|d_{x'}\|_{1,3,S_2}^2+\|d_t\|_{1,2,S_2}^2+ |f_3|_{4/3,S_2}^2+|g|_{2,\Om}^2]\cr
&\quad+|\ro_{x_3}|_{3,\Om}^2(|v_t|_{2,\Om}^2+|f|_{2,\Om}^2)+ |\ro_{x_3}|_{3,\Om}^2|v|_{\infty,\Om}^2|\nabla v|_{2,\Om}^2\}\cdot\cr
&\quad\cdot\exp[\nu_1t-\phi(\ro_*,\ro^*)(D_1^6(t)+V_1^2(t))].\cr}
\label{4.14}
\end{equation}
Integrating (\ref{4.14}) with respect to time from $0$ to $t\le T$ and introducing the notation
\begin{equation}
\La_1 \equiv \La_1(t)=\|d_{x'}\|_{1,3,2,S_2^t}+|d_{x'}|_{2,\infty,S_2^t}+ \|d_t\|_{1,2,S_2^t}+|f_3|_{4/3,2,S_2^t}+|g|_{2,\Om^t}
\label{4.15}
\end{equation}
we obtain
\begin{equation}\eqal{
&\intop_\Om\ro k^2dx\le\phi(\ro_*,\ro^*)(V_1^2(t)+1+A_1^2)\cdot\cr
&\quad\cdot\exp[\phi(\ro_*,\ro^*)(D_1^6(t)+V_1^2(t))][\La_1^2(t)\cr
&\quad+|\ro_{x_3}|_{3,\infty,\Om^t}^2(|v_t|_{2,\Om^t}^2+ |f|_{2,\Om^t}^2)+|\ro_{x_3}|_{3,\infty,\Om^t}^2|v|_{\infty,\infty,\Om^t}^2 A_1^2]\cr
&\quad+\exp[-\nu_1t+\phi(\ro_*,\ro^*)(D_1^6(t)+V_1^2(t))]\intop_\Om \ro(0)k^2(0)dx.\cr}
\label{4.16}
\end{equation}
Integrating (\ref{4.11}) with respect to time and using (\ref{4.16}) yield
\begin{equation}\eqal{
\|k\|_{V(\Om^t)}^2&\le c\phi(\ro_*,\ro^*)(D_1^6+V_1^2)(1+V_1^2+A_1^2)\cdot\cr
&\quad\cdot\{\exp[\phi(\ro_*,\ro^*)(D_1^6+V_1^2)][\La_1^2+ |\ro_{x_3}|_{3,\infty,\Om^t}^2(|v_t|_{2,\Om^t}^2\cr
&\quad+|f|_{2,\Om^t}^2+|v|_{\infty,\infty,\Om^t}^2A_1^2)]\cr
&\quad+\exp[-\nu_1t+\phi(\ro_*,\ro^*)(D_1^6+V_1^2)]\intop_\Om \ro(0)k^2(0)dx\}\cr
&\quad+\phi(\ro_*,\ro^*)(1+V_1^2+A_1^2)[\La_1^2+ |\ro_{x_3}|_{3,\infty,\Om^t}^2(|v_t|_{2,\Om^t}^2+|f|_{2,\Om^t}^2\cr
&\quad+|v|_{\infty,\infty,\Om^t}^2A_1^2)]+\intop_\Om\ro(0)k^2(0)dx.\cr}
\label{4.17}
\end{equation}
Simplifying (\ref{4.17}) implies
\begin{equation}\eqal{
\|k\|_{V(\Om^t)}^2&\le c\phi(\ro_*,\ro^*,D_1,V_1,A_1)[\La_1^2+ |\ro_{x_3}|_{3,\infty,\Om^t}^2(|v_t|_{2,\Om^t}^2\cr
&\quad+|f|_{2,\Om^t}^2+|v|_{\infty,\infty,\Om^t}^2)+|k(0)|_{2,\Om}^2].\cr}
\label{4.18}
\end{equation}
Inequality (\ref{4.18}) implies (\ref{4.7}) and concludes the proof of Lemma \ref{l4.2}.
\end{proof}

\begin{corollary}\label{c4.3}
Since
$$\eqal{
&\|\tilde h\|_{V(\Om^t)}\le c(|d_{x'}|_{2,\infty,S_2^t}+\|d_{x'}\|_{1,2,S_2^t}),\cr
&|\tilde h(0)|_{2,\Om}\le c|d_{x'}(0)|_{2,S_2}\le c|d_{x'}|_{2,\infty,S_2^t}\cr}
$$
we obtain from (\ref{4.7}) that
%\marginpar{brak nawiasu "]"}
\begin{equation}\eqal{
\|h\|_{V(\Om^t)}&\le c\phi(\ro_*,\ro^*,D_1,V_1,A_1)[\La_1+ |\ro_{x_3}|_{3,\infty,\Om^t}(|v_t|_{2,\Om^t}\cr
&\quad+|f|_{2,\Om^t}+|v|_{\infty,\infty,\Om^t})]+|h(0)|_{2,\Om}.\cr}
\label{4.19}
\end{equation}
\end{corollary}

\section{A priori estimates for first derivatives of $\ro$}\label{s5}

Let $x'=(x_1,x_2)$ and $X_r(t)=(|\ro_{x'}(t)|_{r,\Om}^r+|\ro_t(t)|_{r,\Om}^r)^{1/r}$ for $\ro$ - the solution to (\ref{2.1}).

\begin{lemma}\label{l5.1}
Let $X_r(0)<\infty$, $r\in(1,\infty)$ be finite. Assume that $$|\ro_{1,x'}|_{r,S_2^t(-a)}+|\ro_{1,t}|_{r,S_2^t(-a)}<\infty,$$ $v\in L_\infty(\Om^t)$, $v_x,v_t\in L_1(0,t;L_\infty(\Om))$,$d_1\in L_\infty(S_2^t(-a))$, $t\le T$ and there exists a positive constant $d_*$ such that $v_3\ge d_*>0$.\\
Then there exists a positive increasing function $\phi$ such that solutions to (\ref{2.1}) satisfy
\begin{equation}\eqal{
X_r(t)&\le\phi(|v|_{\infty,\Om^t},|v_x|_{\infty,1,\Om^t}, |v_t|_{\infty,1,\Om^t})\cdot\cr
&\quad\cdot[|d_1|_{\infty,S_2^t}(|\ro_{1,x'}|_{r,S_2^t(-a)}+ |\ro_{1,t}|_{r,S_2^t(-a)})+X_r(0)].\cr}
\label{5.1}
\end{equation}
\end{lemma}

\begin{proof}
Differentiate $(\ref{2.1})_1$ with respect to $x_\alpha$, $\alpha=1,2$, multiply by\break $\ro_{x_\alpha}|\ro_{x_\alpha}|^{r-2}$, $r\ge 2$, and integrate over $\Om$. Then we obtain
\begin{equation}
{1\over r}{d\over dt}|\ro_{x_\alpha}|_{r,\Om}^r+{1\over r}\intop_\Om v\cdot\nabla|\ro_{x_\alpha}|^rdx+\intop_\Om v_{x_\alpha}\cdot\nabla\ro\ro_{x_\alpha}|\ro_{x_\alpha}|^{r-2}dx=0.
\label{5.2}
\end{equation}
Using that $v$ is divergence free, the second term in (\ref{5.2}) equals
$$
{1\over r}\intop_\Om\divv(v|\ro_{x_\alpha}|^r)dx=-{1\over r}\intop_{S_2(-a)}d_1|\ro_{1,x_\alpha}|^rdS_2+{1\over r}\intop_{S_2(a)}d_2|\ro_{x_\alpha}|^rdS_2.
$$
The last integral in (\ref{5.2}) has the form
$$
\sum_{\beta=1}^2\intop_\Om v_{\beta,x_\alpha}\ro_{x_\beta}\ro_{x_\alpha}|\ro_{x_\alpha}|^{r-2}dx+ \intop_\Om v_{3,x_\alpha}\ro_{x_3}\ro_{x_\alpha} |\ro_{x_\alpha}|^{r-2}dx.
$$
In view of the above expressions we derive from (\ref{5.2}) the inequality
\begin{equation}\eqal{
&{1\over r}{d\over dt}|\ro_{x_\alpha}|_{r,\Om}^r\le{1\over r}\intop_{S_2(-a)}d_1|\ro_{1,x_\alpha}|^rdS_2+\sum_{\beta=1}^2\intop_\Om |v_{\beta,x_\alpha}|\,|\ro_{x_\beta}|\,|\ro_{x_\alpha}|^{r-1}dx\cr
&\quad+\intop_\Om|v_{3,x_\alpha}|\,|\ro_{x_3}|\,|\ro_{x_\alpha}|^{r-1}dx,\cr}
\label{5.3}
\end{equation}
where $\alpha=1,2$.

Introduce the notation
$$\eqal{
&x'=(x_1,x_2),\ \ \ro_{x'}=(\ro_{x_1},\ro_{x_2}),\ \ |\ro_{x'}|=|\ro_{x_1}|+|\ro_{x_2}|,\ \ v'=(v_1,v_2),\cr
&v'_{x'}=(v_{1,x_1},v_{1,x_2},v_{2,x_1},v_{2,x_2}),\ \ |v'_{x'}|=|v_{1,x_1}|+|v_{1,x_2}|+|v_{2,x_1}|+|v_{2,x_2}|.\cr}
$$
Then (\ref{5.3}) can be written in the form
\begin{equation}\eqal{
&{1\over r}{d\over dt}|\ro_{x'}|_{r,\Om}^r\le{1\over r}\intop_{S_2(-a)} d_1|\ro_{1,x'}|^rdS_2+|v'_{x'}|_{\infty,\Om}|\ro_{x'}|_{r,\Om}^r\cr
&\quad+|v_{3,x'}|_{\infty,\Om} |\ro_{x_3}|_{r,\Om}|\ro_{x'}|_{r,\Om}^{r-1}.\cr}
\label{5.4}
\end{equation}
Differentiate $(\ref{2.1})_1$ with respect to $t$, multiply by $\ro_t|\ro_t|^{r-2}$ and integrate over $\Om$. Then we have
\begin{equation}\eqal{
&{1\over r}{d\over dt}|\ro_t|_{r,\Om}^r-{1\over r}\intop_{S_2(-a)}d_1|\ro_{1,t}|^rdS_2+{1\over r}\intop_{S_2(a)}d_2|\ro_t|^rdS_2\cr
&\quad+\intop_\Om v_t\cdot\nabla\ro\ro_t|\ro_t|^{r-2}dx=0.\cr}
\label{5.5}
\end{equation}
Simplifying, we write (\ref{5.5}) in the form
\begin{equation}\eqal{{1\over r}{d\over dt}|\ro_t|_{r,\Om}^r&\le{1\over r}\intop_{S_2(-a)}d_1|\ro_{1,t}|^rdS_2+ |v'_t|_{\infty,\Om}|\ro_{x'}|_{r,\Om}|\ro_t|_{r,\Om}^{r-1}\cr
&\quad+|v_{3,t}|_{\infty,\Om}|\ro_{x_3}|_{r,\Om}|\ro_t|_{r,\Om}^{r-1}.\cr}
\label{5.6}
\end{equation}
Inequalities (\ref{5.4}) and (\ref{5.6}) imply
$$\eqal{
&{1\over r}{d\over dt}(|\ro_{x'}|_{r,\Om}^r+|\ro_t|_{r,\Om}^r)\le{1\over r}\intop_{S_2(-a)}d_1(|\ro_{1,x'}|^r+|\ro_{1,t}|^r)dS_2\cr
&\quad+|v'_{x'}|_{\infty,\Om}|\ro_{x'}|_{r,\Om}^r+ |v_{3,x'}|_{\infty,\Om}|\ro_{x_3}|_{r,\Om}|\ro_{x'}|_{r,\Om}^{r-1}\cr
&\quad+|v'_t|_{\infty,\Om}|\ro_{x'}|_{r,\Om}|\ro_t|_{r,\Om}^{r-1}+ |v_{3,t}|_{\infty,\Om}|\ro_{x_3}|_{r,\Om}|\ro_t|_{r,\Om}^{r-1}.\cr}
$$
Simplifying, we get
$$\eqal{
&{1\over r}{d\over dt}(|\ro_{x'}|_{r,\Om}^r+|\ro_t|_{r,\Om}^r)\le{1\over r}\intop_{S_2(-a)}d_1(|\ro_{1,x}|^r+|\ro_{1,t}|^r)dS_2\cr
&\quad+|v_{x'}|_{\infty,\Om}\bigg({2r-1\over r}|\ro_{x'}|_{r,\Om}^r+{1\over r}|\ro_{x_3}|_{r,\Om}^r\bigg)\cr
&\quad+|v_t|_{\infty,\Om}\bigg({1\over r}|\ro_{x'}|_{r,\Om}^r+{1\over r}|\ro_{x_3}|_{r,\Om}^r+{2(r-1)\over r}|\ro_t|_{r,\Om}^r\bigg).\cr}
$$
Simplifying again, yields
\begin{equation}\eqal{
&{1\over r}{d\over dt}(|\ro_{x'}|_{r,\Om}^r+|\ro_t|_{r,\Om}^r)\le{1\over r}\intop_{S_2(-a)}d_1(|\ro_{1,x'}|^r+|\ro_{1,t}|^r)dS_2\cr
&\quad+2(|v_{x'}|_{\infty,\Om}+|v_t|_{\infty,\Om})\bigg( |\ro_{x'}|_{r,\Om}^r+{1\over r}|\ro_{x_3}|_{r,\Om}^r+{r-1\over r}|\ro_t|_{r,\Om}^r\bigg).\cr}
\label{5.7}
\end{equation}
We write the equation of continuity in the form
\begin{equation}
\ro_t+v_3\ro_{x_3}+v_\alpha\ro_{x_\alpha}=0,
\label{5.8}
\end{equation}
where $\alpha=1,2$ and the summation convention with respect to repeated $\alpha$ is assumed. We use that
\begin{equation}
v_3\ge d_*>0
\label{5.9}
\end{equation}
holds in whole domain $\Om$ and $d_*$ is a constant.

Using (\ref{5.9}) in (\ref{5.8}) yields
\begin{equation}
\ro_{x_3}=-{1\over v_3}(\ro_t+v_\alpha\ro_{x_\alpha}).
\label{5.10}
\end{equation}
Using (\ref{5.10}) in (\ref{5.7}) implies
\begin{equation}\eqal{
&{1\over r}{d\over dt}(|\ro_{x'}|_{r,\Om}^r+|\ro_t|_{r,\Om}^r)\le {1\over r}\intop_{S_2(-a)}d_1(|\ro_{1,x'}|^r+|\ro_{1,t}|^rdS_2\cr
&\quad+2(|v_{x'}|_{\infty,\Om}+|v_t|_{\infty,\Om})\bigg[ |\ro_{x'}|_{r,\Om}^r+{1\over rd_*^r}|\ro_t+v_\alpha\ro_{x_\alpha}|_{r,\Om}^r+{r-1\over r}|\ro_t|_{r,\Om}^r\bigg].\cr}
\label{5.11}
\end{equation}
Hence, (\ref{5.11}) takes the form
\begin{equation}\eqal{
&{1\over r}{d\over dt}(|\ro_x'|_{r,\Om}^r+|\ro_t|_{r,\Om}^r)\le{1\over r}\intop_{S_2(-a)}d_1(|\ro_{1,x'}|^r+|\ro_{1,t}|^r)dS_2\cr
&\quad+2(|v_x|_{\infty,\Om}+|v_t|_{\infty,\Om})\bigg[\bigg(1+{2^r\over rd_*^r}|v'|_{\infty,\Om}^r\bigg)|\ro_{x'}|_{r,\Om}^r\cr
&\quad+\bigg({2^r\over rd_*^r}+{r-1\over r}\bigg)|\ro_t|_{r,\Om}^r\bigg].\cr}
\label{5.12}
\end{equation}
For any finite $r$, (\ref{5.12}) implies
\begin{equation}\eqal{
{d\over dt}X_r^r&\le\intop_{S_2(-a)}d_1(|\ro_{1,x}|^r+|\ro_{1,t}|^r)dS_2\cr
&\quad+c(|v_x|_{\infty,\Om}+|v_t|_{\infty,\Om})(1+|v'|_{\infty,\Om}^r)X_r^r.\cr}
\label{5.13}
\end{equation}
Integrating with respect to time yields
\begin{equation}\eqal{
X_r^r(t)&\le\exp\bigg[c\intop_0^t(|v_x|_{\infty,\Om}+|v_t|_{\infty,\Om})(1+ |v'|_{\infty,\Om}^r)dt'\bigg]\cdot\cr
&\quad\cdot\bigg[\intop_0^tdt'\intop_{S_2(-a)}d_1(|\ro_{1,x'}|^r+ |\ro_{1,t}|^r)dS_2+X_r^r(0)\bigg].\cr}
\label{5.14}
\end{equation}
The above inequality implies (\ref{5.1}) and concludes the proof.
\end{proof}

\begin{remark}\label{r5.2}
From (\ref{5.10}) we have
\begin{equation}
|\ro_{x_3}|_{r,\infty,\Om^t}\le{1\over d_*}(1+|v'|_{\infty,\Om^t})X_r(t).
\label{5.15}
\end{equation}
\end{remark}

\begin{corollary}\label{c5.3}
Inequalities (\ref{5.1}) and (\ref{5.15}) imply
\begin{equation}\eqal{
&|\ro_x(t)|_{r,\Om}+|\ro_t(t)|_{r,\Om}\le\phi(1/d_*, |v|_{\infty,\Om^t},|v_x|_{\infty,1,\Om^t},|v_t|_{\infty,1,\Om^t})\cdot\cr
&\quad\cdot[|\ro_{1,x'}|_{r,S_2^t(-a)}+|\ro_{1,t}|_{r,S_2^t(-a)}+ |\ro_{x'}(0)|_{r,\Om}+|\ro_t(0)|_{r,\Om}].\cr}
\label{5.16}
\end{equation}
\end{corollary}
\eject

\section{A lower bound for $v_3$}\label{s7}

To prove Lemma \ref{l5.1} we have to know that there exists a positive constant $d_*$ such that
\begin{equation}
v_3(x,t)\ge d_*.
\label{7.1}
\end{equation}
We can expect that (\ref{7.1}) holds because
$$
v_3|_{x_3=-a}=d_1>0,\quad v_3|_{x_3=a}=d_2>0.
$$
In this Section we prove (\ref{7.1}). From (\ref{1.1}) it follows the following problem for $v_3$,
\begin{equation}\eqal{
&\ro v_{3,t}+\ro v\cdot\nabla v_3-\nu\Delta v_3+q=\ro f_3,\cr
&v_3|_{x_3=a_\alpha}=d_\alpha,\ \ \alpha=1,2,\cr
&v_3|_{t=0}=v_3(0),\cr}
\label{7.2}
\end{equation}
where $q=p_{x_3}$.

To find (\ref{7.1}) for solutions to (\ref{7.2}) we need to know that $\ro$ is a solution to the problem
\begin{equation}\eqal{
&\ro_t+v\cdot\nabla\ro=0\quad {\rm in}\  \Om^T, \cr
&\ro=\ro_1, \quad {\rm on}\ S_2^T,\cr
& \ro|_{t=0}=\ro_0.\cr}
\label{7.3}
\end{equation}

\begin{lemma}\label{l7.1}
Assume that $\bar d_0\ge v_3(0)\ge d_0$, where $d_0$, $\bar d_0$ are positive constants and assume that $d_i\ge d_\infty$, $i=1,2$, where $d_\infty$ is also a positive constant.\\
Assume that $p_{x_3},f_3\in L_1(0,t;L_\infty(\Omega))$. Then there exists
$$
d_*=\phi\bigg(\exp\bigg[-{1\over\varrho_*}(|p_{x_3}|_{\infty,1,\Omega^t}+ |f_3|_{\infty,1,\Omega^t})\bigg],{d_{\infty}d_0\over 3\bar d_0+d_{\infty}}\bigg)>0
$$
such that
\begin{equation}
v_{3}\ge d_*.
\label{7.4}
\end{equation}
\end{lemma}

\begin{proof}
Multiply $(\ref{7.2})_1$ by $\displaystyle{{v_3\over|v_3|^{s+1}}}$ and integrate over $\Om$. Then we have
\begin{equation}\eqal{
&\intop_\Om\ro v_{3,t}{v_3\over|v_3|^{s+1}}dx+\intop_\Om\ro v\cdot\nabla v_3{v_3\over|v_3|^{s+1}}dx\cr
&\quad-\nu\intop_\Om\Delta v_3{v_3\over|v_3|^{s+1}}dx+\intop_\Om q{v_3\over|v_3|^{s+1}}dx=\varrho\intop_\Om f_3{v_3\over|v_3|^{s+1}}dx.\cr}
\label{7.5}
\end{equation}
Now, we examine the particular terms in (\ref{7.5}),
$$\eqal{
&v_{3,t}{v_3\over|v_3|^{s+1}}={1\over 2}{\partial_t|v_3|^2\over|v_3|^{s+1}}= \partial_t|v_3|{1\over|v_3|^s}={1\over-s+1}\partial_t|v_3|^{-s+1},\cr
&\nabla v_3{v_3\over|v_3|^{s+1}}={1\over-s+1}\nabla|v_3|^{-s+1}.\cr}
$$
Therefore, the sum of the first two terms in the l.h.s. of (\ref{7.5}) takes the form
$$
I_1={1\over-s+1}\intop_\Om(\ro\partial_t|v_3|^{-s+1}+\ro v\cdot\nabla|v_3|^{-s+1})dx.
$$
Using the equation of continuity
$$
\ro_{t}+\divv(v\ro)=0
$$
in $I_1$ yields
$$
I_1={1\over-s+1}{d\over dt}\intop_\Om\ro|v_3|^{-s+1}dx+{1\over-s+1}\intop_\Om\divv(\ro v|v_3|^{-s+1})dx.
$$
Since $v\cdot\bar n|_{S_1}=0,$ the second term in $I_1$ equals
$$\eqal{
&{1\over-s+1}\intop_{S_2(-a)}\ro v\cdot\bar n|v_3|^{-s+1}dS_2+{1\over-s+1}\intop_{S_2(a)}\ro v\cdot\bar n|v_3|^{-s+1}dS_2\cr
&=-{1\over-s+1}\intop_{S_2(-a)}\ro_1d_1|v_3|^{-s+1}dS_2+{1\over-s+1} \intop_{S_2(a)}\ro d_2|v_3|^{-s+1}dS_2,\cr}
$$
where $d_i\ge d_{\infty}>0$, $i=1,2$.

The third term on the l.h.s. of (\ref{7.5}) takes the form
$$\eqal{
&-\nu\intop_\Om\divv(\nabla v_3v_3|v_3|^{-s-1})dx+\nu\intop_\Om|\nabla v_3|^2|v_3|^{-s-1}dx\cr
&\quad+\nu\intop_\Om v_3\nabla v_3\nabla|v_3|^{-s-1}dx\equiv I_2+I_3+I_4.\cr}
$$
In view of assumptions integral $I_2$ equals
$$\eqal{
I_2&=-\nu\intop_\Om\divv(\nabla v_3v_3^{-s})dx=-{\nu\over-s+1}\intop_\Om \divv(\nabla v_3^{-s+1})dx\cr
&=-{\nu\over-s+1}\intop_S\bar n\cdot\nabla v_3^{-s+1}dS=-\nu\intop_Sv_3^{-s}\bar n\cdot\nabla v_3dS\cr
&=\nu\intop_{S_2(-a)}d_1^{-s}v_{3,x_3}dS_2-\nu\intop_{S_2(a)}d_2^{-s}v_{3,x_3} dS_2-\nu\intop_{S_1}v_3^{-s}\bar n\cdot\nabla v_3dS_1\cr
&\equiv J_1+J_2+J_3.\cr}
$$
Using that $v$ is divergence free
$$\eqal{
J_1&=-\nu\intop_{S_2(-a)}d_1^{-s}v_{\alpha,x_\alpha}dS_2=-\nu\intop_{S_2(-a)} (d_1^{-s}v_\alpha)_{,x_\alpha}dS_2\cr
&\quad-\nu s\intop_{S_2(-a)}d_1^{-s-1}d_{1,x_\alpha}v_\alpha dS_2=-\nu\intop_{\partial S_2(-a)}d_1^{-s}v_\alpha \cdot\bar{n}v_\alpha|_{S_1}dL_1\cr
&\quad-\nu s\intop_{S_2(-a)}d_1^{-s-1}d_{1,x_\alpha}v_\alpha dS_2,\cr}
$$
where the first integral vanishes because $L_1\subset S_1$ and $v\cdot\bar n|_{S_1}=0$.

Similarly,
$$
J_2=\nu s\intop_{S_2(a)}d_2^{-s-1}d_{2,x_\alpha}v_\alpha dS_2.
$$
To examine $J_3$ we recall that condition $(\ref{1.1})_5$ for $\bar\tau_\alpha=\bar e_3$ on $S_1$ has the form
$$
\nu v_{3,n}+\gamma v_3=0.
$$
Then
$$
J_3=\gamma\intop_{S_1}v_3^{-s+1}dS_1.
$$
Summarizing,
$$\eqal{
I_2&=-\nu s\intop_{S_2(-a)}d_1^{-s-1}d_{1,x_\alpha}v_\alpha dS_2+\nu s\intop_{S_2(a)}d_2^{-s-1}d_{2,x_\alpha}v_\alpha dS_2\cr
&\quad+\gamma\intop_{S_1}v_3^{-s+1}dS_1.\cr}
$$
Next, we calculate
$$\eqal{
I_3&=\nu\intop_\Om|\nabla v_3|^2v_3^{-s-1}dx=\nu\intop_\Om|v_3^{-{s\over 2}-1/2}\nabla v_3|^2dx\cr
&={4\nu\over(-s+1)^2}\intop_\Om|\nabla v_3^{-s/2+1/2}|^2dx\cr}
$$
and
$$
I_4=-(s+1)\nu\intop_\Om|\nabla v_3|^2v_3^{-s-1}dx=-{4\nu(s+1)\over(-s+1)^2}\intop_\Om|\nabla v_3^{-s/2+1/2}|^2dx.
$$
Hence,
$$
I_3+I_4=-{4\nu s\over(-s+1)^2}\intop_\Om|\nabla v_3^{-s/2+1/2}|^2dx.
$$
Using the above results in (\ref{7.5}) yields
\begin{equation}\eqal{
&{1\over-s+1}{d\over dt}\intop_\Om\ro v_3^{-s+1}dx-{1\over-s+1}\intop_{S_2(-a)}\ro_1d_1^{-s+2}dS_2\cr
&\quad+{1\over-s+1}\intop_{S_2(a)}\ro d_2^{-s+2}dS_2-\nu s\intop_{S_2(-a)}d_1^{-s-1}d_{1,x_\alpha}v_\alpha dS_2\cr
&\quad+\nu s\intop_{S_2(a)}d_2^{-s-1}d_{2,x_\alpha}v_\alpha dS_2-{4\nu s\over(-s+1)^2}\intop_\Om|\nabla v_3^{-s/2+1/2}|^2dx\cr
&\quad+\intop_\Om qv_3^{-s}dx=\varrho \intop_\Om f_3v_3^{-s}dx.\cr}
\label{7.6}
\end{equation}
In view of the assumptions of this lemma the last term on the l.h.s. of (\ref{7.6}) is bounded by
\begin{equation}
|q|_{\infty,\Om}{1\over\ro_*d_*}\intop_\Om\ro v_3^{-s+1}dx
\label{7.7}
\end{equation}
and the r.h.s. term by
\begin{equation}
|f_3|_{\infty,\Om}{1\over\ro_*d_*}\intop_\Om\ro v_3^{-s+1}dx,
\label{7.8}
\end{equation}
where $d_*=\min_\Omega v_3$. The existence of this quantity is not proved yet. It will be established of the end of this proof.

We introduce the notation
\begin{equation}
X^s=\intop_\Om\ro v_3^{-s+1}dx,
\label{7.9}
\end{equation}
then we multiply (\ref{7.6}) by $-s+1$ and exploit estimates (\ref{7.7}) and (\ref{7.8}) to conclude
\begin{equation}\eqal{
&{d\over dt}X^s-{4\nu s\over-s+1}\intop_\Om|\nabla v_3^{-s/2+1/2}|^2dx\cr
&\le(|q|_{\infty,\Om}+|f_3|_{\infty,\Om})(s-1){1\over\ro_*d_*}X^s+\intop_{S_2(-a)}\ro_1d_1^{-s+2}dS_2\cr
&\quad-\intop_{S_2(a)}\ro d_2^{-s+2}dS_2-\nu s(s-1)\intop_{S_2(-a)}d_1^{-s-1}d_{1,x_\alpha}v_\alpha dS_2\cr
&\quad+\nu s(s-1)\intop_{S_2(a)}d_2^{-s-1}d_{2,x_\alpha}v_\alpha dS_2.\cr}
\label{7.10}
\end{equation}
Let
$$
\alpha(t)=(|q(t)|_{\infty,\Om}+|f_3(t)|_{\infty,\Om}){s-1\over d_*\ro_*}.
$$
Then (\ref{7.10}) implies
\begin{equation}\eqal{
&{d\over dt}\bigg(X^s\exp\bigg(-\intop_0^t\alpha(t')dt'\bigg)\bigg)\cr
&\le\bigg[ \intop_{S_2(-a)}\ro_1d_1^{-s+2}dS_2+\nu s(s-1)\bigg(\intop_{S_2(-a)}d_1^{-s-1}|d_{1,x'}|\,|v'|dS_2\cr
&\quad+\intop_{S_2(a)}d_2^{-s-1}|d_{2,x'}|\,|v'|dS_2\bigg)\bigg]\exp\bigg( -\intop_0^t\alpha(t')dt'\bigg),\cr}
\label{7.11}
\end{equation}
where $d_{x'}=(d_{x_1},d_{x_2})$, $v'=(v_1,v_2)$.

Integrating (\ref{7.11}) with respect to time yields
\begin{equation}\eqal{
X^s&\le\exp\bigg(\intop_0^t\alpha(t')dt'\bigg)\intop_0^t\bigg[\intop_{S_2(-a)} \ro_1d_1^{-s+2}dS_2\cr
&\quad+\nu s(s-1)\bigg(\intop_{S_2(-a)}d_1^{-s-1}|d_{1,x'}|\,|v'|dS_2+\intop_{S_2(a)} d_2^{-s-1}|d_{2,s'}|\,|v'|dS_2\bigg)\bigg]\cr
&\quad\cdot\exp\bigg(-\intop_0^{t'}\alpha(t'')dt''\bigg)dt'+\exp\bigg(\intop_0^t \alpha(t')dt'\bigg)X^s(0).\cr}
\label{7.12}
\end{equation}
Hence
\begin{equation}\eqal{
X&\le\exp\bigg({1\over s}\intop_0^t\alpha(t')dt'\bigg){1\over d_{\infty}}\bigg[\bigg(\intop_{S_2^t(-a)}\ro_1d_1^2dS_2dt'\bigg)^{1/s}\cr
&\quad+(\nu s(s-1))^{1/s}\bigg(\bigg(\intop_{S_2^t(-a)}d_1^{-1}|d_{1,x'}|\, |v'|dS_2dt'\bigg)^{1/s}\cr
&\quad+\bigg(\intop_{S_2^t(a)}d_2^{-1}|d_{2,x'}|\,|v'|dS_2dt' \bigg)^{1/s}\bigg)\bigg]\cr &\quad +\exp\bigg({1\over s}\intop_0^t\alpha(t')dt'\bigg)X(0).\cr}
\label{7.13}
\end{equation}
Passing with $s\to\infty$ implies
$$
\sup_{x\in\Om}\left|{1\over v_3}\right|\le\exp\bigg[{1\over d_*\ro_*}(|q|_{\infty,1,\Om^t}+ |f_3|_{\infty,1,\Om^t})\bigg]\bigg(\frac{3}{d_{\infty}}+\left|{1\over v_3(0)}\right|_{\infty,\Om}\bigg).
$$
Hence, in view of properties of the local solution, we have
\begin{equation}\eqal{
&v_3\ge\inf_{x\in\Omega} v_3(x)\cr
&\ge\exp\bigg(-\bigg[{1\over d_{*\varrho_*}}(|q|_{\infty,1,\Omega^t}+ |f_3|_{\infty,1,\Omega^t})\bigg]\bigg){d_\infty\inf_{x\in\Omega}|v_3(0)|\over 3\inf_{x\in\Omega}|v_3(0)|+d_\infty}\cr
&\ge\exp\bigg(-\bigg[{1\over d_*\varrho_*}(|q|_{\infty,1,\Omega^t}+ |f_3|_{\infty,1,\Omega^t})\bigg]\bigg){d_\infty d_0\over 3 d_0+d_\infty}\cr
&\equiv d_*.\cr}
\label{7.14}
\end{equation}
We introduce the notation
$$\eqal{
&a(t)={1\over\varrho_*}(|q|_{\infty,1,\Omega^t}+|f_3|_{\infty,1,\Omega^t}),\cr
&b={d_0d_{\infty}\over 3d_0+d_{\infty}}.\cr}
$$

It is clear that $a(t)$ and $b$ are positive and $a(t)$ can be estimated by
$$
a(t)\le{1\over\varrho_*}t^{1/\sigma'}(|q|_{\infty,\sigma,\Omega^t}+ |f|_{\infty,\sigma,\Omega^t}),\quad {1\over\sigma}+{1\over\sigma'}=1.
$$
Since $\sigma>1$, $a(t)$ is small for small $t$.

Then (\ref{7.14}) implies the following equation for $d_*$,
\begin{equation}
\exp\bigg(-{1\over d_*}a(t)\bigg)b=d_*.
\label{7.15}
\end{equation}
Therefore
\begin{equation}
\exp(-a(t))=\bigg({d_*\over b}\bigg)^{d_*}
\label{7.16}
\end{equation}
The function
$$
h(x)=\bigg({x\over b}\bigg)^x
$$
equals to 1 for $x=0$ and $x=b$.

 In the interval $(0,b),$ the function $h(x)<1$ and  attains minimum at $x=be^{-b}=b_*.$ Moreover, ${dh\over dx}<0$ for $x\in(0,b_*)$ and ${dh\over dx}>0$ for $x\in(b_*,b)$.

At the minimum
$$
h(be^{-b})=e^{-b^2e^{-b}}\equiv h_*.
$$
Since (\ref{7.16}) holds we have the restriction
$$
e^{-b^2e^{-b}}\le e^{-a(t)}.
$$
Hence
\begin{equation}
e^ba(t)\le b^2
\label{7.17}
\end{equation}
Therefore, (\ref{7.17}) implies that $t$ must be sufficiently small. We have to emphasize that considerations in this Section has been made for a local solution which can be extended in time if the global estimate holds.

Since $h_*<e^{-a(t)}<1,$ there exists a point $x_*\in (0,b)$ such that
$$
\bigg({x_*\over b}\bigg)^{x_*}=e^{-a(t)}.
$$
Hence $x_*=d_*$, which is a solution to (\ref{7.15}).

Moreover, $$d_*=\phi(e^{-a(t)},b) = \phi\bigg(\exp\left[-\frac{|q|_{\infty,1,\Omega^t}+|f_3|_{\infty,1,\Omega^t}}{\varrho_*}\right],
{d_0d_{\infty}\over 3\bar d_0+d_{\infty}}\bigg).$$
Thus, we need to find estimates for $\|q\|_{L_1(0,t;L_\infty(\Om))}= \|p_{x_3}\|_{L_1(0,t;L_\infty(\Om))}$ to conclude the proof and we accomplish this through the global estimate (\ref{1.8}), established in Section~\ref{s11}.
\end{proof}

\section{Estimates for $\chi$ - the third component\\ of vorticity}\label{s8}

We consider system (\ref{2.8}) for function $\chi$, i.e. the third component of vorticity to get the energy type estimate.  For this purpose, we need to derive energy estimate for problem with homogenous Dirichlet boundary conditions on $S_1$. Thus, we introduce the function $\tilde\chi$ as a solution to the problem
\begin{equation}\eqal{
&\tilde\chi_t-\nu\Delta\tilde\chi=0\quad &{\rm in}\ \ \Om^T,\cr
&\tilde\chi=\chi_*\quad &{\rm on}\ \ S_1^T,\cr
&\tilde\chi_{,x_3}=0\quad &{\rm on}\ \ S_2^T,\cr
&\tilde\chi|_{t=0}=\chi(0)\quad &{\rm in}\ \ \Om,\cr}
\label{8.1}
\end{equation}
where $\chi(0)=v_2(0)_{,x_1}-v_1(0)_{,x_2}$ and $\chi_*$ is described by $(\ref{2.8})_2$.

To show the existence of solutions to (\ref{8.1}) and derive appropriate estimates we need the following compatibility conditions
\begin{equation}
\chi(0)|_{S_1}=\chi_*|_{t=0}
\label{8.2}
\end{equation}
and
\begin{equation}
\chi_{*,x_3}=0\quad {\rm on}\ \ \bar S_1\cap\bar S_2.
\label{8.3}
\end{equation}
Calculating (\ref{8.2}) explicitly we have
\begin{equation}\eqal{
&v_{2,x_1}(0)-v_{1,x_2}(0)=v_i(0)(n_{i,x_j}\tau_{1j}+\tau_{1i,x_j}n_j)\cr
&\quad+v(0)\cdot\bar\tau_1(\tau_{12,x_1}- \tau_{11,x_2})+{\gamma\over\nu}v_j(0)\tau_{1j}.\cr}
\label{8.4}
\end{equation}
To satisfy (\ref{8.3}) we differentiate $\chi_*$ with respect to $x_3$. It is possible because $S_1$ is the part of the boundary of $\Om$ which is parallel to the $x_3$-axis. Moreover, vectors $\bar n|_{S_1}$ and $\bar\tau_1|_{S_1}$ do not depend on $x_3$. Therefore, we need to differentiate the components of velocity only. In $\chi_*$ only two-components of velocity $v_1$ and $v_2$ appear. Differentiating them with respect to $x_3$, projecting on $S_2$, and using $(\ref{2.6})_{4,5}$, we obtain the compatibility condition (\ref{8.3}) in the form
\begin{equation}\eqal{
\chi_{*,x_3}|_{\bar S_1\cap\bar S_2}&=-\sum_{i,j=1}^2\bigg[d_{x_i}(n_{i,x_j}\tau_{1j}+\tau_{1i,x_j}n_j)+ {\gamma\over\nu}d_{x_j}\tau_{1j}\cr
&\quad+d_{x_i}\tau_{1i}(\tau_{12,x_1}-\tau_{11,x_2})\bigg]=0.\cr}
\label{8.5}
\end{equation}
Then, we can introduce the new function $\chi'=\chi-\tilde\chi$, which is a solution to the problem
\begin{equation}\eqal{
&\ro(\chi'_t+v\cdot\nabla\chi')-\nu\Delta\chi'=\ro F+\ro h_3\chi-\ro v\cdot\nabla\tilde\chi\cr
&\quad-\ro(h_2v_{3,x_1}-h_1v_{3,x_2})+\ro_{x_1}(v_{2,t}+v\cdot\nabla v_2+f_2)\cr
&\quad-\ro_{x_2}(v_{1,t}+v\cdot\nabla v_1+f_1)\quad &{\rm in}\ \ \Om^T,\cr
&\chi'=0\quad &{\rm on}\ \ S_1^T,\cr
&\chi'_{,x_3}=0\quad &{\rm on}\ \ S_2^T,\cr
&\chi'|_{t=0}=0\quad &{\rm in}\ \ \Om.\cr}
\label{8.6}
\end{equation}
First, we describe properties of solutions to problem (\ref{8.1}).

\begin{lemma}(see Lemma 4.7 from \cite{RZ1})\label{l8.1}
For solutions to problem (\ref{8.1}) we have
\begin{equation}\eqal{
|\tilde\chi|_{3,\infty,\Om^t}&\le c|\chi_*|_{3,\infty,S_1^t}+|\chi(0)|_{3,\Om}\cr
&\le c\|v'\|_{5/6,2,\infty,\Om^t}+|\chi(0)|_{3,\Om},\cr}
\label{8.7}
\end{equation}
\begin{equation}\eqal{
\|\tilde\chi\|_{1,2,\Om^t}&\le c(\|\chi_*\|_{W_2^{1/2,1/4}(S_1^t)}+|\chi(0)|_{2,\Om})\cr
&\le c(\|v'\|_{W_2^{1/2,1/4}(S_1^t)}+|\chi(0)|_{2,\Om})\cr
&\le c(\|v'\|_{W_2^{1,1/2}(\Om^t)}+|\chi(0)|_{2,\Om}).\cr}
\label{8.8}
\end{equation}
\end{lemma}

\begin{proof}
We restrict our considerations to prove (\ref{8.7}) only because a proof of (\ref{8.8}) follows from a potential theory.

We write $\tilde\chi=\chi_1+\chi_2$, where
\begin{equation}
\chi_{1,t}-\nu\Delta\chi_1=0,\quad \chi_1|_{S_1}=\chi_*,\quad \chi_{1,x_3}|_{S_2}=0,\quad \chi_1|_{t=0}=0
\label{8.9}
\end{equation}
and
\begin{equation}
\chi_{2,t}-\nu\Delta\chi_2=0,\quad \chi_2|_{S_1}=0,\quad \chi_{2,x_3}|_{S_2}=0,\quad \chi_1|_{t=0}=\chi(0).
\label{8.10}
\end{equation}
For solutions to (\ref{8.9}) we obtain (\ref{8.7}) and (\ref{8.8}) for $\chi(0)=0$ (see Lemma 4.7 \cite{RZ1}).

Multiply (\ref{8.10}) by $\chi_2$ and integrate over $\Om^t$ yields
$$
|\chi_2|_{2,\infty,\Om^t}^2+\nu|\nabla\chi_2|_{2,\Om}^2=|\chi_2(0)|_{2,\Om}^2.
$$
Multiplying (\ref{8.10}) by $\chi_2^2$ and integrating over $\Om^t$ gives
$$
|\chi_2|_{3,\infty,\Om^t}^3+\nu|\nabla|\chi_2|^{3/2}|_{2,\Om}^2= |\chi_2(0)|_{3,\Om}^3.
$$
Using that $\tilde\chi=\chi_1+\chi_2$ we derive (\ref{8.7}) and (\ref{8.8}). This concludes the proof.
\end{proof}

Consider problem (\ref{2.8})

\begin{lemma}\label{l8.2}
Assume that there exists a local solution to problem (\ref{1.1}) described by Theorem \ref{t1.1}. Assume that there exists positive constants $\ro_*$, $\ro^*$, $D^*$ such that $\ro_*\le\ro\le\ro^*$, $d_i\le \bar{d}, i=1,2$. Assume that $v'\in W_2^{1,1/2}(\Om^t)\cap L_\infty(0,t;W_2^{5/6}(\Om))$, $h\in L_\infty(0,t;L_3(\Om))$, $F=(\rot f)_3 \in L_2(0,t;L_{6/5}(\Om))$, $\chi(0)\in L_3(\Om)$, $\nabla\ro\in L_\infty(0,t;L_r(\Om))$, $(f_1,f_2)\in L_2(0,t;L_{6r/(5r-6)}(\Om))$, $r>3$. Moreover, $v'_t\in L_2(0,t;L_{6r/(5r-6)}(\Om))$, $v\in L_\infty(0,t;L_{3r/(r-3)}(\Om))$ and $t\le T$.\\
Then solutions to (\ref{2.8}) satisfy
\begin{equation}\eqal{
&|\chi|_{2,\infty,\Om^t}^2+\|\chi\|_{1,2,\Om^t}^2\le\phi (1/\ro_*,\ro^*,\bar{d},A_1)[\|v'\|_{W_2^{1,1/2}(\Om^t)}^2\cr
&\quad+\|v'\|_{5/6,2,\infty,\Om^t}^2+A_1^2+|h|_{3,\infty,\Om^t}^2+ |F|_{6/5,2,\Om^t}^2+|\chi(0)|_{3,\Om}^2\cr
&\quad+|\nabla\ro|_{r,\infty,\Om^t}^2(\|v'\|_{5/6,2,\infty,\Om^t}^2+ \|v'\|_{W_2^{1,1/2}(\Om^t)}^2+|\chi(0)|_{3,\Om}^2)\cr
&\quad+|\ro_{x'}|_{r,\infty,\Om^t}^2(|v'_t|_{6r/(5r-6),2,\Om^t}^2+ |v|_{3r/(r-3),\infty,\Om^t}^2\cr
&\quad+|(f_1,f_2)|_{6r/(5r-6),2,\Om^t}^2)].\cr}
\label{8.11}
\end{equation}
\end{lemma}

\begin{proof}
First we examine problem (\ref{8.6}). Multiply $(\ref{8.6})_1$ by $\chi'$, integrate over $\Om$, use boundary conditions and problem (\ref{2.1}). Then we have
\begin{equation}\eqal{
&{1\over 2}{d\over dt}\intop_\Om\ro\chi'^2dx+{1\over 2}\intop_\Om\divv(\ro v\chi'^2)dx+\nu|\nabla\chi'|_{2,\Om}^2\cr
&=\intop_\Om\ro F\chi'dx+\intop_\Om\ro h_3\chi\chi'dx-\intop_\Om\ro v\cdot\nabla\tilde\chi\chi'dx\cr
&\quad-\intop_\Om\ro(h_2v_{3,x_1}-h_1v_{3,x_2})\chi'dx\cr
&\quad+\intop_\Om[\ro_{x_1}(v_{2,t}+v\cdot\nabla v_2+f_2)-\ro_{x_2}(v_{1,t}+v\cdot\nabla v_1+f_1)]\chi'dx.\cr}
\label{8.12}
\end{equation}
The first term on the r.h.s. of (\ref{8.12}) is bounded by
$$
\varepsilon_1|\chi'|_{6,\Om}^2+c(1/\varepsilon_1)(\ro^*)^2|F|_{6/5,\Om}^2,
$$
the second by
$$
\varepsilon_2|\chi'|_{6,\Om}^2+c(1/\varepsilon_2)(\ro^*)^2|h_3|_{3,\Om}^2 |\chi|_{2,\Om}^2.
$$
Integrating by parts in the third term on the r.h.s. of (\ref{8.12}) yields
$$
-\intop_\Om\ro v\cdot\nabla(\tilde\chi\chi')dx+\intop_\Om\ro v\cdot\nabla\chi'\tilde\chi dx\equiv I_1+I_2,
$$
where
$$
|I_2|\le\varepsilon|\nabla\chi'|_{2,\Om}^2+c(1/\varepsilon)(\ro^*)^2 |v|_{6,\Om}^2|\tilde\chi|_{3,\Om}^2.
$$
Integrating by parts in $I_1$ implies
$$
I_1=-\intop_\Om\divv(\ro v\tilde\chi\chi')dx+\intop_\Om\ro\cdot\nabla v\tilde\chi\chi'dx\equiv I_1^1+I_1^2,
$$
where
$$
|I_1^2|\le\varepsilon|\chi'|_{6,\Om}^2+c(1/\varepsilon)|\nabla\ro|_{r,\Om}^2 |v|_{6,\Om}^2|\tilde\chi|_{3r/(2r-3),\Om}^2.
$$
Next, by the Green theorem, we have
$$
I_1^1=\intop_{S_2(-a)}\ro_1d_1\tilde\chi\chi'dS_2-\intop_{S_2(a)}\ro d_2\tilde\chi\chi'dS_2.
$$
Hence,
$$
|I_1^1|\le 2\ro^*\bar{d}\intop_{S_2}|\tilde\chi\chi'|dS_2\le\varepsilon |\chi'|_{4,S_2}^2+c(1/\varepsilon,\ro^*,\bar{d})|\tilde\chi|_{4/3,S_2}^2.
$$
Summarizing, the third term on the r.h.s. of (\ref{8.12}) is bounded by
$$\eqal{
&\varepsilon_3\|\chi'\|_{1,\Om}^2+c(1/\varepsilon_3,\ro^*,\bar{d})[|v|_{6,\Om}^2 |\tilde\chi|_{3,\Om}^2\cr
&\quad+|\nabla\ro|_{r,\Om}^2(|v|_{6,\Om}^2|\tilde\chi|_{3r/(2r-3),\Om}^2+ |\tilde\chi|_{4/3,S_2}^2)].\cr}
$$
We estimate the fourth term on the r.h.s. of (\ref{8.12}) by
$$
\varepsilon_4|\chi'|_{6,\Om}^2+c(1/\varepsilon_4,\ro^*)|h|_{3,\Om}^2 |v_{3,x'}|_{2,\Om}^2.
$$
Finally, the last term on the r.h.s. of (\ref{8.12}) is bounded by
$$\eqal{
&|\ro_{x'}|_{\lambda_1,\Om}(|v'_t|_{\lambda_2,\Om}+ |v|_{\lambda_2\mu_1,\Om}|\nabla v'|_{\lambda_2\mu_2,\Om}+|(f_1,f_2)|_{\lambda_2,\Om})|\chi'|_{6,\Om}\cr
&\le\varepsilon_5|\chi'|_{6,\Om}^2+c(1/\varepsilon_5) |\ro_{x'}|_{\lambda_1,\Om}^2(|v'_t|_{\lambda_2,\Om}^2+ |v|_{\lambda_2\mu_1,\Om}^2|\nabla v'|_{\lambda_2\mu_2,\Om}^2\cr
&\quad+|(f_1,f_2)|_{\lambda_2,\Om}^2),\cr}
$$
where $1/\lambda_1+1/\lambda_2+1/6=1$, $1/\mu_1+1/\mu_2=1$.

Let $\lambda_1=r$. Then $\lambda_2={6r\over 5r-6}$. We also need that $\lambda_2\mu_2=2$ so ${6r\over 5r-6}\mu_2=2$ and $\mu_2={5r-6\over 3r}$. Then $\mu_1={5r-6\over 2r-6}$ and $\lambda_2\mu_1={3r\over r-3}$.

The middle term on the l.h.s. of (\ref{8.12}) equals
$$\eqal{
&\intop_\Om\divv(\ro v\chi'^2)dx=\intop_{S_2(-a)}\ro_1v\cdot\bar n\chi'^2dS_2+\intop_{S_2(a)}\ro v\cdot\bar n\chi'^2dS_2\cr
&=-\intop_{S_2(-a)}\ro_1d_1\chi'^2dS_2+\intop_{S_2(a)}\ro d_2\chi'^2dS_2.\cr}
$$
Using the above estimates in (\ref{8.12}) and assuming that $\varepsilon_1-\varepsilon_5$ are sufficiently small we obtain the inequality
\begin{equation}\eqal{
&{d\over dt}\intop_\Om\ro\chi'^2dx+\nu\|\chi'\|_{1,\Om}^2\le \intop_{S_2(-a)}\ro_1d_1\chi'^2dS_2\cr
&\quad+c(\ro^*)^2[|F|_{6/5,\Om}^2+|h|_{3,\Om}^2(|\chi|_{2,\Om}^2+ |v_{3,x'}|_{3,\Om}^2)]\cr
&\quad+\phi(\ro^*,\bar{d})[|v|_{6,\Om}^2|\tilde\chi|_{3,\Om}^2+ |\nabla\ro|_{r,\Om}^2(|v|_{6,\Om}^2|\tilde\chi|_{3r/(2r-3),\Om}^2+ |\tilde\chi|_{4/3,S_2})]\cr
&\quad+c|\ro_{x'}|_{r,\Om}^2[|v'_t|_{6r/(5r-6),\Om}^2+ |v|_{3r/(r-3),\Om}|\nabla v'|_{2,\Om}^2+|(f_1,f_2)|_{6r/(5r-6),\Om}^2],\cr}
\label{8.13}
\end{equation}
where $r>3$.

Integrating (\ref{8.13}) with respect to time and using Theorem \ref{t3.4} yield
\begin{equation}\eqal{
&\intop_\Om\ro\chi^2dx+\nu\|\chi\|_{1,2,\Om^t}^2\le\phi(\ro^*,\bar{d},A_1) [|\chi'|_{2,S_2^t}^2\cr
&\quad+|\tilde\chi|_{2,\infty,\Om^t}^2+\|\tilde\chi\|_{1,2,\Om^t}^2+ |h|_{3,\infty,\Om^t}^2+|F|_{6/5,2,\Om^t}^2\cr
&\quad+|\tilde\chi|_{3,\infty,\Om^t}^2+|\nabla\ro|_{r,\infty,\Om^t}^2 (|\tilde\chi|_{3r/(2r-3),\infty,\Om^t}^2+|\tilde\chi|_{4/3,2,S_2^t}^2)\cr
&\quad+c|\ro_{x'}|_{r,\infty,\Om^t}^2(|v'_t|_{6r/(5r-6),2,\Om^t}^2+ |v|_{3r/(r-3),\infty,\Om^t}^2\cr
&\quad+|(f_1,f_2)|_{6r/(5r-6),2,\Om^t}^2)],\cr}
\label{8.14}
\end{equation}
where $r>3$ and $A_1\ge 1$.

Using that $\chi'=\chi-\tilde\chi$ we have
\begin{equation}\eqal{
|\chi'|_{2,S_2^t}^2&=|\chi-\tilde\chi|_{2,S_2^t}^2\le|\chi|_{2,S_2^t}^2+ |\tilde\chi|_{2,S_2^t}^2\cr
&\le\varepsilon|\nabla\chi|_{2,\Om^t}^2+c(1/\varepsilon)A_1^2+ \|\tilde\chi\|_{1,2,\Om^t}^2.\cr}
\label{8.15}
\end{equation}
Moreover, Lemma \ref{l8.1} implies
\begin{equation}\eqal{
&|\tilde\chi|_{2,S_2^t}^2+|\tilde\chi|_{4/3,2,S_2^t}^2\le c(\|v'\|_{W_2^{1,1/2}(\Om^t)}+|\chi(0)|_{2,\Om}^2),\cr
&|\tilde\chi|_{2,\infty,\Om^t}^2+|\tilde\chi|_{3,\infty,\Om^t}^2+ |\tilde\chi|_{3r/(2r-3),\infty,\Om^t}^2\cr
&\le c|\tilde\chi|_{3,\infty,\Om^t}^2\le c(\|v'\|_{5/6,2,\infty,\Om^t}^2+|\chi(0)|_{3,\Om}^2).\cr}
\label{8.16}
\end{equation}
Using (\ref{8.15}) and (\ref{8.16}) in the r.h.s. of (\ref{8.14}) implies
\begin{equation}\eqal{
&|\chi|_{2,\infty,\Om^t}^2+\|\chi\|_{1,2,\Om^t}^2\le\phi (1/\ro_*,\ro^*,\bar{d},A_1)\cdot\cr
&\quad\cdot[\|v'\|_{W_2^{1,1/2}(\Om^t)}^2+\|v'\|_{5/6,2,\infty,\Om^t}^2+ A_1^2+|h|_{3,\infty,\Om^t}^2\cr
&\quad+|F|_{6/5,2,\Om^t}^2+|\chi(0)|_{3,\Om}^2\cr
&\quad+|\nabla\ro|_{r,\infty,\Om^t}^2(\|v'\|_{5/6,2,\infty,\Om^t}^2+ \|v'\|_{W_2^{1,1/2}(\Om^t)}^2+|\chi(0)|_{3,\Om}^2)\cr
&\quad+|\ro_{x'}|_{r,\infty,\Om^t}^2(|v'_t|_{6r/(5r-6),2,\Om^t}^2+ |v|_{3r/(r-3),\infty,\Om^t}^2+|(f_1,f_2)|_{6r/(5r-6),2,\Om^t}^2)],\cr}
\label{8.17}
\end{equation}
where $r>3$. The above inequality implies (\ref{8.11}) and concludes the proof.
\end{proof}
\goodbreak

\section{Auxiliary results}\label{s9}

In this Section, we establish some relation for norms of $h,$ $\chi$ and $\ro$ that are useful for considerations of Sections~\ref{s10} and \ref{s11}.

First we recall positive constants $\ro_*$, $\ro^*$, $d_*$, $\bar{d}$ such that
\begin{equation}
\ro_*\le\ro\le\ro^*,\quad d_*\le d_i\le \bar{d},\ \ i=1,2.
\label{9.1}
\end{equation}
Let
\begin{equation}
D_1(t)=|d_1|_{3,2,S_2^t}.
\label{9.2}
\end{equation}
The following quantity is defined in (\ref{4.15})
\begin{equation}
\La_1=\|d_{x'}\|_{1,3,2,S_2^t}+|d_{x'}|_{2,\infty,S_2^t}+ \|d_t\|_{1,2,S_2^t}+|f_3|_{4/3,2,S_2^t}+|g|_{2,\Om^t}.
\label{9.3}
\end{equation}
In (\ref{Lambda}), we have been introduced the second quantity assumed to be small
\begin{equation}
\La_2=|\ro_{1,x'}|_{r,S_2^t}+|\ro_{1,t}|_{r,S_2^t}+ |\ro_{0,x}|_{r,\Om}.
\label{9.4}
\end{equation}
Inequality (\ref{5.14}) has the form
\begin{equation}
|\ro_{x'}|_{r,\infty,\Om^t}+|\ro_t|_{r,\infty,\Om^t}\le\Phi_1\cdot \phi_1\cdot\La_2,
\label{9.5}
\end{equation}
where $\Phi_1$ and $\phi_1$ have the forms
\begin{equation}\eqal{
&\Phi_1=(1+|d_1|_{\infty,S_2^t(a_1)})(1+|v_0|_{\infty,\Om}),\cr
&\phi_1=\exp\bigg[{c\over r}(|v_x|_{\infty,1,\Om^t}+|v_t|_{\infty,1,\Om^t})(1+|v'|_{\infty,\Om^t}) \bigg].\cr}
\label{9.6}
\end{equation}
From (\ref{5.15}) it follows that
\begin{equation}\eqal{
|\ro_{x_3}|_{r,\infty,\Om^t}&\le{1\over d_*}(1+|v'|_{\infty,\Om^t})\cdot\Phi_1\cdot\phi_1\cdot\La_2\cr
&\equiv\Phi_1\cdot\phi_2\cdot\La_2,\cr}
\label{9.7}
\end{equation}
where
$$
\phi_2={1\over d_*}(1+|v'|_{\infty,\Om^t})\phi_1.
$$

\begin{lemma}\label{l9.1}
Let $v\in W_\sigma^{2+s,1+s/2}(\Om^t)$, $s\in(0,1)$, $\sigma>3/s$, $t\le T$.\\
Then there exist increasing positive functions
$$\eqal{
&\Phi_2=\Phi_2(\ro_*,\ro^*,A_1,D_1(t),|d_1|_{\infty,S_2^t(a_1)}, |v_0|_{\infty,\Om},|f|_{2,\Om^t}),\cr
&\phi_3=\phi_3(t^a\|v\|_{W_\sigma^{2+s,1+s/2}(\Om^t)}),\cr}
$$
where $a>0$ such that
\begin{equation}
\|h\|_{V(\Om^t)}\le\Phi_2\cdot\phi_3\cdot(\La_1+\La_2+|h(0)|_{2,\Om})
\label{9.8}
\end{equation}
and
\begin{equation}
|\ro_x|_{r,\infty,\Om^t}+|\ro_t|_{r,\infty,\Om^t}\le\Phi_1\cdot\phi_2 \cdot\La_2,
\label{9.9}
\end{equation}
where $t\le T$.
\end{lemma}

\begin{proof}
Inequality (\ref{4.19}) implies
\begin{equation}\eqal{
\|h\|_{V(\Om^t)}&\le\Phi'_1(\ro_*,\ro^*,A_1,D_1)\cdot\Phi'_1 (\|v\|_{1,3,2,\Om^t})\cdot\cr
&\quad\cdot[\La_1+|\ro_{x_3}|_{3,\infty,\Om^t}\cdot(|v_t|_{2,\Om^t}+ |v|_{\infty,\infty,\Om^t}+|f|_{2,\Om^t})\cr
&\quad+|h(0)|_{2,\Om}].\cr}
\label{9.10}
\end{equation}
Using (\ref{9.7}) yields
\begin{equation}\eqal{
\|h\|_{V(\Om^t)}&\le\Phi'_1[\La_1+\Phi_1\cdot\phi_2\cdot(|v_t|_{2,\Om^t}+ |v|_{\infty,\infty,\Om^t}\cr
&\quad+|f|_{2,\Om^t})\La_2+|h(0)|_{2,\Om}].\cr}
\label{9.11}
\end{equation}
Consider the imbeddings
\begin{equation}\eqal{
&|v|_{3,2,\Om^t}\le c t^{\frac{\si_1-2}{2\si_1}}\|v\|_{W_{\sigma_1}^{2,1}(\Om^t)},\ \ &\sigma_1\ge 2,\cr
&|v|_{\infty,\Om^t}\le c\|v\|_{W_{\sigma_2}^{2,1}(\Om^t)},\ \ &\sigma_2>5/2,\cr
&|v_x|_{\infty,1,\Om^t}\le t^{1-1/\sigma_3}|v_x|_{\infty,\sigma_3,\Om^t}\cr
&\le t^{1-1/\sigma_3}\|v\|_{W_{\sigma_3}^{2,1}(\Om^t)},\ \ &\sigma_3>3,\cr
&|v_t|_{\infty,1,\Om^t}\le t^{1-1/\sigma_4}|v_t|_{\infty,\sigma_4,\Om^t}\cr
&\le t^{1-1/\sigma_4}\|v\|_{W_{\sigma_4}^{2+s,1+s/2}(\Om^t)},\ \ &\sigma_4>3/s,\ \ s\in(0,1).\cr}
\label{9.12}
\end{equation}
In view of the above imbeddings and the properties of $\phi_2$ there exists function $\phi_3$ such that
\begin{equation}
\phi_2\cdot(|v_t|_{2,\Om^t}+|v|_{\infty,\infty,\Om^t})\le\phi_3 (\|v\|_{W_\sigma^{2+s,1+s/2}(\Om^t)}),
\label{9.13}
\end{equation}
where $\sigma>3/s$, $s\in(0,1)$.

Moreover,
\begin{equation}
\Phi'_1\Phi_1\le\Phi_2.
\label{9.14}
\end{equation}
Hence (\ref{9.8}) is proved. From (\ref{9.5}), (\ref{9.7}) and the above considerations we prove (\ref{9.9}). This concludes the proof.
\end{proof}

\begin{lemma}\label{l9.2}
Let $v\in W_\sigma^{2+s,1+s/2}(\Om^t)$, $s\in(0,1)$, $\sigma>3/s$, $t\le T$. Let
\begin{equation}\eqal{
&D_2=A_1+|F|_{6/5,2,\Om^t}+|\chi(0)|_{3,\Om},\cr
&D_3=|(f_1,f_2)|_{6r/(5r-6),2,\Om^t}+|\chi(0)|_{3,\Om}.\cr}
\label{9.15}
\end{equation}
Then
\begin{equation}\eqal{
\|\chi\|_{V(\Om^t)}&\le\Phi_2\cdot[\|v'\|_{W_2^{1,1/2}(\Om^t)}+ \|v'\|_{5/6,2,\infty,\Om^t}\cr
&\quad+|h|_{3,\infty,\Om^t}+D_2+(\phi_3+D_3\phi_2)\La_2],\cr}
\label{9.16}
\end{equation}
where $v'=(v_1,v_2)$, $\phi_3$ is defined in (\ref{9.20}), $\Phi_2$ is (\ref{9.21}) and $\phi_2$ appears in (\ref{9.9}).
\end{lemma}

\begin{proof}
Using (\ref{9.9}) in (\ref{8.11}) yields
\begin{equation}\eqal{
\|\chi\|_{V(\Om^t)}&\le\Phi(1/\ro_*,\ro^*,\bar{d},A_1)\cdot [\|v'\|_{W_2^{1,1/2}(\Om^t)}\cr
&\quad+\|v'\|_{5/6,2,\infty,\Om^t}+|h|_{3,\infty,\Om^t}+D_2\cr
&\quad+\Phi_1\cdot\phi_2\La_2\cdot(\|v'\|_{W_2^{1,1/2}(\Om^t)}+ \|v'\|_{5/6,2,\infty,\Om^t}\cr
&\quad+|v'_t|_{6r/(5r-6),2,\Om^t}+|v|_{3r/(r-3),\infty,\Om^t}+D_3)].\cr}
\label{9.17}
\end{equation}
The following imbeddings hold
\begin{equation}\eqal{
&\|v\|_{5/6,2,\infty,\Om^t}+\|v\|_{W_2^{1,1/2}(\Om^t)}+ |v|_{3r/(r-3),\infty,\Om^t}\cr
&\le ct^a\|v\|_{W_\sigma^{2,1}(\Om^t)},\cr}
\label{9.18}
\end{equation}
where $a>0$, $\sigma>3$, and
\begin{equation}
\|v_t\|_{6r/(5r-6),2,\Om^t}\le t^{1/2-1/\sigma_4}\|v\|_{W_{\sigma_4}^{2+s,1+s/2}(\Om^t)},
\label{9.19}
\end{equation}
where $\sigma_4\ge 3/s$, $s\in(0,1)$. In view of (\ref{9.18}) and (\ref{9.19}) there exists an increasing positive function $\phi_3$ defined by
\begin{equation}\eqal{
&\phi_2\cdot(\|v'\|_{5/6,2,\infty,\Om^t}+]\|v'\|_{W_2^{1,1/2}(\Om^t)}+ |v'_t|_{6r/(5r-6),2,\Om^t}\cr
&\quad+|v|_{3r/(r-3),\infty,\Om^t}\le c\phi_2\cdot t^a\|v\|_{W_\sigma^{2+s,1+s/2}(\Om^t)}\cr
&\le\phi_3(t^a\|v\|_{W_\sigma^{2+s,1+s/2}(\Om^t)}),\ \ \sigma>3/s,\ \ s\in(0,1).\cr}
\label{9.20}
\end{equation}
Using that
\begin{equation}
\Phi\Phi_1\le\Phi_2(1/\ro_*,\ro^*,d_*,A_1),\ \ \Phi\sim\Phi',
\label{9.21}
\end{equation}
we obtain from (\ref{9.17}) inequality (\ref{9.16}). This concludes the proof.
\end{proof}

\begin{lemma}\label{l9.3}
Assume that the r.h.s. of (\ref{9.9}) is finite. Let $\ro_*\le\ro\le\ro^*$.\\
Then
\begin{equation}
\|\ro\|_{C^{\alpha}(\Om^t)}\le c\|\ro\|_{W_{r,\infty}^{1,1}(\Om^t)}\le\Phi_1\cdot\phi_2\cdot\La_2 (t,r)+c\ro^*,
\label{9.22}
\end{equation}
where the above imbedding holds because
$$
{3\over r}+\alpha<1,\ \ r>3
$$
and
$$
\|\ro\|_{W_{r,\infty}^{1,1}(\Om^t)}=|\ro|_{r,\infty,\Om^t}+ |\ro_x|_{r,\infty,\Om^t}+|\ro_t|_{r,\infty,\Om^t}.
$$

\end{lemma}

\section{Increasing regularity for velocity and pressure}\label{s10}

In this Section we increase regularity of $v$ and $p$ step by step, using mainly results on the Stokes system proven in Appendix (Lemmas~\ref{l10.2}-\ref{l10.5}). First, we make use of rot-div problem (\ref{1.15}) to conclude the estimate for $v'$ in higher norms.

\begin{lemma}\label{l10.1}
Assume $D_2$, $D_3$ are defined in (\ref{9.15}), $\phi_3$ in (\ref{9.20}) and $\Phi_2$ in (\ref{9.21}). Moreover, $\La_1$ is introduced in (\ref{9.3}) and $\La_2$ in (\ref{9.4}). Finally, it is assumed that $h\in L_\infty(0,t;L_3(\Om))$, $h(0)=v_{0,x_3}\in L_2(\Om)$, $v'\in L_2(\Om;H^{1/2}(0,t))$, where $v'=(v_1,v_2)$, $t\le T$.\\
Then
\begin{equation}\eqal{
\|v'\|_{V^1(\Om^t)}&\le\Phi_2[\|v'\|_{L_2(\Om;H^{1/2}(0,t))}\cr
&\quad+|h|_{3,\infty,\Om^t}+D_2+(1+D_3)\phi_3(\La_1+\La_2+ \|h(0)\|_{L_2(\Om)})].\cr}
\label{10.1}
\end{equation}
\end{lemma}

\begin{proof}
Let $\Om'$ be the cross-section of $\Om$ with the plane perpendicular to the $x_3$-axis and passing through the point $x_3\in(-a,a)$. Let $S'_1$ be the cross-section of $S_1$ with the same plane. Then $S'_1$ is the boundary of $\Om'$. Therefore, the elliptic system (rot,div) reduced to $\Om'$ yields the problem
\begin{equation}\eqal{
&v_{2,x_1}- v_{1,x_2}=\chi\quad &{\rm in}\ \ \Om',\cr
&v_{1,x_1}+v_{2,x_2}=-h_3\quad &{\rm in}\ \ \Om',\cr
&v'\cdot\bar n'=0\quad &{\rm on}\ \ S'_1,\cr}
\label{10.2}
\end{equation}
where $x_3$ is treated as parameter, $v'=(v_1,v_2)$ and $\bar n'$ is the unit outward vector normal to $S_1$ at a point of $S'_1$.

Solutions to (\ref{10.2}) satisfy the estimate
\begin{equation}\eqal{
&\sup_t\|v'\|_{L_2(-a,a;H^1(\Om'))}+\|v'\|_{L_2(0,t;L_2(-a,a;H^2(\Om'))}\cr
&\le c(\|\chi\|_{V(\Om^t)}+\|h_3\|_{V(\Om^t)}).\cr}
\label{10.3}
\end{equation}
From (\ref{9.8}) it follows
\begin{equation}
\|v'_{x_3}\|_{V(\Om^t)}\le\Phi_2\cdot\phi_3\cdot(\La_1+ \La_2+\|h(0)\|_{L_2(\Om)}).
\label{10.4}
\end{equation}
Hence, (\ref{10.3}) and (\ref{10.4}) imply
\begin{equation}
\|v'\|_{V^1(\Om^t)}\le c[\|\chi\|_{V(\Om^t)}+\Phi_2\cdot\phi_3\cdot(\La_1+ \La_2+\|h(0)\|_{L_2(\Om)})].
\label{10.5}
\end{equation}
Using (\ref{9.16}) yields
\begin{equation}\eqal{
\|v'\|_{V^1(\Om^t)}&\le\Phi_2\cdot[\|v'\|_{W_2^{1,1/2}(\Om^t)}+ \|v'\|_{6/5,2,\infty,\Om^t}\cr
&\quad+|h|_{3,\infty,\Om^t}+D_2+(\phi_3+D_3\phi_2)\La_2+\phi_3\cdot(\La_1\cr
&\quad+\La_2+\|h(0)\|_{L_2(\Om)})].\cr}
\label{10.6}
\end{equation}
Using that
$$\eqal{
\|v'\|_{W_2^{1,1/2}(\Om^t)}&\le\|v'\|_{L_2(0,t;H^1(\Om))}+ \|v'\|_{L_2(\Om;H^{1/2}(0,t))}\cr
&\le A_1+\|v'\|_{L_2(\Om;H^{1/2}(0,t))}\cr}
$$
and the interpolation
$$\eqal{
\|v'\|_{5/6,2,\infty,\Om^t}&\le\varepsilon\|v'\|_{1,2,\infty,\Om^t}+ c(1/\varepsilon)|v'|_{2,\Om^t}\cr
&\le\varepsilon\|v'\|_{1,2,\infty,\Om^t}+c(1/\varepsilon)A_1\cr}
$$
in (\ref{10.6}) implies the inequality for sufficiently small $\varepsilon$
\begin{equation}\eqal{
\|v'\|_{V^1(\Om^t)}&\le\Phi_2\cdot[\|v'\|_{L_2(\Om;H^{1/2}(0,t))}
+|h|_{3,\infty,\Om^t} \cr &\quad +D_2+(\phi_3+D_3\phi_2)\La_2+\phi_3(\La_1\cr
&\quad+\La_2+\|h(0)\|_{L_2(\Om)})].\cr}
\label{10.7}
\end{equation}
Since
$$
\phi_2\le\phi_3
$$
we obtain from (\ref{10.7}) inequality (\ref{10.1}). This concludes the proof.
\end{proof}

\begin{lemma}\label{l10.2}
Let the assumptions of Lemma \ref{l10.1} hold.\\
Let $v\in W_\sigma^{2+s,1+s/2}(\Om^t)$, $p\in L_{5/3}(\Om^t)$, $\si > 3/5,$ $\bar{a} > 0,$
\begin{equation}
H=|h|_{3,\infty,\Om^t}+|h|_{10/3,\Om^t}+\|h(0)\|_{L_2(\Om)},
\label{10.8}
\end{equation}
\begin{equation}
D_4=|f|_{5/3,\Om^t}+\|d\|_{W_{5/3}^{7/5,7/10}(S_2^t)}+ \|v_0\|_{W_{5/3}^{4/5}(\Om)}.
\label{10.9}
\end{equation}
Let $D_2$, $D_3$ be defined in (\ref{9.15}) and
\begin{equation}
\phi_5=\phi_5(\ro^*,A_1,t^{\bar{a}}\|v\|_{W_\sigma^{2+s,1+s/2}(\Om^t)}),
\label{10.10}
\end{equation}
where $\phi_5$ is combination of $\phi_2$, $\phi_3$ and function $\phi$ from (A.1). Moreover, $\phi_5(\ro^*,A_1,0)=0$.\\
Then
\begin{equation}\eqal{
&\|v\|_{W_{5/3}^{2,1}(\Om^t)}+\|\nabla p\|_{L_{5/3}(\Om^t)}\le\phi_5(\ro^*,A_1, t^{\bar{a}}\|v\|_{W_\sigma^{2+s,1+s/2}(\Om^t)})\cdot\cr
&\quad\cdot(\La_1+\La_2)\cdot[\|v\|_{W_{5/3}^{1,1/2}(\Om^t)}+ |p|_{5/3,\Om^t}+D_3]+c(\ro^*,A_1)\cdot\cr
&\quad\cdot[H+D_2+D_4+\phi_3\|h(0)\|_{L_2(\Om)}]\cr}
\label{10.11}
\end{equation}
and
\begin{equation}\eqal{
\|v'\|_{V^1(\Om^t)}&\le\phi_5(\ro^*,A_1, t^{\bar{a}}\|v\|_{W_\sigma^{2+s,1+s/2}(\Om^t)})\cdot(\La_1+\La_2)\cdot\cr
&\quad\cdot[\|v\|_{W_{5/3}^{1,1/2}(\Om^t)}+|p|_{5/3,\Om^t}+D_3]+c (\ro^*,A_1)[H+D_2+D_4\cr
&\quad+\phi_3\|h(0)\|_{L_2(\Om)}].\cr}
\label{10.12}
\end{equation}
\end{lemma}

\begin{proof}
Consider problem (\ref{1.1}) written in the form of the two following problems
\begin{equation}\eqal{
&\ro v_t-\divv\T(v,p)=-\ro v'\cdot\nabla v-\ro v_3h+\ro f\quad &{\rm in}\ \ \Om^T,\cr
&\divv v=0\quad &{\rm in}\ \ \Om^T,\cr
&v\cdot\bar n=0,\ \ \nu\bar n\cdot\D(v)\cdot\bar\tau_\alpha+\gamma v\cdot\bar\tau_\alpha=0,\ \ \alpha=1,2\quad &{\rm on}\ \ S_1^T,\cr
&v\cdot\bar n=d,\ \ \bar n\cdot\D(v)\cdot\bar\tau_\alpha=0,\ \ \alpha=1,2\quad &{\rm on}\ \ S_2^T,\cr
&v|_{t=0}=v_0\quad &{\rm in}\ \ \Om,\cr}
\label{10.13}
\end{equation}
where $\ro$ is treated as given and $\ro$ satisfies
\begin{equation}\eqal{
&\ro_t+v\cdot\nabla\ro=0\quad &{\rm in}\ \Om^T,\cr
&\divv v=0 \quad &{\rm in}\ \  \Om^T, \cr
&\ro=\ro_1\quad &{\rm on}\ \ S_2^T(a_1),\cr
&\ro|_{t=0}=\ro_0\quad &{\rm in}\ \ \Om,\cr}
\label{10.14}
\end{equation}
where $v$ is treated as given.

From the proof of Lemma 3.4 in \cite {Z4} we have
\begin{equation}
\|v'\|_{L_{10}(\Om^T)}\le c\|v'\|_{V^1(\Om^T)}.
\label{10.15}
\end{equation}
Then
\begin{equation}\eqal{
&|v'\cdot\nabla v|_{5/3,\Om^t}\le cA_1\|v'\|_{V^1(\Om^t)},\cr
&|v_3h|_{5/3,\Om^t}\le cA_1|h|_{10/3,\Om^t}.\cr}
\label{10.16}
\end{equation}
In view of (\ref{10.16}) and Lemma \ref{A.1} we have
\begin{equation}\eqal{
&\|v\|_{W_{5/3}^{2,1}(\Om^t)}+\|\nabla p\|_{L_{5/3}(\Om^t)}\le\phi_4(\ro^*,t^{\bar{a}}\|v\|_{W_\sigma^{2+s,1+s/2}(\Om^t)}) \La_2\cdot\cr
&\quad\cdot[\|v\|_{W_{5/3}^{1,1/2}(\Om^t)}+|p|_{5/3,\Om^t}]+c(\ro^*) [A_1\|v'\|_{V_2^1(\Om^t)}\cr
&\quad+A_1|h|_{10/3,\Om^t}+D_4].\cr}
\label{10.17}
\end{equation}
Using (\ref{10.7}) and the interpolation inequality
\begin{equation}\eqal{
\|v'\|_{L_2(\Om;H^{1/2}(0,t))}&\le\varepsilon\|v'\|_{W_{5/3}^{2,1}(\Om^t)}+ c(1/\varepsilon)|v'|_{2,\Om^t}\cr
&\le\varepsilon\|v'\|_{W_{5/3}^{2,1}(\Om^t)}+c(1/\varepsilon)A_1\cr}
\label{10.18}
\end{equation}
we obtain
\begin{equation}\eqal{
&\|v\|_{W_{5/3}^{2,1}(\Om^t)}+\|\nabla p\|_{L_{5/3}(\Om^t)}\le\phi_4(\ro^*,t^{\bar{a}}\|v\|_{W_\sigma^{2+s,1+s/2}(\Om^t)})\cdot\cr
&\quad\cdot\La_2[\|v\|_{W_{5/3}^{1,1/2}(\Om^t)}+|p|_{{5\over 3},\Om^t}]+c(\ro^*,A_1)\cdot\cr
&\quad\cdot[H+D_2+D_4+(2\phi_3+D_3\phi_2)\La_2+\phi_3(\La_1+ \|h(0)\|_{L_2(\Om)})]\cr
&\le\phi_5(\ro^*,A_1,\|v\|_{W_\sigma^{2+s,1+s/2}(\Om^t)})\cdot (\La_1+\La_2)\cdot[\|v\|_{W_{5/3}^{1,1/2}(\Om^t)}\cr
&\quad+|p|_{{5\over 3},\Om^t}+D_3]+c(\ro^*,A_1)(H+D_2+D_4+\phi_3\|h(0)\|_{L_2(\Om)}),\cr}
\label{10.19}
\end{equation}
where $\phi_5$ is a combination of $\phi_2$, $\phi_3$ and function $\phi$ from (A.1).

Using (\ref{10.19}) in (\ref{10.7}) yields
\begin{equation}\eqal{
\|v'\|_{V^1(\Om^t)}&\le\phi_5(\ro^*,A_1, t^{\bar{a}}\|v\|_{W_\sigma^{2+s,1+s/2}(\Om^t)})\cdot\cr
&\quad\cdot(\La_1+\La_2)\cdot[\|v\|_{W_{5/3}^{1,1/2}(\Om^t)}+ |p|_{5/3,\Om^t}+D_3]\cr
&\quad+c(\ro^*,A_1)(H+D_2+D_4+\phi_3\|h(0)\|_{L_2(\Om)}).\cr}
\label{10.20}
\end{equation}
Inequalities (\ref{10.19}) and (\ref{10.20}) imply (\ref{10.11}) and (\ref{10.12}), respectively. This ends the proof.
\end{proof}

To increase the above regularity we need

\begin{lemma}\label{l10.3}
Let the assumptions of Lemma \ref{l10.2} hold. Let
\begin{equation}
D_5=|f|_{2,\Om^t}+\|d\|_{W_2^{3/2,3/4}(S_2^t)}+\|v_0\|_{H^1(\Om)}.
\label{10.21}
\end{equation}
Then
\begin{equation}\eqal{
&\|v\|_{W_2^{2,1}(\Om^t)}+\|\nabla p\|_{L_2(\Om^t)}\le\phi_6(\ro^*,A_1, t^{\bar{a}}\|v\|_{W_\sigma^{2+s,1+s/2}(\Om^t)})\cdot\cr
&\quad\cdot(\La_1+\La_2)\cdot[\|v\|_{W_2^{1,1/2}(\Om^t)}^2+ \|v\|_{W_2^{1,1/2}(\Om^t)}+|p|_{2,\Om^t}^2+|p|_{2,\Om^t}\cr
&\quad+(H+D_2+D_4+\phi_3\cdot\|h(0)\|_{L_2(\Om)}+1)^2]+c(\ro^*,A_1) [H+D_2+D_4\cr
&\quad+D_5+\phi_3\|h(0)\|_{L_2(\Om)}],\cr}
\label{10.22}
\end{equation}
where $\phi_6$ is a combination of $\phi$ from (A.2) and $\phi_5$ and $\La_1$, $\La_2$ are assumed to be small.
\end{lemma}

\begin{proof}
We have
$$
|v|_{5,\Om^t}+|\nabla v|_{5/2,\Om^t}\le c\|v\|_{W_{5/3}^{2,1}(\Om^t)}.
$$
Then
\begin{equation}\eqal{
&|v'\cdot\nabla v|_{2,\Om^t}\le|v'|_{10,\Om^t}|\nabla v|_{5/2,\Om^t}\le c\|v'\|_{V^1(\Om^t)}\|v\|_{W_{5/3}^{2,1}(\Om^t)},\cr
&|v_3h|_{2,\Om^t}\le|v_3|_{5,\Om^t}|h|_{10/3,\Om^t}\le c|h|_{10/3,\Om^t}\|v\|_{W_{5/3}^{2,1}(\Om^t)}.\cr}
\label{10.23}
\end{equation}
Applying Lemma \ref{A.1} to problem (\ref{10.13}) for $r=2$ and using (\ref{10.23}) we obtain
\begin{equation}\eqal{
&\|v\|_{W_2^{2,1}(\Om^t)}+\|\nabla p\|_{L_2(\Om^t)}\le\phi(\ro^*, t^{\bar{a}}\|v\|_{W_\sigma^{2+s,1+s/2}(\Om^t)})\cdot\cr
&\quad\cdot\La_2\cdot[\|v\|_{W_2^{1,1/2}(\Om^t)}+|p|_{2,\Om^t}]\cr
&\quad+c(\ro^*)[\|v'\|_{V^1(\Om^t)}\|v\|_{W_{5/3}^{2,1}(\Om^t)}+ |h|_{10/3,\Om^t}\|v\|_{W_{5/3}^{2,1}(\Om^t)}\cr
&\quad+D_5].\cr}
\label{10.24}
\end{equation}
Using (\ref{10.11}) and (\ref{10.12}) in (\ref{10.24}) we derive a qualitatively equivalent inequality
$$\eqal{
&\|v\|_{W_2^{2,1}(\Om^t)}+\|\nabla p\|_{L_2(\Om^t)}\le\phi_6(\ro^*, A_1, t^{\bar{a}}\|v\|_{W_\sigma^{2+s,1+s/2}(\Om^t)})\cdot\cr
&\quad\cdot(\La_1+\La_2)\cdot[\|v\|_{W_2^{1,1/2}(\Om^t)}^2+ \|v\|_{W_2^{1,1/2}(\Om^t)}+|p|_{2,\Om^t}^2\cr
&\quad+|p|_{2,\Om^t}+(H+D_2+D_4+\phi_3\cdot\|h(0)\|_{L_2(\Om)}+1)^2]\cr
&\quad+c(\ro^*,A_1)(H+D_2+D_4+\phi_3\|h(0)\|_{L_2(\Om)}+D_5),\cr}
$$
where $\phi_6$ is a combination of $\phi$ from (A.2) with $\phi_5$ and $\La_1,\La_2$ are assumed small. The above inequality implies (\ref{10.22}) and ends the proof.
\end{proof}

We proceed with

\begin{lemma}\label{l10.4}
Let the assumptions of Lemma \ref{l10.3} hold. Let
\begin{equation}\eqal{
&D_6=|f|_{5/2,\Om^t}+\|d\|_{W_{5/2}^{8/5,4/5}(S_2^t)}+ \|v_0\|_{W_{5/2}^{6/5}(\Om)},\cr
&A_1+|F|_{6/5,2,\Om^t}+|\chi(0)|_{3,\Om}+D_6\le D_7.\cr}
\label{10.25}
\end{equation}
Then
\begin{equation}\eqal{
&\|v\|_{W_{5/2}^{2,1}(\Om^t)}+\|\nabla p\|_{L_{5/2}(\Om^t)}\le\phi(\ro^*,A_1,t^{\bar{a}}\|v\|_{W_\sigma^{2+s,1+s/2}(\Om^t)})\cr
&\quad\cdot(\La_1+\La_2+\|h(0)\|_{L_2(\Om)})\cdot [\|v\|_{W_{5/2}^{1,1/2}(\Om^t)}^4+|p|_{5/2,\Om^t}^4\cr
&\quad+\|v\|_{W_{5/2}^{1,1/2}(\Om^t)}+|p|_{5/2,\Om^t}+(H+D_7+1)^4]\cr
&\quad+c(\ro^*,A_1)(H+D_7+1)^2.\cr}
\label{10.26}
\end{equation}
\end{lemma}

\begin{proof}
We consider problem (\ref{10.13}) in the form
\begin{equation}\eqal{
&\ro v_t-\divv\T(v,p)=-\ro v\cdot\nabla v+\ro f\quad &{\rm in}\ \ \Om^T,\cr
&\divv v=0\quad &{\rm in}\ \ \Om^T,\cr
&v\cdot\bar n=0,\ \ \nu\bar n\cdot\D(v)\cdot\bar\tau_\alpha+\gamma v\cdot\bar\tau_\alpha=0,\ \ \alpha=1,2,\quad &{\rm on}\ \ S_1^T,\cr
&v\cdot\bar n=d,\ \ \bar n\cdot\D(v)\cdot\bar\tau_\alpha=0,\ \ \alpha=1,2\quad &{\rm on}\ \ S_2^T,\cr
&v|_{t=0}=v_0\quad &{\rm in}\ \ \Om.\cr}
\label{10.27}
\end{equation}
Using the estimate
\begin{equation}
|\ro v\cdot\nabla v|_{5/2,\Om^t}\le c\ro^*\|v\|_{W_2^{2,1}(\Om^t)}^2
\label{10.28}
\end{equation}
we apply Lemma \ref{A.1} for $r=5/2$ to problem (\ref{10.27}). Then we obtain
\begin{equation}\eqal{
&\|v\|_{W_{5/2}^{2,1}(\Om^t)}+\|\nabla p\|_{L_{5/2}(\Om^t)}\le\phi(\ro^*, t^{\bar{a}}\|v\|_{W_\sigma^{2+s,1+s/2}(\Om^t)})\cdot\cr
&\quad\cdot\La_2\cdot[\|v\|_{W_{5/2}^{1,1/2}(\Om^t)}+|p|_{5/2,\Om^t}]\cr
&\quad+c(\ro^*,A_1)[\|v\|_{W_2^{2,1}(\Om^t)}+D_7].\cr}
\label{10.29}
\end{equation}
Using (\ref{10.22}) in (\ref{10.29}) yields
\begin{equation}\eqal{
&\|v\|_{W_{5/2}^{2,1}(\Om^t)}+\|\nabla p\|_{L_{5/2}(\Om^t0}\le\phi(\ro^*,A_1, t^{\bar{a}}\|v\|_{W_\sigma^{2+s,1+s/2}(\Om^t)})\cdot\cr
&\quad\cdot(\La_1+\La_2)\cdot[\|v\|_{W_{5/2}^{1,1/2}(\Om^t)}^4+ |p|_{5/2,\Om^t}^4+\|v\|_{W_{5/2}^{1,1/2}(\Om^t)}\cr
&\quad+|p|_{5/2,\Om^t}+(H+D_2+D_4+\phi_3\cdot\|h(0)\|_{L_2(\Om)}+D_5+1)^4]\cr
&\quad+c(\ro^*,A_1)[(H+D_2+D_4+\phi_3\cdot\|h(0)\|_{L_2(\Om)}+D_5+1)^2+D_6].\cr}
\label{10.30}
\end{equation}
Simplifying (\ref{10.30}) implies (\ref{10.26}). This concludes the proof.
\end{proof}

Next, we derive

\begin{lemma}\label{l10.5}
Let $\si > 3/s.$ Let the assumptions of Lemma \ref{l10.4} holds. Let $5'< 5$ be a number close to 5 and let 
\begin{equation}
D_8=|f|_{5',\Om^t}+\|d\|_{W_{5'}^{2-1/5,1-1/10}(S_2^t)}+ \|v(0)\|_{W_{5'}^{2-2/5}(\Om)} 
\label{10.31}
\end{equation}
be finite.
Then
\begin{equation}\eqal{
&\|v\|_{W_{5'}^{2,1}(\Om^t)}+\|\nabla p\|_{L_{5'}(\Om^t)}\le\phi(\ro^*,A_1, t^{\bar{a}}\|v\|_{W_\sigma^{2+s,1+s/2}(\Om^t)})\cdot\cr
&\quad\cdot(\La_1+\La_2+\|h(0)\|_{L_2(\Om)})[\|v\|_{W_{5'}^{1,1/2}(\Om^t)}^8+ \|p\|_{L_{5'}(\Om^t)}^8+\|v\|_{W_{5'}^{1,1/2}(\Om^t)}\cr
&\quad+\|p\|_{L_{5'}(\Om^t)}+(H+D_7+1)^8]+c(\ro^*,A_1)[(H+D_7+1)^4+D_8].\cr}
\label{10.32}
\end{equation}
\end{lemma}

\begin{proof}
To prove the lemma we use Lemma A.2. Since
$$
|v\cdot\nabla v|_{5',\Om^t}\le c\|v\|^2_{W_{5/2}^{2,1}(\Om^t)}
$$
inequality (A.2) takes the form
\begin{equation}\eqal{
&\|v\|_{W_{5'}^{2,1}(\Om^t)}+\|\nabla p\|_{L_{5'}(\Om^t)}\le\phi(\ro^*, t^{\bar{a}}\|v\|_{W_\sigma^{2+s,1+s/2}(\Om^t)})\cdot\cr
&\quad\cdot\La_2\cdot[\|v\|_{W_{5'}^{1,1/2}(\Om^t)}+|p|_{5',\Om^t}]\cr
&\quad+c(\ro^*)[\|v\|^2_{W_{5/2}^{2,1}(\Om^t)}+D_8].\cr}
\label{10.33}
\end{equation}
Using (\ref{10.26}) in (\ref{10.33}) yields (\ref{10.32}). This concludes the proof.
\end{proof}

Finally, we obtain

\begin{lemma}\label{l10.5}
Assume that $\si > 3/s,$ $v \in W^{2,1}_{5'}(\Om^t),$ $5'< 5$ but close to 5. Let
$$
D_9=\|f\|_{W_\sigma^{s,s/2}(\Om^t)}+ \|d\|_{W_\sigma^{2+s-1/\sigma,1+s/2-1/2\sigma}(S_2^t)}+ \|v(0)\|_{W_\sigma^{2+s-2/\sigma}(\Om)}
$$ be finite. 
Then
\begin{equation}\eqal{
&\|v\|_{W_\sigma^{2+s,1+s/2}(\Om^t)}+\|\nabla p\|_{W_\sigma^{s,s/2}(\Om^t)}\le\phi(\|v\|_{W_\sigma^{2+s,1+s/2}(\Om^t)})\cdot\cr
&\quad\cdot\La_2[\|v\|_{W_\sigma^{2+s,1+s/2}(\Om^t)}+ \|p\|_{W_\sigma^{s,s/2}(\Om^t)}+\|v\|_{W_\sigma^{2,1}(\Om^t)}^2\cr
&\quad+\|f\|_{W_\sigma^{s,s/2}(\Om^t)}]+c(\ro^*) [\|v\|_{W_{5'}^{2,1}(\Om^t)}^a+D_9+A_1],\cr}
\label{10.34}
\end{equation}
where $a>2.$
\end{lemma}

\begin{proof}
To prove the lemma we use inequality (A.7). Then we have to estimate the first term under the second square bracket on the r.h.s. of (A.7) in terms of Corollary \ref{c5.3} and Lemma \ref{l10.5}. Hence, we consider
$$\eqal{
J&=\|\ro v\cdot\nabla v\|_{W_\sigma^{s,s/2}(\Om^t)}=\|\ro v\cdot\nabla v\|_{L_\sigma(0,t;W_\sigma^s(\Om))}\cr
&\quad+\|\ro v\cdot\nabla v\|_{L_\sigma(\Om;W_\sigma^{s/2}(0,t)}\equiv J_1+J_2.\cr}
$$
Consider $J_1$,
$$\eqal{
J_1&=\bigg(\intop_0^tdt\intop_\Om\intop_\Om dx'dx''{|\ro(x',t)v(x',t)\nabla v(x',t)\over|x'-x''|^{3+s\sigma}}\cr
&\hskip5,8cm-{\ro(x'',t)v(x'',t)\nabla v(x'',t)|^\sigma\over|x'-x''|^{3+ s\sigma}} \bigg)^{1/\sigma}\cr
&\le\bigg(\intop_0^tdt\intop_\Om\intop_\Om dx'dx''{|\ro(x',t)-\ro(x'',t)|^\sigma|v(x',t)|^\sigma|\nabla v(x',t)|^\sigma\over|x'-x''|^{3+s\sigma}}\bigg)^{1/\sigma}\cr
&\quad+\ro^*\bigg(\intop_0^tdt\intop_\Om\intop_\Om dx'dx''{|v(x',t)-v(x'',t)|^\sigma\over|x'-x''|^{3+s\sigma}}|\nabla v(x',t)|^\sigma\bigg)^{1/\sigma}\cr
&\quad+\ro^*\bigg(\intop_0^tdt\intop_\Om\intop_\Om dx'dx''{|v(x'',t)|^\sigma|\nabla v(x',t)-\nabla v(x'',t)|^\sigma\over|x'-x''|^{3+s\sigma}}\bigg)^{1/\sigma}\cr
&\equiv J_1^1+J_1^2+J_1^3.\cr}
$$
Consider $J_1^1$. By the H\"older inequality we have
$$\eqal{
J_1^1&\le\bigg(\intop dt\intop_\Om\intop_\Om dx'dx''{|\ro(x',t)-\ro(x'',t)|^{\sigma\lambda_1}\over |x'-x''|^{3+\sigma\lambda_1}[{1\over\sigma\lambda_1}({3\over 2}\lambda_1-3)+s]} \bigg)^{1/\sigma\lambda_1}\cdot\cr
&\quad\cdot\bigg(\intop dt\intop_\Om\intop_\Om dx'dx'' {|v(x',t)|^{\sigma\lambda_2}|\nabla v(x',t)|^{\sigma\lambda_2}\over |x'-x''|^{(3/2)\lambda_2}}\bigg)^{1/\sigma\lambda_2}\equiv L_1L_2,\cr}
$$
where $1/\lambda_1+1/\lambda_2=1$, $\lambda_2<2$ and $\lambda_1>2$.

Since
$$
L_1\le\|\ro\|_{L_{\sigma\lambda_1}(0,t;W_{\sigma\lambda_1}^{s'}(\Om))}\equiv L_1^1,
$$
where $s'={1\over\sigma\lambda_1}\big({3\over 2}\lambda_1-3\big)+s$ we use the imbedding
$$
L_1^1\le c\|\ro\|_{W_{r,\infty}^{1,1}(\Om^t)}
$$
which holds for
\begin{equation}
{3\over r}-{5\over\sigma\lambda_1}+s'\le 1
\label{10.35}
\end{equation}
Next,
$$
L_2\le c|v|_{\infty,\Om^t}|\nabla v|_{\sigma\lambda_2,\Om^t}\equiv L_2^1.
$$
To estimate $L_2^1$ we use the imbedding
$$
|\nabla v|_{\sigma\lambda_2,\Om^t}\le \|v\|_{W_{5'}^{2,1}(\Om^t)}
$$
which holds for
\begin{equation}
\frac{5}{5'}-{5\over\sigma\lambda_2}\le 1.
\label{10.36}
\end{equation}
Inequalities (\ref{10.35}) and (\ref{10.36}) imply
$$
{3\over r}-{8\over\sigma\lambda_1}+{3\over 2\sigma}+s\le 1,\quad \frac{5}{5'}-{5\over\sigma\lambda_2}\le 1.
$$
Multiplying the second inequality by $8/5$ and adding to the first we get
\begin{equation}
{3\over r}-{13\over 2\sigma}+s\le 1,
\label{10.37}
\end{equation}
so there is no restriction.

The first factor in $L_2^1$ we estimate by
$$
|v|_{\infty,\Om^t}\le c\|v\|_{W_5^{2,1}(\Om^t)}.
$$
Using Corollary \ref{c5.3} we obtain
\begin{equation}
J_1^1\le\phi(t^{\bar{a}}\|v\|_{W_{\si}^{2+s,1+s/2}(\Om^t)})\cdot\La_2\cdot \|v\|_{W_5^{2,1}(\Om^t)}.
\label{10.38}
\end{equation}
Consider $J_1^2$. By the H\"older inequality it holds
$$ J_1^2\le\ro^*\|v\|_{L_{\sigma\lambda_1}(0,t;W_{\sigma\lambda_1}^{s'}(\Om))} |\nabla v|_{\sigma\lambda_2,\Om^t}=N_1N_2,
$$
where $1/\lambda_1+1/\lambda_2=1$, $\lambda_1>2$, $\lambda_2<2$ and $s'={1\over\sigma\lambda_1}\big({3\over 2}\lambda_1-3\big)+s$.

Continuing
$$
N_1\le c\|v\|_{W_{5'}^{2,1}(\Om^t)}\quad {\rm for}\ \ \frac{5}{5'}-{5\over\sigma\lambda_1}+{3\over 2\sigma}-{3\over\sigma\lambda_1}+s\le 2
$$
and
$$
N_2\le c\|v\|_{W_5^{2,1}(\Om^t)}\quad {\rm for}\ \ \frac{5}{5'}-{5\over\sigma\lambda_2}\le 1.
$$
Therefore
$$
J_1^2\le c\|v\|_{W_{5'}^{2,1}(\Om^t)}
$$
holds for
\begin{equation}
s-{13\over 2\sigma} < 1.
\label{10.39}
\end{equation}
Finally
$$
J_1^3\le c|v|_{\sigma\lambda_1,\Om^t}\|\nabla v\|_{L_{\sigma\lambda_2}(0,t;W_{\sigma\lambda_2}^{s'}(\Om))}\le c\|v\|_{W_{5'}^{2,1}(\Om^t)}^2,
$$
where the last inequality holds for
\begin{equation}
s-{13\over 2\sigma}\le{8\over 5}.
\label{10.40}
\end{equation}
Summarizing,
\begin{equation}
J_1\le\phi(\|v\|_{W_\sigma^{2+s,1+s/2}(\Om^t)})\cdot\La_2\cdot \|v\|_{W_5^{2,1}(\Om^t)}+c\|v\|_{W_5^{2,1}(\Om^t)}.
\label{10.41}
\end{equation}
Next, we examine $J_2$,
$$\eqal{
J_2&\le\bigg(\intop_\Om dx\intop_0^t\intop_0^t {dt'dt''|\ro(x,t')v(x,t')\nabla v(x,t')\over|t'-t''|^{1+s\sigma/2}}\cr
&\hskip5,8cm-{\ro(x,t'')v(x,t'')\nabla v(x,t'')|^\sigma\over|t'-t''|^{1+s\sigma/2}}\bigg)^{1/\sigma}\cr
&\le\bigg(\intop_\Om dx\intop_0^t\intop_0^tdt'dt'' {|\ro(x,t')-\ro(x,t'')|^\sigma\over|t'-t''|^{1+s\sigma/2}} |v(x,t')|^\sigma|\nabla v(x,t')|^\sigma\bigg)^{1/\sigma}\cr
&\quad+\ro^*\bigg(\intop_\Om dx\intop_0^t\intop_0^t dt'dt'' {|v(x,t')-v(x,t'')|^\sigma|\nabla v(x,t')|^\sigma\over|t'-t''|^{1+s\sigma/2}}\bigg)^{1/\sigma}\cr
&\quad+\ro^*\bigg(\intop_\Om dx\intop_0^t\intop_0^t dt'dt'' {|v(x,t'')|^\sigma|\nabla v(x,t')-\nabla v(x,t'')|^\sigma\over|t'-t''|^{1+s\sigma/2}}\bigg)^{1/\sigma}\cr
&\equiv J_2^1+J_2^2+J_2^3.\cr}
$$
By the H\"older inequality we have
$$
J_2^1\le c|v|_{\infty,\Om^t} \|\ro\|_{L_{\sigma\lambda_1}(\Om;W_{\sigma\lambda_1}^{s'/2}(0,t))}|\nabla v|_{\sigma\lambda_2,\Om^t}\equiv J_2^{11},
$$
where $1/\lambda_1+1/\lambda_2=1$, $\lambda_2<2$, $s'={1\over\sigma\lambda_1}(\lambda_1-2)+s$. Let $\sigma\lambda_1=r$. Then $\sigma\lambda_2={\sigma r\over r-\sigma}$. Using the imbeddings
$$
\|\ro\|_{L_r(\Om;W_r^{s'/2}(0,t))}\le c\|\ro\|_{W_{r,\infty}^{1,1}(\Om^t)}
$$
which holds for $s'\le 2$, and
$$\eqal{
&|v|_{\infty,\Om^t}\le c\|v\|_{W_{5'}^{2,1}(\Om^t)},\ \ {5\over 5'}<2\cr
&|\nabla v|_{{\sigma r\over r-\sigma},\Om^t}\le c\|v\|_{W_{5'}^{2,1}(\Om^t)}\cr}
$$
which holds for ${5}/{5'}-5/\sigma+5/r<1$. The last inequality is valid for $r > \sigma$.

Then
$$
J_2^{11}\le\phi(t^{\bar{a}}\|v\|_{W_\sigma^{2+s,1+s/2}(\Om^t)})\Lambda_2 \|v\|_{W_5^{2,1}(\Om^t)}^2.
$$
Applying the H\"older inequality with respect to space variables we get
$$
J_2^2\le\ro^*\bigg(\intop_0^t\intop_0^t {|v(t')-v(t'')|_{\sigma\lambda_1,\Om}^\sigma|\nabla v(t')|_{\sigma\lambda_2,\Om}^\sigma\over|t'-t''|^{1+s\sigma/2}}dt'dt'' \bigg)^{1/\sigma}\equiv J_2^{21},
$$
where $1/\lambda_1+1/\lambda_2=1$.

Applying the H\"older inequality with respect to time yields
$$\eqal{
J_2^{21}&\le\ro^*\bigg(\intop_0^t\intop_0^t {|v(t')-v(t'')|_{\sigma\lambda_1,\Om}^{\sigma\lambda_1}\over |t'-t''|^{1+\sigma\lambda_1s'/2}}dt'dt''\bigg)^{1/\sigma\lambda_1}\cdot\cr
&\quad\cdot\bigg(\intop_0^t\intop_0^t{|\nabla v(t')|_{\sigma\lambda_2,\Om}^{\sigma\lambda_2}\over|t'-t''|^{\lambda_2/2}} dt'dt''\bigg)^{1/\sigma\lambda_2}\equiv J_2^{22},\cr}
$$
where $s'=s+{1\over\sigma\lambda_1}(\lambda_1-2)$. For $\lambda_2<2$ we can perform integration with respect to $t''$ in the second factor of $J_2^{22}$. Then we obtain
$$
J_2^{22}\le\ro^*\|v\|_{L_{\sigma\lambda_1}(\Om;W_{\sigma\lambda_1}^{s'}(0,t))} |\nabla v|_{\sigma\lambda_2,\Om^t}\equiv J_2^{23}.
$$
We use the imbeddings
$$
\|v\|_{L_{\sigma\lambda_1}(\Om;W_{\sigma\lambda_1}^{s'/2}(0,t))}\le c\|v\|_{W_{5'}^{2,1}(\Om^t)}
$$
for $\frac{5}{5'}-{5\over\sigma\lambda_1}+s'\le 2$ so $\frac{5}{5'}-{5\over\sigma\lambda_1}+s+{1\over\sigma}-{2\over\sigma\lambda_1}\le 2$, and
$$
\|\nabla v\|_{L_{\sigma\lambda_2}(\Om^t)}\le c\|v\|_{W_{5'}^{2,1}(\Om^t)}
$$
which holds for $\frac{5}{5'}-{5\over\sigma\lambda_2}\le 1$.

The above restrictions are satisfied. Hence
$$
J_2^{23}\le c\ro^*\|v\|_{W_{5'}^{2,1}(\Om^t)}^2.
$$
Finally,
$$
J_2^3\le\ro^*|v|_{\infty,\Om^t}\|\nabla v\|_{L_\sigma(\Om;W_\sigma^{s/2}(0,t))}
$$
%where the imbedding
%$$
%\|\nabla v\|_{L_\sigma(\Om;W_\sigma^{s/2}(0,t))}\le c\|v\|_{W_5^{2,1}(\Om^t)}
%$$
%holds for $s\le 5/\sigma$.
We apply the interpolation
$$ \|\nb v\|_{L_\sigma(\Om;W_\sigma^{s/2}(0,t))} \le c \|v\|_{W_\sigma^{2+s,1+s/2}(\Om^t)}^{\theta}\|v\|^{1-\theta}_{L_{\infty}(\Om^t)}$$
where $\theta = \frac{1+s-5/\si}{2+s-5/\si},$ then this implies
$$ J_2^3\le \eps \|v\|_{W_\sigma^{2+s,1+s/2}(\Om^t)}+c(1/\eps,\ro^*)\|v\|^{\frac{2-\theta}{1-\theta}}_{L_{\infty}(\Om^t)}.$$
Summarizing the above estimates yields
\begin{equation}\eqal{
J_2&\le\phi(t^{\bar{a}}\|v\|_{W_\sigma^{2+s,1+s/2}(\Om^t)})\La_2\|v\|_{W_{5'}^{2,1}(\Om^t)}\cr
&\quad+  \eps \|v\|_{W_\sigma^{2+s,1+s/2}(\Om^t)}+c(1/\eps,\ro^*)\|v\|^{\frac{2-\theta}{1-\theta}}_{W_{5'}^{2,1}(\Om^t)}.\cr}
\label{10.42}
\end{equation}

In view of (\ref{10.41}) and (\ref{10.42}) we have
\begin{equation}\eqal{
J&\le J_1+J_2\le\phi(t^{\bar{a}}\|v\|_{W_\si^{2+s,1+s/2}(\Om^t)})\La_2 \|v\|_{W_{5'}^{2,1}(\Om^t)}^2\cr
&\quad+ \eps \|v\|_{W_\si^{2+s,1+s/2}(\Om^t)}+c(1/\eps,\ro^*)\|v\|^{\frac{2-\theta}{1-\theta}}_{W_{5'}^{2,1}(\Om^t)}.\cr}
\label{10.43}
\end{equation}
To apply (A.7) we need the estimate
\begin{equation}
\|\ro\|_{C^\alpha(\Om^t)}\le\|\ro\|_{W_{r,\infty}^{1,1}(\Om^t)}
\label{10.44}
\end{equation}
which holds for
$$
{3\over r}+\alpha<1.
$$
Using (A.7) and estimates (\ref{10.43}), (\ref{10.44}) we obtain (\ref{10.34}). This ends the proof.
\end{proof}

\section{Global estimate}\label{s11}
In this Section we estimate $H$ (using the Stokes regularity theory) and simplify formulas for constants $D_i$ in order to infer the global estimate.

\begin{remark}\label{r11.1}
Let $\La=\La_1+\La_2+\|h(0)\|_{L_2(\Om)}$. Let $\si > 3/s, s\in (0,1).$ Using (\ref{10.32}) in (\ref{10.34}) yields
\begin{equation}\eqal{
&\|v\|_{W_\sigma^{2+s,1+s/2}(\Om^t)}+\|\nabla p\|_{W_\sigma^{s,s/2}(\Om^t)}\cr
&\le\phi(t^{\bar{a}}\|v\|_{W_\sigma^{2+s,1+s/2}(\Om^t)},t^{\bar{a}}\|p\|_{W_\sigma^{s,s/2}(\Om^t)}, H,D_7)\cdot\La\cr
&\quad+c(\ro^*,A_1)[(H+D_7+1)^8+D_8^2+D_9+1],\cr}
\label{11.1}
\end{equation}
where $\bar{a}> 0$ and 
\begin{equation}
H=|h|_{3,\infty,\Om^t}+\|h\|_{1,2,\Om^t}+|h|_{10/3,\Om^t}\le c\|h\|_{W_{5/3}^{2,1}(\Om^t)}.
\label{11.2}
\end{equation}
\end{remark}
Finally, we have to find an estimate for $\|h\|_{W_{5/3}^{2,1}(\Om^t)}$.

\begin{lemma}\label{r11.2}
Assume that $D_2, D_4, D_{10}$ are finite, $s\in (0,1),$ $h \in W_{5/3}^{1,1/2}(\Om^t),$  $q \in L_{5/3}(\Om^t), v \in W_2^{1,1/2}(\Om^t), p \in L_2(\Om^t),$ $f \in W_\sigma^{s,s/2}(\Om^t),$ $A_1$ and $\La$ are defined in (\ref{1.7}) and (\ref{Lambda}). Then $(h,q)$ solutions to problem (\ref{2.6}) satisfy
\begin{equation}\eqal{
&\|h\|_{W_{5/3}^{2,1}(\Om^t)}+\|\nabla q\|_{L_{5/3}(\Om^t)}\le \phi(\ro^*,A_1,t^{\bar{a}}\|v\|_{W_\sigma^{2+s,1+s/2}(\Om^t)})\cdot\cr
&\quad\cdot\La\cdot[\|h\|_{W_{5/3}^{1,1/2}(\Om^t)}+\|q\|_{L_{5/3}(\Om^t)}+ \|v\|_{W_2^{1,1/2}(\Om^t)}^2\cr
&\quad+\|v\|_{W_2^{1,1/2}(\Om^t)}+\|p\|_{L_2(\Om^t)}^2+\|p\|_{L_2(\Om^t)}+ (H+D_2+D_4+1)^2\cr
&\quad+H+D_2+D_4+1+\|f\|_{W_\sigma^{s,s/2}(\Om^t)}]+c(\ro^*,A_1)D_{10}.\cr}
\label{11.7}
\end{equation}
\end{lemma}
\begin{proof}
Applying Lemma \ref{A.1} to problem (\ref{2.6}) implies
\begin{equation}\eqal{
&\|h\|_{W_{5/3}^{2,1}(\Om^t)}+\|\nabla q\|_{L_{5/3}(\Om^t)}\le\phi (\ro^*,t^{\bar{a}}\|v\|_{W_\sigma^{2+s,1+s/2}(\Om^t)})\cdot\cr
&\quad\cdot\La\cdot[\|h\|_{W_{5/3}^{1,1/2}(\Om^t)}+ \|q\|_{L_{5/3}(\Om^t)}]\cr
&\quad+c(\ro^*,A_1)[\|v\cdot\nabla h\|_{L_{5/3}(\Om^t)}+\|h\cdot\nabla v\|_{L_{5/3}(\Om^t)}+\|g\|_{L_{5/3}(\Om^t)}\cr
&\quad+\|\ro_{x_3}(v_t+v\cdot\nabla v-f)\|_{L_{5/3}(\Om^t)}+ \|d_{x'}\|_{W_{5/3}^{7/5,7/10}(S_2^t)}\cr
&\quad+\|h(0)\|_{W_{5/3}^{4/5}(\Om)}].\cr}
\label{11.3}
\end{equation}
In view of (\ref{9.8}) we have
$$\eqal{
|h\cdot\nabla v|_{5/3,\Om^t}&\le|h|_{10/3,\Om^t}|\nabla v|_{10/3,\Om^t}\le c\|h\|_{V(\Om^t)}\|v\|_{W_2^{2,1}(\Om^t)}\cr
&\le\phi(\|v\|_{W_\sigma^{2+s,1+s/2}(\Om^t)})\cdot\La,\cr
|v\cdot\nabla h|_{5/3,\Om^t}&\le|v|_{10,\Om^t}|\nabla h|_{2,\Om^t}\le c\|h\|_{V(\Om^t)}\|v\|_{W_2^{2,1}(\Om^t)}\cr
&\le\phi(\|v\|_{W_\sigma^{2+s,1+s/2}(\Om^t)})\cdot\La.\cr}
$$
$$\eqal{
&|\ro_{x_3}(v_t+v\cdot\nabla v-f)|_{5/3,\Om^t}\le|\ro_{x_3}|_{r,\infty,\Om^t}(|v_t|_{{5r\over 3r-5},{5\over 3},\Om^t}\cr
&\quad+|v|_{\infty,\Om^t}|\nabla v|_{{5r\over 3r-5},{5\over 3},\Om^t}+|f|_{{5r\over 3r-5},{5\over 3},\Om^t})\cr
&\le\phi(\|v\|_{W_\sigma^{2+s,1+s/2}(\Om^t)})\cdot\La\cdot(1+ \|f\|_{W_\sigma^{s,s/2}(\Om^t)}).\cr}
$$
In view of the above estimates and (\ref{10.22}) inequality (\ref{11.3}) takes the form
\begin{equation}\eqal{
&\|h\|_{W_{5/3}^{2,1}(\Om^t)}+\|\nabla q\|_{5/3(\Om^t)}\le\phi(\ro^*,t^{\bar{a}}\|v\|_{W_\sigma^{2+s,1+s/2}(\Om^t)})\cdot\cr
&\quad\cdot\La\cdot[\|h\|_{W_{5/3}^{1,1/2}(\Om^t)}+ \|q\|_{L_{5/3}(\Om^t)}]+c(\ro^*,A_1)\|h\|_{V(\Om^t)} \|v\|_{W_2^{2,1}(\Om^t)}\cr
&\quad+\phi(\ro^*,t^{\bar{a}}\|v\|_{W_\sigma^{2+s,1+s/2}(\Om^t)})\cdot\La\cdot (1+\|f\|_{W_\sigma^{s,s/2}(\Om^t)})+c(\ro^*,A_1)D_{10},\cr}
\label{11.4}
\end{equation}
where
\begin{equation}
D_{10}=|g|_{5/3,\Om^t}+\|d_{x'}\|_{W_{5/3}^{7/5,7/10}(S_2^t)}+ \|h(0)\|_{W_{5/3}^{4/5}(\Om)}.
\label{D10}
\end{equation}
Inequality (\ref{9.8}) implies
\begin{equation}
\|h\|_{V(\Om^t)}\le\phi(\ro^*,t^{\bar{a}}\|v\|_{W_\sigma^{2+s,1+s/2}(\Om^t)})\cdot \La
\label{11.5}
\end{equation}
and (\ref{10.22}) gives
\begin{equation}\eqal{
\|v\|_{W_2^{2,1}(\Om^t)}&\le\phi(\ro^*, t^{\bar{a}}\|v\|_{W_\sigma^{2+s,1+s/2}(\Om^t)})\cdot\La\cr
&\quad\cdot[\|v\|_{W_2^{1,1/2}(\Om^t)}^2+\|v\|_{W_2^{1,1/2}(\Om^t)}+ |p|_{2,\Om^t}^2+|p|_{2,\Om^t}\cr
&\quad+(H+D_2+D_4+1)^2]+c(\ro^*,A_1)[H+D_2+D_4+D_5].\cr}
\label{11.6}
\end{equation}
Using (\ref{11.5}) and (\ref{11.6}) in (\ref{11.4}) we obtain the inequality (\ref{11.7})
$$\eqal{
&\|h\|_{W_{5/3}^{2,1}(\Om^t)}+\|\nabla q\|_{L_{5/3}(\Om^t)}\le \phi(\ro^*,A_1,t^{\bar{a}}\|v\|_{W_\sigma^{2+s,1+s/2}(\Om^t)})\cdot\cr
&\quad\cdot\La\cdot[\|h\|_{W_{5/3}^{1,1/2}(\Om^t)}+\|q\|_{L_{5/3}(\Om^t)}+ \|v\|_{W_2^{1,1/2}(\Om^t)}^2\cr
&\quad+\|v\|_{W_2^{1,1/2}(\Om^t)}+\|p\|_{L_2(\Om^t)}^2+\|p\|_{L_2(\Om^t)}+ (H+D_2+D_4+1)^2\cr
&\quad+H+D_2+D_4+1+\|f\|_{W_\sigma^{s,s/2}(\Om^t)}]+c(\ro^*,A_1)D_{10}.\cr}
$$
\end{proof}

\begin{remark}\label{r11.3}
We decrease the number of constants $D_i$, $i=1,\dots,11$.
$$\eqal{
&D_1=|d_1|_{3,2,S_2^t},\quad D_2=A_1+|F|_{6/5,2,\Om^t}+|\chi(0)|_{3,\Om},\cr
&D_3=|(f_1,f_2)|_{{6r\over 5r-6},2,\Om^t}+|\chi(0)|_{3,\Om},\cr
&D_4=|f|_{5/3,\Om^t}+\|d\|_{W_{5/3}^{7/5,7/10}(S_2^t)}+ \|v(0)\|_{W_{5/3}^{4/5}(\Om)},\cr
&D_5=|f|_{2,\Om^t}+\|d\|_{W_2^{3/2,3/4}(S_2^t)}+\|v(0)\|_{1,\Om}.\cr}
$$
Hence, for $r\ge 3$ we have
$$
D_1+D_3+D_4\le D_2+D_5.
$$
Next,
$$
D_6=\|f\|_{L_{5/2}(\Om^t)}+\|d\|_{W_{5/2}^{8/5,4/5}(S_2^t)}+ \|v(0)\|_{W_{5/2}^{6/5}(\Om)},
$$
where $r\ge\sigma$, $s> 5/r$, $r>3$.

Continuing, $D_5\le D_6$ and
$$ \eqal{
D_7 &= D_2+ D_6 = A_1+|F|_{6/5,2,\Om^t}+|\chi(0)|_{3,\Om}+\|f\|_{L_{5/2}(\Om^t)} \cr
& \quad +\|d\|_{W_{5/2}^{8/5,4/5}(S_2^t)}+ \|v(0)\|_{W_{5/2}^{6/5}(\Om)},\cr
D_8 & =|f|_{5,\Om^t}+\|d\|_{W_5^{9/5,9/10}(S_2^t)}+\|v(0)\|_{W_5^{8/5}(\Om)}.\cr}
$$
Then
$$
D_7\le|F|_{6/5,2,\Om^t}+A_1+D_8\equiv\bar D_8 $$
since $$ \|\chi(0)\|_{L_3(\Om)} \le \|v(0)\|_{W^1_3(\Om)} \le \|v(0)\|_{W^{8/5}_5(\Om)}. $$

Next, we introduce
$$\eqal{
D_9&=\|f\|_{W_\sigma^{s,s/2}(\Om^t)}+ \|d\|_{W_\sigma^{2+s-1/\sigma,1+s/2-1/2\sigma}(S_2^t)}\cr
&\quad+\|v(0)\|_{W_\sigma^{2+s-2/\sigma}(\Om)}.\cr}
$$
Then,
$$
\bar D_8\le A_1+|F|_{6/5,2,\Om^t}+D_9.
$$
Finally
$$
D_{10}=\|g\|_{L_{5/3}(\Om^t)}+\|d_{x'}\|_{W_{5/3}^{7/5,7/10}(S_2^t)}+ \|h(0)\|_{W_{5/3}^{4/5}(\Om)}.
$$
Then we define the final constant which estimates all other constants
$$
D_{11}=D_9+D_{10}+|F|_{6/5,2,\Om^t}+A_1.
$$
\end{remark}

\begin{lemma}[global estimate]\label{l11.4}
Assume that $r>\sigma$, $3/r<s$, $5/\sigma<1+s,$
$$\eqal{
&D_9=\|f\|_{W_\sigma^{s,s/2}(\Om^t)}+ \|d\|_{W_\sigma^{2+s-1/\sigma,1+s/2-1/2\sigma}(S_2^t)}+ \|v(0)\|_{W_\sigma^{2+s-2/\sigma}(\Om)}<\infty,\cr
&D_{10}=\|g\|_{L_{5/3}(\Om^t)}+\|d_{x'}\|_{W_{5/3}^{7/5,7/10}(S_2^t)}+ \|h(0)\|_{W_{5/3}^{4/5}(\Om)}<\infty,\cr
&\|F\|_{L_2(0,t;L_{6/5}(\Om))}<\infty,\quad D_{11}=D_9+D_{10}+\|F\|_{L_2(0,t;L_{6/5}(\Om))}+A_1.\cr}
$$
Let
$$\eqal{
\La_1 & =\|d_{x'}\|_{L_2(0,t;W^1_3(S_2))}+\|d_{x'}\|_{L_\infty(0,t;L_2(S_2))}+\|d_t\|_{L_2(0,t;H^1(S_2))}\cr & \quad +\|f_3\|_{L_2(0,t;L_{4/3}(S_2))}+\|f_{x_3}\|_{L_2(\Om^t)},\cr
 \La_2 &  =\|\ro_{1,x'}\|_{L_r(S_2^t(-a))}+\|\ro_{1,t}\|_{L_r(S_2^t(-a))}+\|\ro_{0,x}\|_{L_r(\Om)},  \cr}
$$
Let $\La=\La_1+\La_2+\|h(0)\|_{L_2(\Om)}$ and $\phi(t)\La,$ where $\phi$ is an increasing positive function, be sufficiently small. Then
\begin{equation}\eqal{
&\|v\|_{W_\sigma^{2+s,1+s/2}(\Om^t)}+\|\nabla p\|_{W_\sigma^{s,s/2}(\Om^t)}+\|h\|_{W_{5/3}^{2,1}(\Om^t)}\cr
&\quad+\|\nabla q\|_{L_{5/3}(\Om^t)}\le\phi(\ro^*,A_1,D_{11}).\cr}
\label{11.8}
\end{equation}
\end{lemma}

\begin{proof}
Introduce the quantity
\begin{equation}\eqal{
X(t)&=\|v\|_{W_\sigma^{2+s,1+s/2}(\Om^t)}+\|\nabla p\|_{W_\sigma^{s,s/2}(\Om^t)}\cr
&\quad+\|h\|_{W_{5/3}^{2,1}(\Om^t)}+\|\nabla q\|_{L_{5/3}(\Om^t)}.\cr}
\label{11.9}
\end{equation}

Then inequalities (\ref{11.1}) and (\ref{11.7}) imply
$$
X\le\phi(t^{\bar{a}}X)\cdot\La+\phi(\ro^*,A_1,D_{11}).
$$
Hence for $\phi(t)\La$ sufficiently small (\ref{11.8}) holds.
\end{proof}

\section*{Appendix}\label{sA}
\setcounter{equation}{0}
\setcounter{theorem}{0}

\renewcommand*\thesection{A.\arabic{section}}
\renewcommand*\theequation{A.\arabic{equation}}
\renewcommand*\thetheorem{A.\arabic{theorem}}

In this Section we consider the Stokes problem corresponding to system (\ref{1.1}) and derive some estimates for solutions. Let $\zeta^{(k,l)}$ be the partition of unity introduced in the proof of Lemma 5.2 for \cite{RZ3}. Let $\zeta^{(k,l)}=\zeta^{(k)}(x)\zeta_0^{(l)}(t)$, $k,l\in\N$. Let
$$\eqal{
&\sup_k\diam\supp\zeta^{(k)}(x)\le\lambda,\cr
&\sup_l\diam\supp\zeta_0^{(l)}(t)\le\lambda,\cr}
$$
where $\lambda$ will be chosen later. Let $\xi^{(k)}$, $\xi_0^{(l)}$ be interior points of $\supp\zeta^{(k)}$ and $\supp\zeta_0^{(l)}$, respectively. If $\bar{\supp}\zeta^{(k)}\cap S$ then $\xi^{(k)}$ is an interior point of $\bar{\supp}\zeta^{(k)}\cap S$.

Let $\tilde v^{(k,l)}=v\zeta^{(k,l)}$, $\tilde p^{(k,l)}=p\zeta^{(k,l)}$, $\tilde f^{(k,l)}=f\zeta^{(k,l)}$. Then the localized problem (\ref{1.1}) takes the form
\begin{equation}\eqal{
&\ro(\xi^{(k)},\xi_0^{(l)})\tilde v_{t}^{(k,l)}-\nu\Delta\tilde v^{(k,l)}+\nabla\tilde p^{(k,l)}=[\ro(\xi^{(k)},\xi_0^{(l)})\cr
&\quad-\ro(x,t)]\tilde v_{t}^{(k,l)}+\ro v\zeta_{t}^{(k,l)}+p\nabla\zeta^{(k,l)}-2 \nu\nabla v\nabla\zeta^{(k,l)}\cr
&\quad-\nu v\Delta\zeta^{(k,l)}-\ro v\cdot\nabla v\zeta^{(k,l)}+\ro\tilde f^{(k,l)},\quad &{\rm in}\  \Om^T,\cr
&\divv\tilde v^{(k,l)}=v\cdot\nabla\zeta^{(k,l)},\quad &{\rm in}\  \Om^T,\cr
&\bar n\cdot\tilde v^{(k,l)}=0, \quad &{\rm on}\  S_1, \cr
&\nu\bar n\cdot\D(\tilde v^{(k,l)})\cdot\bar\tau_\alpha=\nu n_i(v_i\zeta_{x_j}^{(k,l)}+v_j\zeta_{x_i}^{(k,l)})\tau_{\alpha_j}-\gamma\tilde v^{(k,l)}\cdot\bar\tau_\alpha\quad &{\rm on}\ \ S_1,\cr
&\bar n\cdot\tilde v^{(k,l)}|=\tilde d^{(k,l)} \quad &{\rm on}\ \ S_2,\cr
&\bar n\cdot\D(\tilde v^{(k,l)})\cdot\bar\tau_\alpha= n_i(v_i\zeta_{x_j}^{(k,l)}+v_j\zeta_{x_i}^{(k,l)})\cdot\tau_{\alpha j}\quad &{\rm on}\ \ S_2,\cr
&\tilde v^{(k,l)}|_{t=0}=\tilde v^{(k,l)}(0).
\cr}
\label{A.1}
\end{equation}

\begin{lemma}\label{lA.1}
Assume that $v\in W_r^{2+\sigma,1+\sigma/2}(\Om^t)$, $v\in W_r^{1,1/2}(\Om^t)$, $p\in L_r(\Om^t)$, $v\cdot\nabla v\in L_r(\Om^t)$, $f\in L_r(\Om^t)$, $d\in W_r^{2-1/r,1-1/2r}(S_2^t)$, $v(0)\in W_r^{2-2/r}(\Om)$, $3/r<\sigma$, $r>3$, $\alpha<1-3/r$, $\La_2=\|\ro_{1,x'}\|_{L_r(S_2^t)}+\|\ro_{1,t}\|_{L_r(S_2^t)}+ \|\ro_{0,x}\|_{L_r(\Om)}$. Let $\phi$ be and increasing positive function such that $\phi(0)=0$. Then
\begin{equation}\eqal{
&\|v\|_{W_r^{2,1}(\Om^t)}+\|\nabla p\|_{L_r(\Om^t)}\le c(\ro^*)\phi(t^{\bar{a}}\|v\|_{W_r^{2+\sigma,1+\sigma/2}(\Om^t)})\cdot\La_2\cdot\cr
&\quad\cdot[\|v\|_{W_r^{1,1/2}(\Om^t)}+\|p\|_{L_r(\Om^t)}]+c(\ro^*)[\|v\cdot\nabla v\|_{L_r(\Om^t)}+\|f\|_{L_r(\Om^t)} \cr
&\quad+\|d\|_{W_r^{2-1/r,1-1/2r}(S_2^t)}+\|v(0)\|_{W_r^{2-2/r}(\Om)}].\cr}
\label{A.2}
\end{equation}
\end{lemma}

\begin{proof}
From Theorem~6.1 \cite{RZ4} we have
\begin{equation}\eqal{
&\|\tilde v^{(k,l)}\|_{W_r^{2,1}(\Om^t)}+\|\nabla\tilde p^{(k,l)}\|_{L_r(\Om^t)}\cr
&\le c(\|(\ro(\xi^{(k)},\xi_0^{(l)})-\ro(x,t))\tilde v_t^{(k,l)}\|_{L_r(\Om^t)}+\|\ro v\zeta_{t}^{(k,l)}\|_{L_r(\Om^t)}\cr
&\quad+\|p\nabla\zeta^{(k,l)}\|_{L_r(\Om^t)}+\|\nabla v\nabla\zeta^{(k,l)}\|_{L_r(\Om^t)}+\|v\Delta\zeta^{(k,l)}\|_{L_r(\Om^t)}\cr
&\quad+\|\ro v\cdot\nabla v\|_{L_r(\Om^{(k,l)}\times(0,t))}+\|\ro\tilde f^{(k,l)}\|_{L_r(\Om^t)}\cr
&\quad+\|v\nabla\zeta^{(k,l)}\|_{W_r^{1-1/r,1/2-1/2r}(S_2^t)}+\|\tilde d^{(k,l)}\|_{W_r^{2-1/r,1-1/2r}(S_2^t)}\cr
&\quad+\|\tilde v^{(k,l)}(0)\|_{W_r^{2-2/r}(\Om)}),\cr}
\label{A.3}
\end{equation}
where $\Om^{(k,l)}=\Om^t\cap\supp\zeta^{(k,l)}$.

The first term on the r.h.s. of (\ref{A.3}) is bounded by
$$\eqal{
&\bigg(\sup_{x,x',t\in\Om^{(k,l)}}{|\ro(x,t)-\ro(x',t)|\over|x-x'|^\alpha} \lambda^\alpha\cr
&\quad+\sup_{x,t,t'\in\Om^{(k,l)}}{|\ro(x,t)-\ro(x,t')|\over |t-t'|^\alpha}\lambda^\alpha\bigg)\|\tilde v_{t}^{(k,l)}\|_{L_r(\Om^t)}\cr
&\le c\|\ro\|_{\dot C^\alpha(\Om^{(k,l)})}\lambda^\alpha\|\tilde v_{t}^{(k,l)}\|_{L_r(\Om^t)}.\cr}
$$
Summing up inequalities (\ref{A.3}) over all neighborhoods of the partition of unity, using that
\begin{equation}
c\|\ro\|_{\dot C^\alpha(\Om\times(0,t))}\lambda^\alpha\le{1\over 2}
\label{A.4}
\end{equation}
and using that $|\nabla\zeta^{(k,l)}|\le{c\over\lambda}$ and so on we obtain the inequality
\begin{equation}\eqal{
&\|v\|_{W_r^{2,1}(\Om^t)}+\|\nabla p\|_{L_r(\Om^t)}\le c(\ro^*)\phi(\|\ro\|_{\dot C^\alpha(\Om^t)})\cdot\cr
&\quad\cdot(\|v\|_{W_r^{1,1/2}(\Om^t)}+\|p\|_{L_r(\Om^t)})\cr
&\quad+c(\ro^*)(\|v\cdot\nabla v\|_{L_r(\Om^t)}+\|f\|_{L_r(\Om^t)}+ \|d\|_{W_r^{2-1/r,1-1/2r}(S_2^t)}\cr
&\quad+\|v(0)\|_{W_r^{2-2/r}(\Om)}).\cr}
\label{A.5}
\end{equation}
Let $r>3$. Then for $\alpha<1-3/r$ we obtain
\begin{equation}
\|\ro\|_{\dot C^\alpha(\Om^t)}\le c\|\ro\|_{\dot W_{r,\infty}^{1,1}(\Om^t)}\le\phi (t^{\bar{a}}\|v\|_{W_r^{2+\sigma,1+\sigma/2}(\Om^t)})\cdot\La_2,
\label{A.6}
\end{equation}
where $3/r<\sigma$. This ends the proof.
\end{proof}

\begin{lemma}\label{lA.2}
Assume that $r>\sigma$, $3/r<s$, $5/\sigma<1+s$.
 Assume also that $v\in W_\sigma^{2+s,1+s/2}(\Om^t)$, $p\in W_\sigma^{s,s/2}(\Om^t)$, $\ro \in W_{r,\infty}^{1,1}(\Om^t),$ $f\in W_\sigma^{s,s/2}(\Om^t)$, $d\in W_\sigma^{2+s-1/\sigma,1+s/2-1/2\sigma}(\Om^t)$, $v(0)\in W_\sigma^{2+s-2/\sigma}(\Om)$ and $A_1$ is introduced in (\ref{1.7}).
Then
\begin{equation}\eqal{
&\|v\|_{W_\sigma^{2+s,1+s/2}(\Om^t)}+\|\nabla p\|_{W_\sigma^{s,s/2}(\Om^t)}\le[\phi(t^{\bar{a}}\|v\|_{W_\sigma^{2+s,1+s/2}(\Om^t)})\cdot\cr
&\quad\cdot\La_2+c\ro^*](\|v\|_{W_\sigma^{1+s,1/2+s/2}(\Om^t)}+ \|p\|_{W_\sigma^{s,s/2}(\Om^t)})\cr
&\quad+\phi(\ro^*)[\|\ro v\cdot\nabla v\|_{W_\sigma^{s,s/2}(\Om^t)}+\|\ro\|_{C^\alpha(\Om^t)} \|f\|_{W_\sigma^{s,s/2}(\Om^t)}\cr
&\quad+A_1+\|d\|_{W_\sigma^{2+s-1/\sigma,1+s/2-1/2\sigma}(S_2^t)}+ \|v(0)\|_{W_\sigma^{2+s-2/\sigma}(\Om)}].\cr}
\label{A.7}
\end{equation}
\end{lemma}

\begin{proof}
Applying Theorem 1.1 from \cite{RZ4} to (\ref{A.1}) yields
\begin{equation}\eqal{
&\|\tilde v^{(k,l)}\|_{W_\sigma^{2+s,1+s/2}(\Om^t)}+\|\nabla\tilde p^{(k,l)}\|_{W_\sigma^{s,s/2}(\Om^t)}\cr
&\le c[\|(\ro(\xi^{(k)},\xi_0^{(l)})-\ro(x,t))\tilde v_t^{(k,l)}\|_{W_\sigma^{s,s/2}(\Om^t)}\cr
&\quad+|\ro v\zeta_{t}^{(k,l)}\|_{W_\sigma^{s,s/2}(\Om^t)}+ \|p\nabla\zeta^{(k,l)}\|_{W_\sigma^{s,s/2}(\Om^t)}\cr
&\quad+\|\nabla v\nabla\zeta^{(k,l)}\|_{W_\sigma^{s,s/2}(\Om^t)}+ \|v\Delta\zeta^{(k,l)}\|_{W_\sigma^{s,s/2}(\Om^t)}\cr
&\quad+\|\ro v\cdot\nabla v\zeta^{(k,l)}\|_{W_\sigma^{s,s/2}(\Om^t)}+ \|\ro\tilde f^{(k,l)}\|_{W_\sigma^{s,s/2}(\Om^t)}\cr
&\quad+\|v\nabla\zeta^{(k,l)}\|_{W_\sigma^{1+s-1/\sigma,1/2+s/2-1/2\sigma}(S_2^t)}\cr
&\quad+\|\tilde v^{(k,l)}\|_{W_\sigma^{1+s-1/\sigma,1/2+s/2-1/2\sigma}(S_2^t)}\cr
&\quad+\|\tilde d^{(k,l)}\|_{W_\sigma^{2+s-1/\sigma,1+s/2-1/2\sigma}(S_2^t)}+ \|\tilde v^{(k,l)}(0)\|_{W_\sigma^{2+s-2/\sigma}(\Om)}].\cr}
\label{A.8}
\end{equation}
Let $u,v\in L_\infty(\Om^t)$. Then
\begin{equation}
\|uv\|_{W_\sigma^{s,s/2}(\Om^t)}\le\|u\|_{L_\infty(\Om^t)} \|v\|_{W_\sigma^{s,s/2}(\Om^t)}+\|v\|_{L_\infty(\Om^t)} \|u\|_{W_\sigma^{s,s/2}(\Om^t)}.
\label{A.9}
\end{equation}
Let $\Om_{k,l}^t=\Om^t\cap\supp\zeta^{(k,l)}$.

Using (\ref{A.9}), the second term on the r.h.s. of (\ref{A.8}) is bounded by
$$
{c\over\lambda}(\|\ro\|_{L_\infty(\Om^t)}\|v\|_{W_\sigma^{s,s/2}(\Om_{k,l}^t)}+ \|\ro\|_{W_\sigma^{s,s/2}(\Om_{k,l}^t)}\|v\|_{L_\infty(\Om^t)}+ \|\ro\|_{L_\infty(\Om^t)}\|v\|_{L_\infty(\Om_{k,l}^t)})
$$
and we use the imbedding
$$\eqal{
&\|\ro\|_{W_\sigma^{s,s/2}(\Om^t)}\le c\|\ro\|_{W_{r,\infty}^{1,1}(\Om^t)}\quad &{\rm for}\ \ {3\over r}-{5\over\sigma}+s<1,\cr
& \|\ro\|_{L_\infty(\Om^t)}\le c\|\ro\|_{W_{r,\infty}^{1,1}(\Om^t)}\quad &{\rm for}\ \ {3\over r}<1,\cr
&\|v\|_{L_\infty(\Om^t)}\le c\|v\|_{W_\sigma^{2+s,1+s/2}(\Om^t)}\quad &{\rm for}\ \ {5\over\sigma}<2+s.\cr}
$$
Similarly, the fifth term is bounded by
$$
{c\over\lambda}\|v\|_{W_\sigma^{2+s,1+s/2}(\Om_{k,l}^t)}.
$$
Assume that
$$
\|\nb v\|_{L_\infty(\Om^t)}\le c\|v\|_{W_\sigma^{2+s,1+s/2}(\Om^t)}
$$
which holds for $5/\sigma<1+s$.

Then the fourth term is bounded by
$$
{c\over\lambda}\|\nabla v\|_{W_\sigma^{2+s,1+s/2}(\Om^t)}.
$$
The seventh term on the r.h.s. of (\ref{A.8}) equals
$$
I\equiv\|\ro\tilde f^{(k,l)}\|_{L_\sigma(0,t;W_\sigma^s(\Om))}+ \|\ro\tilde f^{(k,l)}\|_{L_\sigma(\Om;W_\sigma^{s/2}(0,t))}\equiv I_1+I_2,
$$
where $I_1$ is bounded by
$$\eqal{
&\bigg({\displaystyle{\intop_0^tdt\intop_\Om\intop_\Om} dx'dx''{|\ro(x',t)-\ro(x'',t)}|^\sigma |\tilde f^{(k,l)}(x',t)|^\sigma\over|x'-x''|^{3+s\sigma/2}}\bigg)^{1/\sigma}\cr
&\quad+\bigg(\intop_0^tdt\intop_\Om\intop_\Om dx'dx''{|\ro(x'',t)|^\sigma|\tilde f^{(k,l)}(x',t)-\tilde f^{(k,l)}(x'',t)|^\sigma\over|x'-x''|^{3+s\sigma/2}}\bigg)^{1/\sigma}\cr
&\equiv I_1^1+I_1^2.\cr}
$$
Hence
$$
I_1^2\le c \|\ro\|_{L_\infty(\Om^t)}\|\tilde f^{(k,l)}\|_{L_\sigma(0,t;W_\sigma^s(\Om))}.
$$
To estimate $I_1^1$ we use the H\"older inequality. Then we obtain
$$\eqal{
I_1^1&\le\bigg[\intop_0^tdt\bigg(\intop_\Om\intop_\Om dx'dx'' {|\ro(x',t)-\ro(x'',t)|^{\sigma\lambda_1}\over|x'-x''|^{3+\sigma\lambda_1s'}} \bigg)\bigg]^{1/\sigma\lambda_1}\cdot\cr
&\quad\cdot\bigg[\intop_0^tdt\bigg(\intop_\Om\intop_\Om{|\tilde f^{(k,l)}(x',t)|^{\sigma\lambda_2}\over|x'-x''|^{3/2\lambda_2}}dx'dx''\bigg) \bigg]^{1/\sigma\lambda_2}\equiv I_1^{11},\cr}
$$
where $s'={1\over\sigma\lambda_1}\big({3\over 2}\lambda_1-3)+s$ and $1/\lambda_1+1/\lambda_2=1$. Using that $\lambda_2<2$ we obtain
$$
I_1^{11}\le\|\ro\|_{L_{\sigma\lambda_1}(0,t;W_{\sigma\lambda_1}^{s'}(\Om))} \|\tilde f^{(k,l)}\|_{L_{\sigma\lambda_2}(0,t;L_{\sigma\lambda_2}(\Om))}\equiv I_1^{12}.
$$
We use the imbeddings
$$
\|\ro\|_{L_{\sigma\lambda_1}(0,t;W_{\sigma\lambda_1}^{s'}(\Om))}\le c\|\ro\|_{W_{r,\infty}^{1,1}(\Om^t)}
$$
for ${3\over r}-{5\over\sigma\lambda_1}+s'\le 1$ and
$$
\|\tilde f^{(k,l)}|\|_{L_{\sigma\lambda_2}(\Om^t)}\le c\|\tilde f^{(k,l)}\|_{W_\sigma^{s,s/2}(\Om^t)}
$$
for ${5\over\sigma}-{5\over\sigma\lambda_2}\le s$.

The above restrictions imply
$$
{3\over r}+{3\over2\sigma}-{3\over\sigma\lambda_1}+s\le 1+s\quad {\rm so}\quad {3\over r}+{3\over 2\sigma}\le 1+{3\over\sigma\lambda_1}\le 1+{3\over 2\sigma}
$$
because $1/\lambda_1<1/2$. Hence we get ${3\over r}\le 1$.

Summarizing
$$
I_1\le c\|\ro\|_{W_{r,\infty}^{1,1}(\Om^t)}\|\tilde f^{(k,l)}\|_{W_\sigma^{s,s/2}(\Om^t)}\quad {\rm for}\ \ {3\over r}\le 1.
$$
Consider $I_2$,
$$\eqal{
I_2&=\bigg(\intop_\Om dx\intop_0^t\intop_0^tdt'dt''{|\ro(x,t')\tilde f^{(k,l)}(x,t')-\ro(x,t'')\tilde f^{(k,l)}(x,t'')|^\sigma\over|t'-t''|^{1+{\sigma}s/2}}\bigg)^{1/\sigma}\cr
&\le\bigg(\intop_\Om dx\intop_0^t\intop_0^t dt'dt''{|\ro(x,t')-\ro(x,t'')|^\sigma|\tilde f^{(k,l)}(x,t')|^\sigma\over|t'-t''|^{1+\sigma s/2}}\bigg)^{1/\sigma}\cr
&\quad+\bigg(\intop_\Om dx\intop_0^t\intop_0^t dt'dt''{|\ro(x,t'')|^\sigma|\tilde f^{(k,l)}(x,t')-\tilde f^{(k.l)}(x,t'')|^\sigma\over|t'-t''|^{1+\sigma s/2}} \bigg)^{1/\sigma}\cr
&\equiv I_2^1+I_2^2,\cr}
$$
where
$$
I_2^2\le \|\ro\|_{L_\infty(\Om^t)}\|\tilde f^{(k,l)}\|_{L_\sigma(\Om;W_\sigma^{s/2}(0,t))}
$$
and
$$\eqal{
I_2^1&\le\bigg(\intop_\Om dx\intop_0^t\intop_0^t dt'dt''{|t'-t''|^\sigma\sup_t|\ro_t(x,t)|^\sigma|\tilde f^{(k,l)}(x,t')|^\sigma\over|t'-t''|^{1+\sigma s/2}}\bigg)^{1/\sigma}\cr
&\le c\sup_t|\ro_t|_{\lambda_1\sigma,\Om}|\tilde f^{(k,l)}|_{\lambda_2\,\sigma,\sigma,\Om^t}\equiv I_2^{11},\cr}
$$
where $1/\lambda_1+1/\lambda_2=1$, $\lambda_1\sigma=r$, ${5\over\sigma}-{3\over\lambda_2\sigma}-{2\over\sigma}\le s$.

Hence ${3\over r}\le s$, $r>\sigma$. Then
$$
I_2^{11}\le c\|\ro\|_{W_{r,\infty}^{1,1}(\Om^t)}\|\tilde f^{(k,l)}\|_{W_\sigma^{s,s/2}(\Om^t)}.
$$
Summarizing, we have
$$
I\le c\|\ro\|_{W_{r,\infty}^{1,1}(\Om^t)}\|\tilde f^{(k,l)}\|_{W_\sigma^{s,s/2}(\Om^t)}
$$
for $r>\sigma$, ${3\over r}\le s$.

Similarly, the third term on the r.h.s. of (\ref{A.8}) is estimated by
$$
{c\over\lambda}\|p\|_{W_\sigma^{s,s/2}(\Om_{k,l}^t)}
$$
and the eight term by
$$
{c\over\lambda}\|v\|_{W_\sigma^{2+s,1+s/2}(\Om_{k,l}^t)}
$$
\goodbreak

\noindent
Finally, we estimate the first term on the r.h.s. of (\ref{A.8}). It is bounded by
$$\eqal{
&\sup_{\Om_{k,l}^t}|\ro(x,t)-\ro(\xi^{(k)},\xi_0^{(l)})|\,\|\tilde v_t^{(k,l)}\|_{W_\sigma^{s,s/2}(\Om^t)}\cr
&\quad+\bigg(\intop dt\intop_\Om\intop_\Om{|\ro(x',t)-\ro(x'',t)|^\sigma|\tilde v_t^{(k,l)}(x',t)|^\sigma\over|x'-x''|^{3+s\sigma}}dx'dx''\bigg)^{1/\sigma}\cr
&\quad+\bigg(\intop dx\intop_0^t\intop_0^t{|\ro(x,t')-\ro(x,t'')|^\sigma|\tilde v_t^{(k,l)}(x,t')|^\sigma\over|t'-t''|^{1+s\sigma/2}}dt'dt''\bigg)^{1/\sigma}\cr
&\equiv J_1+J_2+J_3,\cr}
$$
where
$$\eqal{
J_1&\le\|\ro\|_{\dot C^\alpha(\Om_{k,l}^t)}\lambda^\alpha\|\tilde v_t^{(k,l)}\|_{W_\sigma^{s,s/2}(\Om^t)},\cr
J_2+J_3&\le c\|\ro\|_{W_{r,\infty}^{1,1}(\Om^t)}\|\tilde v_t^{(k,l)}\|_{W_\sigma^{s/2,s/4}(\Om^t)},\cr}
$$
where the last inequality is obtained in the same way as it was done in the estimate of the seventh term.

Summing up over all neighborhoods of the partition of unity we get
\begin{equation}\eqal{
&\|v\|_{W_\sigma^{2+s,1+s/2}(\Om^t)}+\|\nabla p\|_{W_\sigma^{s,s/2}(\Om^t)}\cr
&\le\phi(\|\ro\|_{\dot C^\alpha(\Om^t)})[\|v\|_{W_\sigma^{1+s,1/2+s/2}(\Om^t)}+ \|p\|_{W_\sigma^{s,s/2}(\Om^t)}]\cr
&\quad+c(\ro^*)[\|\ro v\cdot\nabla v\|_{W_\sigma^{s,s/2}(\Om^t)}+ \|\ro\|_{C^\alpha(\Om^t)}\|f\|_{W_\sigma^{s,s/2}(\Om^t)}\cr
&\quad+\|v\|_{W_\sigma^{1+s,1/2+s/2}(\Om^t)}+ \|d\|_{W_\sigma^{2+s-1/\sigma,1+s/2-1/2\sigma}(S_2^t)}\cr
&\quad+\|v(0)\|_{W_\sigma^{2+s-2/\sigma}(\Om)}],\cr}
\label{A.10}
\end{equation}
where $\phi(0)=0$.

By some interpolations and Corollary \ref{c5.3} we obtain (\ref{A.7}). This ends the proof.
\end{proof}

\section*{The conflict of interest and the data availability statements}

On behalf of all authors, the corresponding author states that there is no conflict of interest.

\noindent The datasets generated during and/or analysed during the current study are available from the corresponding author on reasonable request.
\goodbreak

\end{document}